\documentclass[12pt, a4paper, parskip=half, abstracton]{scrartcl}

\usepackage{array}
\usepackage{xcolor}
\usepackage{amscd,amssymb,amsfonts,amsmath,latexsym,amsthm}
\usepackage{hyperref}
\usepackage[all,cmtip]{xy}
\usepackage{enumitem}
\usepackage{mathrsfs}
\usepackage{graphicx}
\usepackage{bm} 
\usepackage[nameinlink, capitalise]{cleveref}

\usepackage{etoolbox}

\def\hB{\hspace*{\fill}$\qed$}

\usepackage[nottoc]{tocbibind} 

\usepackage{defs_pp1}
\usepackage{slashed}
\usepackage[utf8]{inputenc}
\usepackage{microtype}
\usepackage[english]{babel}
\usepackage{mathtools}
\usepackage{bm}
\usepackage{esvect}

\title{ $E$-theory  is compactly assembled}
\author{
Ulrich Bunke\thanks{Fakult{\"a}t f{\"u}r Mathematik,
Universit{\"a}t Regensburg,
93040 Regensburg,
ulrich.bunke@mathematik.uni-regensburg.de} 
 \ and Benjamin Dünzinger\thanks{Fakult{\"a}t f{\"u}r Mathematik,
	Universit{\"a}t Regensburg,
	93040 Regensburg, 
	benjamin.duenzinger@ur.de} 
}

\numberwithin{equation}{section}
\newtheorem{theorem}{Theorem}[section] 
\newtheorem{prop}[theorem]{Proposition}
\newtheorem{lem}[theorem]{Lemma}
\newtheorem{constr}[theorem]{Construction}

\newtheorem{ddd}[theorem]{Definition}
\newtheorem{kor}[theorem]{Corollary}

\newtheorem{prob}[theorem]{Problem}

\theoremstyle{remark}
\theoremstyle{definition}

\newtheorem{ex}[theorem]{Example}
\newtheorem{rem}[theorem]{Remark}

\newcommand{\bGs}{\cB^{\Es} }
\newcommand{\nGs}{\cN^{\Es} }

\newcommand{\cKue}{\cK}

\newcommand{\bGsk}{\cB^{\KKs} }
\newcommand{\nGsk}{\cN^{\KKs} }

\newcommand{\cTmakk}{\cK^{\KKs}_{\max} }
\newcommand{\cTmikk}{\cK^{\KKs}_{\min} }
\newcommand{\cTe}{\cK^{\Es} }

\newcommand{\trkma}{\cT^{\KKs}_{\max} }
\newcommand{\trkmi}{\cT^{\KKs}_{\min} }
\newcommand{\tre}{\cT^{\Es} }

\newcommand{\unit}{1}

\newcommand{\kks}{\mathrm{kk}_{\sepa}}
\newcommand{\KKs}{\mathrm{KK}_{\sepa}}

\newcommand{\cfil}{\mathrm{cfil}}
\newcommand{\lex}{\mathrm{lex}}

\newcommand{\AGc}{\mathrm{A}^G}
\newcommand{\AsG}{\mathrm{A}^{G}_{\mathrm{sep},0}}
\newcommand{\asG}{a^G_{\mathrm{sep},0}}
\newcommand{\AsGc}{\mathrm{A}^{G}_{\mathrm{sep}}}
\newcommand{\asGc}{a^G_{\mathrm{sep}}}
\newcommand{\asGck}{a^G_{\mathrm{sep},K_{G}}}

\newcommand{\EsGf}{\mathrm{E}^G_{\sepa, \fin}}
\newcommand{\esGf}{\mathrm{e}^G_{\sepa, \fin}}
\newcommand{\esGfh}{\mathrm{e}^G_{\sepa, h,\fin}}

\newcommand{\aGc}{\mathrm{a}^{G}}

\newcommand{\EssIm}{\mathrm{EssIm}}
\newcommand{\lax}{\mathrm{lax}}
\newcommand{\pcas}{compactly assembled}
\newcommand{\casmbl}{\mathrm{casmbl}}

\newcommand{\disc}{\mathrm{disc}}

\newcommand{\fin}{\mathrm{fin}}

\newcommand{\EE}{\mathrm{E}}
\newcommand{\sepa}{\mathrm{sep}}

\newcommand{\A}{\mathbb{A}}

\newcommand{\kkGs}{\mathrm{kk}_{\sepa}^{G}}
\newcommand{\KKG}{\mathrm{KK}^{G}}
\newcommand{\KKGs}{\mathrm{KK}_{\sepa}^{G}}

\newcommand{\nCalg}{C^{*}\mathbf{Alg}^{\mathrm{nu}}}

\newcommand{\F}{\mathbb{F}}

\newcommand{\ho}{\mathrm{ho}}

\newcommand{\Res}{\mathrm{Res}}

\newcommand{\Pro}{\mathbf{Pro}}

\newcommand{\cN}{\mathcal{N}}

\newcommand{\Fin}{\mathbf{Fin}}

\newcommand{\bB}{{\mathbf{B}}}
\newcommand{\Fib}{{\mathrm{Fib}}}

\newcommand{\incl}{\mathrm{incl}}

\newcommand{\cK}{\mathcal{K}}

\newcommand{\bA}{{\mathbf{A}}}

\newcommand{\const}{{\mathtt{const}}}

\newcommand{\cO}{{\mathcal{O}}}
\newcommand{\cU}{{\mathcal{U}}}

\newcommand{\cD}{{\mathcal{D}}}

 \newcommand{\Cat}{{\mathbf{Cat}}}

\newcommand{\Group}{\mathbf{Group}}

\newcommand{\uHom}{\underline{\mathrm{Hom}}}
\newcommand{\sHom}{\underline{\mathrm{sHom}}}

\newcommand{\Spc}{\mathbf{An}}

\renewcommand{\Pr}{\mathbf{Pr}}

\newcommand{\st}{\mathrm{st}}

\newcommand{\op}{\mathrm{op}}

\newcommand{\stCatp}{\mathbf{Cat}^{\mathrm{perf}}_{\infty}}

\newcommand{\KK}{\mathrm{KK}}

\renewcommand{\loc}{\mathrm{loc}}

\newcommand{\exa}{\mathrm{ex}}

\newcommand{\K}{\mathrm{K}}

\newcommand{\PrLst}{\mathbf{\Pr}^{\mathrm{L}}_{\mathrm{st}}}

\newcommand{\EsG}{\mathrm{E}^{G}_{\sepa}}
\newcommand{\EsH}{\mathrm{E}^{H}_{\sepa}}
\newcommand{\EsGght}{\mathrm{E}^{G}_{\sepa,GHT}}
\newcommand{\Es}{\mathrm{E}_{\sepa}}

\newcommand{\EsGs}{ \mathrm{E}_{\sepa,\oplus}^{G}}
\newcommand{\esGs}{ \mathrm{e}_{\sepa,\oplus}^{G}}

\newcommand{\EG}{\mathrm{E}^{G}}
\newcommand{\eG}{\mathrm{e}^{G}}

\newcommand{\eGG}{\mathrm{e}^{G}}

\newcommand{\EsGn}{\mathrm{E}^{G}_{\sepa,0}}

\newcommand{\esG}{\mathrm{e}^{G}_{\sepa}}
\newcommand{\esH}{\mathrm{e}^{H}_{\sepa}}
\newcommand{\esGght}{\mathrm{e}^{G}_{\sepa,GHT}}
\newcommand{\es}{\mathrm{e}_{\sepa}}

\newcommand{\hesG}{\hat{\mathrm{e}}^{G}_{\sepa}}
\newcommand{\esGn}{\mathrm{e}^{G}_{\sepa,0}}
\newcommand{\esGnh}{\mathrm{e}^{G}_{\sepa,h,0}}

\newcommand{\Profinl}{\mathbf{ProFin}^{\aleph_{1}}}

\newcommand{\shp}{\mathrm{shp}}
\newcommand{\dual}{\mathrm{dual}}

\newcommand{\sh}{\mathrm{sh}}
\newcommand{\cd}{\mathrm{cd}}

\begin{document}

\maketitle

\begin{abstract}   
We show that the equivariant $E$-theory category $\EsG$ for separable $C^{*}$-algebras is a compactly assembled stable $\infty$-category. We derive this result as a consequence of the shape theory for $C^{*}$-algebras developed by Blackadar and Dardarlat and {a  new $ \infty $-categorical refinement of the category of asymptotic morphisms $ \AsGc $.}   As an application we investigate  {a} topological enrichment of the homotopy category of a compactly assembled $\infty$-category in general  and argue that  the results of Carrión and Schafhauser on 
the enrichment of the classical $E$-theory category can be  derived by  specialization. 	\end{abstract}

\tableofcontents
\setcounter{tocdepth}{5}
	
\section{Introduction}

\subsection{Statement of the main result}

The   $\infty$-category $\Pr^{L}_{\st}$ of stable presentable $\infty$-categories and colimit-preserving functors  is equipped   with the symmetric monoidal structure given by the Lurie tensor product   and therefore admits a notion of dualizable objects. According to
 \cite[D.7.0.7]{sag}  the  dualizable objects in $\Pr^{L}_{\st}$ can equivalently be characterized as  the compactly assembled
presentable stable $\infty$-categories, {which will be introduced later on in this introduction.}

{Compactly}  generated   presentable  stable $\infty$-categories are 
examples of {dualizable} 
  stable $\infty$-categories. They comprise the image  
of the ind-completion functor  $\Ind:\stCatp\to \Pr^{L}_{\st}$. This functor identifies  the category of 
small idempotent complete stable $\infty$-categories with the subcategory 
$\Pr^{LL}_{\omega,\st}\subseteq \Pr^{L}_{\st}$ of compactly generated stable presentable $\infty$-categories  and
  strongly cocontinuous  functors (colimit preserving functors  whose  right-adjoints also preserve colimits). 
  
  {Dualizable  stable $\infty$-categories 
  recently attracted great attention since  \cite{Efimov_2024} has shown that
  every localizing invariant $L:\stCatp\to \mathbf{C}$ (a reduced functor {preserving} {bi-fibre}
 sequences with a stable codomain)  
  has a unique extension $\hat L$ as in 
  {$$\xymatrix{\stCatp\ar[r]_{\simeq}^{\Ind}\ar[rd]_{L}& \Pr^{LL}_{\omega,\st}\ar[r]^{\subseteq} & \Pr^{{\dual}}\ar@{..>}[dl]^{\hat L}   \\ & \mathbf{C} } \ ,$$}where $\Pr^{{\dual}}\subseteq \Pr^{L}_{\st}$ is the subcategory of dualizable 
  presentable stable  $\infty$-categories and  strongly cocontinuous  functors.}
 
  Known classes of examples of   dualizable   presentable stable $\infty$-categories
  beyond the compactly generated ones are given by 
  sheaves on locally compact topological spaces with values in  spectra \cite[Proposition 21.1.7.1]{sag},  certain
  {``nuclear"} module  categories appearing in 
 analytic geometry  \cite{Efimov_Inverse_Limits} or the theory of analytic stacks developed by Clausen-Scholze,  but also the 
 target  of the universal {finitary} localizing invariant
 \begin{equation}\label{gwerfwerfwef}\cU_{\loc}:\stCatp\to\mathrm{Mot}_{\mathrm{loc}}
\end{equation} \cite{Efimov_Motives}. We  refer to \cite{som}
 for a good introduction to many of these aspects. 
 
 The main goal of the present paper is to contribute a new class of examples
of dualizable  presentable stable  $\infty$-categories coming from non-commutative homotopy theory of $C^{*}$-algebras. Let $G$ be a countable {discrete} group and $G\nCalg$ be the category of $G$-$C^{*}$-algebras. By \cref{oopwerfewrsfg}.\ref{oigopregpwerferwf9} 
we then have the universal
homotopy invariant, $K_{G}$-stable, exact, s-finitary and countable sum-preserving functor
\begin{equation}\label{gweroijfweiofwerfwerfwerfwr}\eG:G\nCalg\to \EG
\end{equation} to a cocomplete stable $\infty$-category. This functor is an $\infty$-categorical enhancement of the classical equivariant $E$-theory functor \cite{Guentner_2000}.  
The following is the main result of the present paper.  
\begin{theorem}\label{kohperthtgertgert} For every countable group $G$
the stable $\infty$-category $\EG$ is presentable and dualizable.
 \end{theorem}
 The  equivariant $E$-theory functor  \eqref{gweroijfweiofwerfwerfwerfwr} could be viewed as the non-commutative homotopy theoretic analogue of the universal localizing invariant \eqref{gwerfwerfwef}, and then \cref{kohperthtgertgert} becomes the analogue of the dualizability of $\mathrm{Mot}_{\mathrm{loc}}$ from \cite{Efimov_Motives}.

 As shown in \cite{Efimov_Motives} every  {dualizable  presentable} stable $\infty$-category $\bC$
can be written in the form
$\bC\simeq \Ind_{\aleph_{1}}(\bC^{\aleph_{1}})$
where   $\bC^{\aleph_{1}} \subseteq \bC$ is  the full 
  {subcategory}  of $\aleph_{1}$-compact objects and $ \Ind_{\aleph_{1}}$ is the functor which freely adds all $\aleph_{1}$-filtered colimits. 
In fact, the $\aleph_{1}$-$\Ind$-completion functor 
$$\Ind_{\aleph_{1}}:\Cat_{\infty}^{\exa,\casmbl} \to \Pr^{{\dual}}$$
induces an equivalence from the category of  countably cocomplete {small} compactly assembled 
  stable $\infty$-categories and countable colimit-  and compact maps preserving functors to  $\Pr^{{\dual}}$.

In the present paper we will mainly work on the level of small  
catergories. Consequently, we decided to call small compactly assembled categories just compactly assembled categories.  In the case of 
equivariant $E$-theory  by  \cref{iuhiowerfwefewf}  we have  an equivalence $(\EG)^{\aleph_{1}}\simeq \EsG $, where $\EsG$ is the target of the universal
homotopy invariant, $K_{G}$-stable, exact and countable sum-preserving functor
$$\esG :G\nCalg_{\sepa}\to \EsG $$
from separable $G$-$C^{*}$-algebras to a countably cocomplete stable $\infty$-category.
In order to show \cref{kohperthtgertgert} we must therefore  show that $\EsG $ is compactly assembled (in the sense of countably cocomplete stable $\infty$-categories). To this end
we generalize and apply the  shape theory for  $C^{*}$-algebras \cite{zbMATH03996430}, \cite{Dadarlat_1994}. 

In addition to the verification of  \cref{kohperthtgertgert},  this paper 
develops   some foundational material on compact maps and shape systems in countably cocomplete $\infty$-categories
in \cref{kopwretgretreht}, and it also discusses {applications to the homotopy theory of $C^{*}$-algebras in \cref{tgkowpegfwerwggerfwf}.} 

%
%
%
%
 \subsection{A detailed introduction}
 
 This subsection gives a more detailed overview {on} the results of this paper including the main definitions and precise statements.
We let $G$ be a  countable {discrete} group and denote by $G\nCalg_{\sepa}$   the category of separable  $G$-$C^{*}$-algebras.
The  homotopy theory in this category  will be captured {by} 
  functors \begin{equation}\label{fqwefewdqwedwefqf}\asGc:G\nCalg_{\sepa}\to \AsGc\ ,\qquad \esG:G\nCalg_{\sepa}\to \EsG\ .
\end{equation} 
Here $\asGc$ is {defined as}
the universal homotopy invariant and countable filtered colimit preserving functor to an $\infty$-category admitting countable filtered colimits, and $\esG$ is {defined as} the universal homotopy invariant, $K_{G}$-stable, exact and countable sum-preserving functor to a countably cocomplete stable $\infty$-category (see \cref{fjioqwefdqwedewdqwed} for a description of these properties). 
 These functors {turn out to be}  $\infty$-categorical enhancements {of} {classically known}  functors
  \begin{equation}\label{fqwefewdqwedeerererwefqf}\asG:G\nCalg_{\sepa}\to \AsG\ ,\qquad \esGn:G\nCalg_{\sepa}\to \EsGn\ ,\end{equation}  where
 $\AsG$ is the category of separable $C^{*}$-algebras and homotopy classes  of asymptotic morphisms introduced in \cite{Guentner_2000} (see also \cite{Dadarlat_1994} for the case of a trivial group), while $\EsGn$ is the classical equivariant  $E$-theory category constructed in the same reference (see also \cite{MR1068250}, \cite{zbMATH04182148}  for the case of a trivial group  $G$). 
 The proof of the  existence  of the functors {in \eqref{fqwefewdqwedwefqf}} by providing constructions is a central topic of this paper. 
 
We next introduce the notion of a compactly assembled $\infty$-category.
 If $\bD$ is an $\infty$-category admitting countable filtered colimits, then
 the canonical morphism $y:\bD\to \Ind^{\aleph_{1}}(\bD)$ has a left-adjoint $|-|: \Ind^{\aleph_{1}}(\bD)\to \bD$
 called the realization functor.  Here $\Ind^{\aleph_{1}}(\bD)$ is the closure of the essential image of $y:\bD\to \Ind(\bD)$ under countable filtered colimits. Adapting the characterization \cite[Sec. 21.1.2.10]{sag} of being compactly assembled for large $\infty$-categories to the case of small $\infty$-categories just admitting  countable filtered colimits we make the following definition. 
 
\begin{ddd}\label{eruwigowergwergferfrw} The $\infty$-category $\bD$ is called  \pcas{} if the realization functor admits a    left-adjoint.
\end{ddd}
{In order to highlight the analogy with the shape systems for $C^{*}$-algebras as described below} this left-adjoint $S:\bD\to  \Ind^{\aleph_{1}}(\bD)$ will be called the shape functor.

 \begin{rem}\label{ojkbpertherthregertg}In our context of an $\infty$-category $\bD$ just admitting countable filtered colimits  an object $A$ of $\bD$ will be called compact if the functor $\Map_{\bD}(A,-):\bD\to \Spc$ preserves countable filtered colimits. We say that $\bD$ is compactly generated if every object of $\bD$ is equivalent to a countable filtered colimit of a system of compact objects.
If $\bD$ is   compactly generated, then it is \pcas{} in the sense of \cref{eruwigowergwergferfrw}.  But  a general     $\infty$-category with countable filtered colimits might have very few compact objects and being  \pcas{} is a good replacement of the condition of being compactly generated.\hB \end{rem}
  
  \begin{ex}
  
{If $H$ is a finite group{,} then by \cref{gjiseogpgregesg} we know that $\esH(\C)$ is a compact object of $\EsH$.
 The induction functor 
 $\Ind_{H}^{G}:\EsH\to \EsG$  has a countable colimit preserving right-adjoint $\Res^{G}_{H}$ (see \cref{kogregwergwerg}) and therefore preserves compact objects. In particular the objects $\Ind_{H}^{G}(\esH(\C))\simeq \esG(C_{0}(G/H))$
 are compact {in} $\EsG$ for all finite subgroups $H$ of $G$. Further note that the collection of compact objects is closed under finite colimits and retracts. But in general $\EsG$ is not generated by these compact objects. {If it were, then as a consequence of \cite[Thm. 13.1]{Guentner_2000}  the Baum-Connes assembly map for the maximal crossed product would be an equivalence, but this is known to be false e.g. for groups with property $T$ having a $\gamma$-element \cite[Cor. 5.6]{Aparicio_2019}.} We also do not expect that {$\EsG$} is compactly generated at all.} \hB\end{ex}

The central result of the present paper is the following (see \cref{vsiowerfsdfuiohsdfoijasdf} and \cref{vsiowerfsdfuiohsdfoijasdf1}).
\begin{theorem}\label{jifofqweewf9}
The $\infty$-categories $\AsGc$ and $\EsG$ are  \pcas{}.
\end{theorem}

If $\bD$ is an $\infty$-category admitting countable filtered colimits, then according to \cref{qrjiofoqdedeqw} we have various  notions of compactness of morphisms in $\bD$. 
The following definition is a translation of   \cref{qrjiofoqdedeqw}.\ref{trgkpwegfwerfweffwef}.   
\begin{ddd}\label{kogpwegerferfw1} A map $f:A\to B$ in $\bD$   is weakly compact if  for every system $(C_{n})_{n\in \nat}$ with colimit $C:=\colim_{n\in \nat}C_{n}$ and map $B\to C$ in $\bD$ there exists a lift $$\xymatrix{A\ar[d]^{f}\ar@{..>}[r]&C_{n}\ar[d]\\ B\ar[r]&  { C}}$$ {for some $n$ in $\nat$.}
\end{ddd}

\begin{ddd}\label{kgopwertgerferfwref}
A system {$(A_{n})_{n\in \nat}$} in $\bD$ is called {an} slp-compact exhaustion of $A$ if $A\simeq \colim_{n\in \nat} A_{n}$
and the structure maps $A_{n}\to A_{n+1}$ are slp-compact for all $n$ in $\nat$.\end{ddd}
 The addition of  $slp$ indicates a strengthening of weak compactness  \cref{qrjiofoqdedeqw}.\ref{trgkprererwegfwerfweffwef}.
We will use the following characterization of being \pcas{} in terms of slp-exhaustions
which follows from a combination of \cref{gkopwerfweferwfw}
and \cref{regijfowergffreferwfw}. {This characterization provides a version of the internal characterization of compactly assembled categories due to Clausen in the least structured setting possible.}
    \begin{prop}\label{tkogpwegfrfweg9} If $\bD$ is  an $\infty$-category admitting countable filtered colimits, then
 it is \pcas{} if and only if every object  admits an slp-compact exhaustion.
 \end{prop}
 {We now turn attention to $C^{*}$-algebras.}
  The basic notion of the shape theory in the category $G\nCalg_{\sepa}$   is that of a semi-projective homomorphism $f:A\to B$ {which was} introduced in the non-equivariant situation by \cite{zbMATH03996430}.  
\begin{ddd}\label{kogpwegerferfw}  {A} homomorphism $f:A\to B$  in $G\nCalg_{\sepa}$ is semi-projective if for every   separable $G$-$C^{*}$-algebra $Z$ and increasing sequence $(I_{n})_{n\in \nat}$   of {invariant} ideals in $Z$   with  direct limit ideal $I:=\overline{\bigcup_{n\in \nat} I_{n}}$ and map $B\to Z/I$  there exists a factorization
$$\xymatrix{A\ar[d]^{f}\ar@{..>}[r]&Z/I_{n}\ar[d]\\ B\ar[r]&  { Z/I} }$$
{for some $n$ in $\nat$.}
\end{ddd}

The similarity of \cref{kogpwegerferfw} and  \cref{kogpwegerferfw1} is not only formal.
The following  theorem (see \cref{gjoiowepgrfwefrefw})  directly  relates the semi-projectivity of maps in $G\nCalg_{\sepa}$
with the strong compactness (this is our strongest notion of compactness \cref{qrjiofoqdedeqw}.\ref{trgkprerergwerferfwfwegfwerfweffwef}) of the induced morphism in $\AsGc$ or $\EsG$.
We call $B$ in $G\nCalg_{\sepa}$ twice suspended and $K_{G}$-stable if
$B\cong B'\otimes C_{0}(\R^{2})\otimes K_{G}$ for some
$B'$ in $G\nCalg_{\sepa}$, where $K_{G}:=K(L^{2}(G)\otimes \ell^{2}) $. 

\begin{theorem}\label{rjifoqopffefwewerfr}\mbox{}
Let $A\to B$ be a semi-projective morphism in $G\nCalg_{\sepa}$.
\begin{enumerate}
\item
 The induced morphism   $\asGc(f): \asGc(A)\to \asGc(B)$ is strongly  compact in  $\AsGc$.   \item\label{jiwergowegwerfwerf} If $B$ is twice suspended and $K_{G}$-stable, then $\esG(f): \esG(A)\to \esG(B)$   is strongly compact in  $\EsG$.  \end{enumerate}
\end{theorem}

The proof of \cref{jifofqweewf9} is based on the
 shape theory for $C^{*}$-algebras  developed by  \cite{zbMATH03996430} 
and the insights of  \cite{zbMATH02240070}.  
Let $A$ be  a separable $G$-$C^{*}$-algebra.
\begin{ddd}A  system $(A_{n})_{n\in \nat}$ in $G\nCalg_{\sepa}$ is called a shape system for   $A$  
if $A\cong \colim_{n\in \nat} A_{n}$ and the structure maps $A_{n}\to A_{n+1}$ are semi-projective for all $n$ in $\nat$.
\end{ddd}
{Note the analogy with \cref{kgopwertgerferfwref}.}
In \cref{erijgioepgwegerfwef} we extend the theory of \cite{zbMATH03996430} to the equivariant case.  In particular we will show in \cref{ojkergperwerge}:\begin{prop} \label{rejgioiowergerfwf9} Every object of $G\nCalg_{\sepa}$
admits an explicit shape system.
\end{prop}
The additional adjective explicit is technically relevant but not  important for the purpose of this introduction.
  In order to understand why \cref{jifofqweewf9} is true  in the case of $\AsGc$ one could observe   that $\asGc$ by \cref{rjifoqopffefwewerfr}   maps   shape systems in $G\nCalg_{\sepa}$ to slp-compact exhaustions in $\AsGc$   and then apply the combination of \cref{rejgioiowergerfwf9} and \cref{tkogpwegfrfweg9}.  In order to deduce that $\EsG$ is \pcas{} we use that this category is obtained from $\AsGc$ by a sequence of Bousfield localizations
which preserve the property of being \pcas.  
But note that our actual argument is more indirect as we first show \cref{jifofqweewf9} and then use it in order to deduce
\cref{rjifoqopffefwewerfr}.

 
  In the present paper we will discuss some consequences   of  being \pcas{} which are interesting in the context of $C^{*}$-algebra theory.

By \cref{ijwqfofjewfqedfq9} a morphism $B\to C$ in a {pointed countably cocomplete $\infty$-category $\bD$} is called a strong phantom map if the composition $A\to B\to C$ vanishes for every weakly compact map $A\to B$. {In the present paper, we will study properties of the class of {strong} phantom maps under the assumption that $\bD$ is \pcas{}.}

{Strong phantom maps in equivariant $E$-theory  occur naturally. For example, if}
two homomorphisms $f,g:A\to B$ in $G\nCalg_{\sepa}$ are approximatively unitary equivalent, then $\esG(f)-\esG(g)$ is a  {strong} phantom map in $\EsG$ (see  \cref{jireogwegrewgw9}).  {In \cref{werrwer34rwffefwefwerf} we will show that the boundary {map} $\esG(C)\to \Sigma \esG(A)$ associated to a weakly quasi-diagonal extension $0\to A\to B\to C\to 0$ is a  {strong} phantom map.}
 {Strong phantom} maps in equivariant $E$-theory were also used in  \cite{Bunke:2023ab} in order to capture the Rokhlin property of $G$-$C^{*}$-algebras for finite groups $G$.

If $\bD$ is an $\infty$-category
admitting countable  filtered colimits, then the  morphism sets of its homotopy category $\ho(\bD)$ have {two} interesting topologies, namely the shape topology and the light condensed topology (the latter is only defined under the additional assumption of the existence of finite products), see \cref{erjogpwergwerfrwefwrefrwefwer} and  \cref{werjigowegerfwerfwrefw}. If $\bD$ is \pcas{}, then these topologies coincide and have very nice properties. In the case of $\AsGc$ and $\EsG$ the  topologies on the morphism sets of their homotopy categories $\AsG$ and $\EsGn$ have previously been introduced in \cite{Carrion:2023aa}. In \cref{wgjkopwetgerwfrefwrfwref}
we will argue that most of the statements about these topologies obtained in  \cite{Carrion:2023aa} are specializations of  general statements which are true for the homotopy category of any \pcas{} $\infty$-category.

\subsection{Description of the contents of the paper}

The first section  provides a reference for the basic definitions and results  from the theory of {compactly assembled} $\infty$-categories with countable filtered colimits. In \cref{werogjwpegorefkmlerwgwergwerg} we introduce the notions of {weakly compact, slp compact} and strongly compact morphisms. In \cref{ergkoperwgwerfwefref} we introduce the notion of a shape of an object. 
If every object of the category admits a shape, then the category  turns out to be \pcas. We then analyse 
the interaction of compactness and shapes with functors. In particular for any $\infty$-category with countable  filtered colimits $\bD$ we find a maximal full subcategory which is \pcas, see \cref{wregiojwoeprgerwfrewfwref} for a precise formulation.  In \cref{erwgjoiwpergwerfrefrewfw} we  give a sufficient criterion {for} the existence of shapes of an object in terms of the existence {of} slp-compact exhaustions (see \cref{wkotpgsfvsfgfdgg}).
Note that some of the statements shown in these sections would simplify if one assumed that the $\infty$-category were \pcas. But we decided to
develop this theory in the general case in order to make it applicable also to the stable $\infty$-category  $\KK^{G}_{\sepa}$ from \cite{KKG} which is not known to be \pcas. So we can now form and study (in future) the  largest \pcas{} subcategory $(\KK^{G}_{\sepa})^{\casmbl}$ of $\KKG_{\sepa}$ or consider those  objects in $\KKG_{\sepa}$ which are known to admit 
shapes.
In \cref{qifjgioerfwefdqwefqwef} we show a technical result (crucially used in  \cref{erwgjoiwpergwerfrefrewfw}) ensuring that a slp-map of  {towers} $(A_{n})_{n\in \nat^{\op}}\to (B_{n})_{n\in \nat^{\op}}$  of  anima induces an equivalence on the level of limits.   Finally, in \cref{weokgpwegwerfwfrfrefwerf} we introduce the notion of strong phantom maps in a pointed $\infty$-category $\bD$ admitting countable filtered colimits (see \cref{ijwqfofjewfqedfq9}). If $\bD$ is stable, then we show  in \cref{iujgofffsdafdfaf} that the existence of shapes  of objects  (or even the weaker condition of weakly compact  approximations) allows 
to form countable infinite sums of strong {phantoms.  
If} $\bD$ is \pcas, then as a consequence of \cref{weokrgpwregrefrfw9} the composition of any two phantom morphisms in $\bD$ is trivial. We furthermore show in \cref{wiotgjwogirwefwefw} that a filtered colimit preserving exact  functor between compactly assembled stable $\infty$-categories preserves strong phantom maps.

The \cref{wrejgiowergerfwfwef} is devoted to the construction of the functors $\asGc$ and $\esG$ in \eqref{fqwefewdqwedwefqf} and the verification of our main result \cref{jifofqweewf9}.
As a start, in \cref{erijgioepgwegerfwef} we extend the theory of \cite{zbMATH03996430} to the equivariant case. 
The main new insight is that the homotopy  lifting results for  explicit semi-projective  maps shown in  \cite{zbMATH03996430}
can be interpreted  as the slp lifting property  for squares (see \cref{kotgpergoeprtkhgerth}) of mapping spaces in the homotopy localization $G\nCalg_{\sepa,h}$, see \cref{xjnbcuiohwedfiohgeriosdfylk}. This is the basis for translating the
semi-projectivity of $f$  in $G\nCalg_{\sepa}$ to versions of compactness of the maps   $\asGc(f)$ or $\esG(f)$  in $\AsGc$ or $\EsG$. 

In \cref{gjkowperefwfrefrfrefw} we characterize the equivariant $E$-theory functor $\esG:G\nCalg_{\sepa}\to \EsG$ essentially uniquely by universal properties. In particular we explain in detail
all $C^{*}$-algebraic notions appearing in this characterization.
  This sets up the task of showing the existence of such a functor 
by providing a construction. This  construction is the program of the subsequent \cref{twhkogprthtrerge},  \cref{twhkogprthtrerge1} and \cref{irthjzgjdiogjhjdicnmjksdisjf}. We provide three essentially independent constructions of $\esG$ which highlight different properties and foundations.

 In    \cref{twhkogprthtrerge} we construct a functor
$\esGght$ as a Dwyer-Kan localization of $G\nCalg_{\sepa}$ at the morphisms which are inverted by the classical $E$-theory functor $\esGn$. This construction closely follows  \cite{LN}, \cite{KKG} giving   analogous constructions of stable $\infty$-categories representing $KK$-theory. The subscript $GHT$ abbreviates Guentner-Higson-Trout
and indicates that this functor is designed to be an $\infty$-categorical enhancement of the classical functor $\esGn$ constructed by these authors in  \cite{Guentner_2000}. The verification that it has the desired universal properties depends heavily on the classical equivariant $E$-theory \cite{Guentner_2000}. 

In \cref{twhkogprthtrerge1} we give a purely homotopy theoretic construction of a functor $\esGs$ by forcing the desired universal properties. This construction is    an equivariant  generalization of the construction of the non-equivariant $E$-theory  $\es:\nCalg_{\sepa}\to \Es$ in \cite{keb}. It is independent of the classical $E$-theory. 
In the following we stress two aspects which are new in comparison with  \cite{keb}. In
 the non-equivariant case the $K$-stabilization was a left Bousfield localization $L_{K}$ induced by an idempotent
algebra in $\nCalg_{\sepa,h}$  represented by the algebra $K$ of compact operators. 
The $K_{G}$-stabilization   in the case of finite groups is completely analogous. But for {infinite} groups $G$
 the corresponding $K_{G}$-stabilization consists of the composition of a left Bousfield localization  induced by an idempotent algebra $\hat K_{G}=K((L^{2}(G)\oplus \C)\otimes \ell^{2})$ in $G\nCalg_{\sepa,h}$ and a right Bousfield localization induced by $K_{G}$ considered as an idempotent coalgebra in $L_{\hat K_{G}}G\nCalg_{\sepa,h}$. The  problem in the case of an {infinite}   group  $G$ is that
$K_{G}$ lacks invariant projections. They are   added to $\hat K_{G}$ by adding the artifical summand $\C$ to $L^{2}(G)$,
but then $L_{\hat K_{G}}$ does not invert all $K_{G}$-equivalences so that the second localization becomes necessary.
The details are explained in \cref{tgokptegerferfwefwf}.
The second difference to \cite{keb} is that we {modify the step forcing exactness}  in order to ensure that the $E$-theory functor
preserves countable sums. This {property} is indicated by the subscript $\oplus$ in the notation.

Finally in \cref{irthjzgjdiogjhjdicnmjksdisjf} we first construct the universal homotopy invariant and countable {filtered} colimit preserving
functor $\asGc:G\nCalg\to \AsGc$. This construction is strongly  inspired {by}  but technically independent {from} 
\cite{Dadarlat_1994}. We then proceed similarly as for $\esGs$ and force the  remaining universal properties  
in order to get a functor $\esGf$. The subscript $\fin$ stands for finitary and indicates the relevance of the property of preserving countable filtered colimits. The step in which we enforce Bott periodicity
depends technically on the identification  of $\ho\circ \asGc$ with the classical functor $\asG$
and  facts about the asymptotic morphism representing the Bott element  shown in  \cite{Guentner_2000}.

In \cref{xvbnfiohregiohsdfpojsdf} we then verify that the functors $\esGght$, $\esGs$ and $\esGf$ are really equivalent in an essentially unique manner. We can now safely use the notation $\esG$ for any of them.
Note that each of the constructions adds some additional information. So $\esGght$ provides the information that $\esG$ is a Dwyer-Kan localization. To understand the  {functor $\esGs$} does not require to understand asymptotic morphisms so that this construction may make equivariant $E$-theory more accessible to homotopy theory oriented readers. Finally, the construction of the functor $\esGf$ is crucial for our proof of the main \cref{jifofqweewf9} since 
it allows to follow through the consequences of the existence of shape systems for separable  $G$-$C^{*}$-algebras to the existence of shapes of the objects in $\EsGf$.

 In \cref{kogpwergefrefwrfwrf} we finally show our main  \cref{jifofqweewf9}.
The idea is to present $\EsGf$ as the result of a sequence of left- and right Bousfield localizations with filtered colimit preserving right adjoints of 
the  category $\Ind^{\aleph_{1}}(G\nCalg_{\sepa,h})$. We then use that a category of the form $\Ind^{\aleph_{1}}(\dots)$ is always compactly generated and therefore \pcas, and that the localizations preserve the property of being \pcas.
  In addition in \cref{jigowggergwe9} we describe explicit representatives for the shapes of objects in $\AsGc$ or $\EsG$.

 {In \cref{gugeiwprgejrgpowjeipof} we discuss further properties of the category $\Es$. Using the classical example of Skandalis \cite{zbMATH04065714} we observe in  \Cref{vbnuiojaweiofhjjdf}.\ref{uirjofjfgjfid} that 
 $\Es$ is strictly larger than its bootstrap class.
 This shows that $\Es$ is compactly assembled for a non-trivial reason. The question whether
  $\Es$  is even rigid remains open for the moment.  But  in \Cref{vbnuiojaweiofhjjdf}.\ref{wiofnvgkeqka}-\ref{oiiaenwfgonaoej2qr}. we can answer the analogous question for
  $KK$-theory by showing that     $\KKs$ with one of  $\otimes$ or $\otimes_{\min}$ is not rigid.}
  
  {As an application of $\Es$ {being compactly assembled} we can show that 
  a map of separable $C^*$-algebras $f: A \to B$ is an $E$-theory equivalence  if and only if it is inverted by all functors $L: \nCalg_{{\sepa}} \to \Sp$ that are $K$-stable, homotopy invariant, exact and preserve filtered colimits \Cref{oenvgpwffoq}. Note that these are precisely the formal properties of the $K$-theory functor, see \cref{ijvwiopjkdfehj9}.

 In \cref{tgkowpegfwerwggerfwf} we study topologies on the morphism sets $[A,B]$ of the homotopy category of an $\infty$-category $\bD$  with countable filtered colimits.
In \cref{erjogpwergwerfrwefwrefrwefwer} we introduce  (under the additional assumption of the existence of finite products) the light condensed topology. The name {is}  motivated by the recent development of {\em condensed mathematics} advertised by D. Clausen and P. Scholze. If $X$  is a light profinite  {space} (i.e., countable cofiltered limit $X\cong \lim_{i\in I} X_{i}$ of finite discrete spaces), then we can define the power $B^{X}:=\colim_{i\in I^{\op}}\prod_{X_{i}}B $. Given an element of $[A,B^{X}]$ the composition with the evaluations at the points of $X$ gives a map  $X\to [A,B]$.
The light condensed topology on the  morphism sets  is defined by   declaring  all these maps to be continuous. 

In  \cref{werjigowegerfwerfwrefw} we introduce a second topology called the shape topology on the morphisms sets  by the condition that $[A,B]\to [A',B]^{\disc}$ should be continuous for all weakly compact morphisms ${A'\to A}$. {Here, the superscript $(-)^{\mathrm{disc}}$ indicates that the codomain carries the discrete topology.}
We will see that in the presence of weakly compact approximations (so in particular if $\bD$ is \pcas) the light condensed topology coincides with the shape topology.  This interplay has nice consequences:
the topology is  first countable, the composition map is continuous,  the set of converging sequences in $[A,B]$ {is} modelled by the set $[A,B^{\nat_{+}}]$, {and in the stable case $[A,B]$ is a topological group.}  The study of the topological morphism sets of $\ho(\bD)$ is just a one-categorical reflection of a light condensed anima structure on the mapping  {anima} of $\bD$ which will be investigated in a future work.

In \cref{erjogiopwregwrefwref} {we} define a category $\bar \bD$ by forming the Hausdorff quotient of the morphism spaces.   If $\bD$ is \pcas{}, then $\bar \bD$ is topologically enriched, i.e., the composition is continuous. Furthermore,  the  functor $\ho(\bD)\to \bar \bD$    preserves filtered colimits \cref{kopgpertherhetrhgeget} and   is conservative  by \cref{wekotgpwgrefwf} provided  {that} $\bD$ is stable. 

Again we formulate the statements in  \cref{tgkowpegfwerwggerfwf} for general
$\infty$-categories admitting countable filtered colimits and require additional assumptions about the existence of weakly compact approximations etc locally in order to make it applicable e.g. to  categories like $\KKGs$ which are not known to be \pcas.

In the final \cref{wgjkopwetgerwfrefwrfwref} we specialize the results of the preceding sections to the categories $\AsGc$ and $\EsG$ and explain how they are related with the results of \cite{Carrion:2023aa}. A new aspect in the case of $C^{*}$-algebras is that one can relate the analysis in $G\nCalg_{\sepa}$  directly with analysis in the $\infty$-categories $\AsGc$ or $\EsG$ (in the sense of condensed mathematics).  This applies e.g. to {the} construction of function objects.
By \cref{ojiviopwgfrwerferfwf} for every $B$ in $G\nCalg_{\sepa}$ and light profinite space $X$ we have equivalences $$\asGc(C(X,B))\simeq \asGc(B)^{X}\ , \qquad  \esG(C(X,B))\simeq \esG(B)^{X}\ .$$
Furthermore, by \cref{gjweriofwefgdbgb} the functors
$$\asG:G\nCalg_{\sepa}\to \AsG\ , \qquad  \esGn: G\nCalg_{\sepa}\to \EsGn$$
are continuous with respect to the topological enrichments of the domain (point-norm topology on the sets of homomorphisms) and the shape topology on the morphism sets.

 {\em Acknowledgement: {We thank Anupam Datta and the referee for helpful comments on \Cref{jgriowergwrefrefrefwfrfwf}.} The authors were supported by the SFB 1085 (Higher Invariants) funded by the Deutsche Forschungsgemeinschaft (DFG). }

\section{$\infty$-categories}\label{kopwretgretreht}

\subsection{Compact maps}\label{werogjwpegorefkmlerwgwergwerg}
 
 In {this} section we introduce variants of the concept of a compact map  in the context of countably cocomplete $\infty$-categories.
As  preparation we introduce the following two conditions on commutative squares in the $\infty$-category of anima $\Spc$. 
  We call the bold commutative square
 $$ \xymatrix{X \ar[r] & Y \\ K \ar[r] \ar[u] & L\ar@{-->}[ul]\ar[u]}$$
 of anima liftable if it admits the dashed completion. Note that this involves the condition that
 the composition of the fillers of the two triangles is equivalent to the filler of the square.
 	Equivalently, a bold square as above 
	  is liftable if and only if the associated bold diagram
	\begin{equation}\label{puighygbqrjkb}
  \xymatrix{& K\ar[dd]\ar[dl]\ar[dr]& \\ L\ar[dr]  \ar@{-->}[rr] & & X\ar[dl] \\&Y&} \end{equation}
	in the slice category $\Spc_{/Y}$ of anima over $Y$ admits a dashed filler.

\begin{ddd}\label{kotgpergoeprtkhgerth} We  say that  the upper bold  commutative square
$$ \xymatrix{X_0 \ar[r] & Y_0 \\ X_1 \ar[r] \ar[u] & Y_1 \ar[u]\\ K \ar@{..>}[r]  \ar@{..>}[u] &L\ar@{-->}[uul]  \ar@{..>}[u]}$$
satisfies the shifted lifting property (slp), if for every   dotted square with finite anima $K$  and $L$
the outer square is liftable as indicated by the dashed arrow. 

If we  fix $K\to L$,  then we will also say that it has the slp for $K\to L$. \end{ddd}
		
We consider an   $\infty$-category $\bD$  which admits all countable filtered colimits.
 \setlist[enumerate]{itemsep=-0.4cm}
By a  system $(Z_{i})_{i\in I}$  in $\bD$   we mean a functor $Z:I\to \bD$ from a countable filtered $\infty$-category. 
  
  We consider a morphism $f:A\to B$ in $\bD$.
 If $(Z_{i})_{i\in I}$  is a  system  in $\bD$, then we can consider the square of anima
 \begin{equation}\label{qrfwfewdqewdqe} 
\xymatrix{\colim_{i\in I}\Map_{\bD}(A,Z_{i})\ar[r] & \Map_{\bD}(A,\colim_{i\in I}Z_{i})  \\\ar[u]^{f^{*}} \colim_{i\in I}\Map_{\bD}(B,Z_{i})\ar[r] & \ar[u]^{f^{*}} \Map_{\bD}(B,\colim_{i\in I}Z_{i})} 
\end{equation}

\begin{ddd}\label{qrjiofoqdedeqw}
 The morphism $ f:A\to B$   is called:
 \begin{enumerate}
 \item \label{trgkprerergwerferfwfwegfwerfweffwef} strongly compact if \eqref{qrfwfewdqewdqe} is liftable
 \item\label{trgkprererwegfwerfweffwef}  slp compact, if \eqref{qrfwfewdqewdqe} has the slp
 \item\label{trgkpwegfwerfweffwef} weakly compact if  \eqref{qrfwfewdqewdqe} has the slp for $\emptyset\to *$
 \end{enumerate}
 for every  system $(Z_{i})_{i\in I}$.
 \end{ddd}
For all three versions of compactness the sets of corresponding morphisms form a two-sided ideal in $\bD$.
 We have  the following implications:
 $$\mbox{strongly compact}\implies \mbox{slp compact}\implies \mbox{weakly compact}\ .$$
 If $A$ is a compact object, then any map $A\to B$ is strongly compact since the upper map in \eqref{qrfwfewdqewdqe} is an equivalence.
 
 \begin{rem}\label{jowtogpwgerferfwefewfref}The condition of being weakly compact is equivalent to the condition that 
  for every 
 system $(Z_{i})_{i\in I}$
and bold part of the diagram
\begin{equation}\label{vsdfvsdfefefvsvfsdvsv}\xymatrix{A\ar[d]^{f}\ar@{..>}[r] & Z_{k}\ar[d] \\B \ar[r] & \colim_{i\in I}Z_{i}}
\end{equation}  
there exists  $k$ in $I$ and a commutative completion by the dotted part. \hB
 \end{rem}

\begin{rem}
Strongly and  weakly compact maps are usually considered in the context of large presentable $\infty$-categories \cite{som}  where the definitions are attributed to J. Lurie and D. Clausen.
In this case one drops the countability condition for the systems $(Z_{i})_{i\in I}$.

A typical example for the  $\infty$-category $\bD$  is the category of $\aleph_{1}$-compact
objects in a large presentable $\infty$-category. In this case (weak) compactness of a morphism in $\bD$ is equivalent to (weak) compactness (as defined in \cite{som}) of the morphism considered in the large category. 
  
We introduce the  intermediate notion of slp compactness since  it is the  slp condition which can   
directly be checked in our applications to $C^{*}$-algebras, see e.g. \cref{xjnbcuiohwedfiohgeriosdfylk}.
\hB
\end{rem}

 \begin{rem}\label{iuergoregregwf}
If $\bD$ is a stable countably cocomplete $\infty$-category, then weak compactness of a morphism  $f:A\to B$ can also be characterized by the following conditions:
 \begin{enumerate} \item 
For any   system $(F_{j})_{j\in J}$ with $\colim_{j\in J}F_{j}\simeq 0$ and map $B\to F_{k}$ for some $k$ in $J$ there exists $l$ in $J$ with a morphism $k\to l$ such that
$A\xrightarrow{f}B\to F_{k}\to F_{l}$ vanishes.
\item  For every family $(Y_{n})_{n\in \nat}$ in $\bD$ and bold part of the diagram 
$$\xymatrix{A\ar@{..>}[r]\ar[d] &\bigoplus_{n=0}^{k} Y_{n}\ar[d] \\ B\ar[r] & \bigoplus_{n\in \nat} Y_{n}} $$
there exists an integer $k$ and the dotted completion.
\end{enumerate} 
Furthermore, in the stable case it follows from  \cref{vniowsdfiojgdf} that 
  weakly compact maps are slp compact. 
\hB
\end{rem}

The following is straighforward to check for all three cases   of $?$ in \{strongly, slp, weakly\}.
\begin{lem}\label{aerpogjkwpergwefffwrfrefw}
A left-adjoint functor $L:\bD\to \bC$ with a countable colimit-preserving right-adjoint $R:\bC\to \bD$ preserves $?$-compact maps.
\end{lem}

\subsection{Shapes and \pcas{} $\infty$-categories}\label{ergkoperwgwerfwefref}

In this section we introduce the concept of a shape of an object in an $\infty$-category which admits countable filtered colimits.   
The word shape is motivated by our application of the theory to  $C^{*}$-algebras culminating in \cref{jigowggergwe9} asserting 
that a shape system  for a $C^{*}$-algebra represents the shape of its image in equivariant $E$-theory. 

The condition that every object of $\bD$ admits a shape turns out to be  one of many equivalent characterizations of $\bD$ being \pcas. But since we also want to apply the theory  to the stable $\infty$-category $\KK^{G}_{\sepa}$ from  \cite{KKG} which is not known to be   \pcas{}
we develop the notions as much as possible locally  in $\bD$.

Let $\bD$ be a small $\infty$-category admitting all countable filtered colimits. 
By $y:\bD\to \Ind(\bD)$ we denote the  ind-completion of $\bD$. We then let
$\Ind^{\aleph_{1}}(\bD) $ denote the full subcategory of $\Ind(\bD)$ obtained as the closure of the image of $y$ under countable filtered colimits.  Note
 that  $\Ind^{\aleph_{1}}(\bD) $ is essentially small.

 From now one we consider the functor $y:\bD\to \Ind^{\aleph_{1}}(\bD)$. 
 Below we will use the following facts.
  \begin{enumerate}\item   For  every object in $ \Ind^{\aleph_{1}}(\bD) $ there exists a  system $(A_{n})_{n\in \nat}$ in $\bD$ such that    
$$\colim_{{n \in \nat}}y(A_{{n}}) \simeq A\ .$$ {This} can be shown using \cite[Prop. 5.3.1.18]{htt}.
\item The functor $y$ is fully faithful  and preserves all limits. 
\item The objects in the image of $y$ are  compact.
\item  If $\bD$ is stable, then so is $\Ind^{\aleph_{1}}(\bD) $  and $y$ preserves   finite  colimits.   
  \end{enumerate}

Since we assume that $\bD$ admits all countable  filtered  colimits
we have an adjunction
\begin{equation}\label{fvvsqervfsvsdf}|-|: \Ind^{\aleph_{1}}(\bD) \leftrightarrows  \bD:y\ ,
\end{equation} 
where $|-|$ is called the realization. 
Since  $y$ is fully faithful the counit of this adjunction  is an  equivalence  \begin{equation}\label{sfdvojkogpvvsfdvsdvsdfv}|-| \circ y \stackrel{\simeq}{\to} \id_{\bD}\ .
\end{equation} 
\begin{ddd}\label{ekrogpegegferfwf}
We say that an object $A$ in $\bD$  admits a shape  if the functor $$\Map_{\bD}(A,|-|): \Ind^{\aleph_{1}}(\bD) \to \Spc$$
is corepresentable.
 We let $\bD^{\shp}$ denote the full subcategory of objects which admit a shape. \end{ddd}
 By definition,  the shape of $A$ is an object $S(A)$ in $ \Ind^{\aleph_{1}}(\bD)$ together with a map
 $\eta_{A}:A\to |S(A)|$   such that the induced map 
\begin{equation}\label{sdfvfeqrvcfvsdvfdsvs}\Map_{ \Ind^{\aleph_{1}}(\bD)} (S(A),\hat B)\stackrel{|-|}{\to}\Map_{\bD}(|S(A)|,|\hat B|)
\stackrel{{\eta_{A}^{\ast}}}{\to} \Map_{\bD}(A,|\hat B|)
\end{equation} 
is an equivalence for all $\hat B$ in $ \Ind^{\aleph_{1}}(\bD) $.
The collection of shapes and maps assemble to the shape functor 
 and natural transformation \begin{equation}\label{fwqeddcvsdvcs}S:\bD^{\shp}\to \Ind^{\aleph_{1}}(\bD) \ , \qquad \eta: \incl_{\bD^{\shp}\to \bD}\to |-|\circ  S \ . 
\end{equation}  

\begin{lem}\label{jwoirthopgfrfwrgre}
The transformation $\eta$ in \eqref{fwqeddcvsdvcs} is an equivalence.
\end{lem}
\begin{proof}
We consider  $A$ in $\bD^{\shp}$ and $B$ in $\bD$ and the following diagram
   $$\xymatrix{&\Map_{\Ind^{\aleph_{1}}(\bD)   }(S(A),y(B))\ar[d]^{|-|}\ar[ddr]^{\simeq}\ar[ddl]_{\simeq}&\\&\Map_{\bD}(|S(A)|,|y(B)|)\ar[dr]^{\simeq}_{counit}\ar[dl]^{\eta_{A}^{*}}&\\ \Map_{\bD}(A,|y(B)|)\ar[dr]_{\simeq}^{counit}&&\Map_{\bD}(|S(A)|,B)\ar[dl]_{\eta_{A}^{*}} \\&\Map_{\bD}(A,B)&}$$
   where the equivalences denoted by  $counit$  are induced by post-composition with the counit map of  {the} adjunction \eqref{fvvsqervfsvsdf}. {This adjunction} also implies that the upper right diagonal map is an equivalence.
   The upper left diagonal map is  the  equivalence  \eqref{sdfvfeqrvcfvsdvfdsvs} applied to $\hat B:=y(B)$.  
    We conclude from the upper {right} triangle that the upper vertical map $|-|$ is an equivalence.
        Consequently the left diagonal map $\eta_{A}^{*}$ is an equivalence. 
        We finally conclude that the right diagonal map $\eta_{A}^{*}$ is an equivalence. 
 This implies the assertion.
      \end{proof}

\begin{lem}\label{werojgpwertgwerrfwerfw}
$\bD^{\shp}$ is   closed under countable filtered  colimits.
If $\bD$ is stable, then so is $\bD^{\shp}$.
\end{lem}
\begin{proof}
 Let $(A_{i})_{i\in I}$ be a system in $\bD^{\shp}$.
Then using  \eqref{sdfvfeqrvcfvsdvfdsvs} one checks that  {$\colim_{i \in I} S(A_i)$} represents a shape for  {$\colim_{i \in I}A_i$}. 
 Hence $\bD^{\shp}$ is closed under countable filtered colimits.
 If $\bD$ is stable, then one furthermore checks that $\bD^{\shp}$ is closed under finite colimits, and  
  that $\Omega  S(A)$ represents the shape of $\Omega A$.
 Consequently $\bD^{\shp}$ is a stable subcategory of $\bD$.
 \end{proof}

\begin{lem}\label{wegjiopewgfsg} We assume that  $A$ is in $\bD$, $B$ is in $\bD^{\shp}$, and  that $f:y(A)\to S (B)$ is any map. Then the map  \begin{equation}\label{efwefaefsdfaef}  |y(A)|\xrightarrow{|f|} |S(B)|
\end{equation}  is strongly compact. \end{lem}
 {\begin{proof}
	Let $(Z_{i})_{i\in I}$ be a  system in $\bD$.  
	Then we consider the bold part of the diagram 
	\begin{equation}\label{qrfwfewdqewdqe1} 
		\xymatrix{\colim_{i\in I}\Map_{\bD}(\lvert y(A)\rvert,Z_{i}) \ar[r] & \Map_{\bD}(\lvert y(A) \rvert,\colim_{i\in I} Z_i) \\\colim_{i\in I}\Map_{\Ind^{\aleph_{1}}(\bD)}(y(A),y(Z_{i}))\ar[r]^-{\simeq}_{!} \ar[u]^{\lvert - \rvert}_{\simeq}&\ar[u]^{\lvert - \rvert} \Map_{\Ind^{\aleph_{1}}(\bD)}(y(A), \colim_{i \in I} y(Z_i))    \\\ar[u]^{f^{*}} \colim_{i\in I}\Map_{\Ind^{\aleph_{1}}(\bD) }(S(B),y(Z_{i}))\ar[r] & \ar[u]^{f^{*}}  \Map_{\Ind^{\aleph_{1}}(\bD)}(S(B),\colim_{i \in I}y(Z_i))     \\\ar@{<-}[u]^{\lvert - \rvert}_{\simeq} \colim_{i\in I}\Map_{\bD}(|S(B)|,Z_{i})\ar[r] & \ar@{<-}[u]^{\lvert - \rvert}_{\simeq} \Map_{\bD}(|S(B)|,\colim_{i \in I}Z_i)\ar@{-->}[uuul] } \ .
	\end{equation}
	Here, we have implicitly used the natural identifications $\lvert y(Z_i) \rvert \simeq Z_i$ and $\lvert \colim_{i \in I} y(Z_i) \rvert \simeq \colim Z_i $. 
	The marked map is an equivalence since $y(A)$ is compact, and the two lower vertical maps are {equivalences} due to the universal property of the shape \eqref{sdfvfeqrvcfvsdvfdsvs}. It is straightforward to see that the composite of the vertical maps is given by $\lvert f \rvert^\ast$. Using that the marked map is an equivalence one find the desired dashed arrow.  \end{proof}}

 \begin{lem}\label{kowogpergrefrefrfw} If $A,B$ are in $\bD^{\shp}$, then any weakly compact  map $f:A\to B$   is automatically strongly compact.   \end{lem} \begin{proof}  Let $f:A\to B$ be a weakly compact map. Then $S(f):S(A)\to S(B)$ is
	weakly compact by \cref{aerpogjkwpergwefffwrfrefw}.  We consider a system
	$(B_{n})_{n\in \nat}$  in $\bD$ such that $\colim_{n\in \nat} y(B_{n})\simeq S(B)$. Using \cref{jowtogpwgerferfwefewfref} we see that there exists  
	$n$ in $\nat$ such that $S(f)$ factorizes over $S(A)\to y(B_{n})\to S(B)$. We recover $f$ by applying $|-|$ 
	and use \cref{wegjiopewgfsg} in order to conclude that $f$ is strongly compact.
\end{proof}

\begin{ddd}\label{ijtowggwerwgerg9}
We say that $\bD$  is  \pcas{}  if  $\bD^{\shp}=\bD$.  
\end{ddd}

If $\bD$ is  \pcas, then the shape functor is  {the} left-adjoint {in the}  adjunction 
\begin{equation}\label{fwqeddcvsdvcs1}S:\bD\leftrightarrows \Ind^{\aleph_{1}}(\bD) :|-|\ , \qquad \eta:\id_{\bD}\stackrel{\simeq}{\to} |-|\circ S \end{equation} whose 
  unit  {$\eta$}    is an equivalence by \cref{jwoirthopgfrfwrgre}.

{
\begin{rem}
	A small \pcas{} $\infty$-category is the category of $\aleph_{1}$-compact objects in a
	large $\infty$-category which is compactly assembled in the sense of Lurie \cite[Sec. 21.1.2.]{sag}. {The latter notion}
	 is an $\infty$-categorical version of the notion of a {continuous} category by Johnstone-Joyal \cite{jj}.
	\hB
\end{rem}
}

 The following is a consequence of \cref{kowogpergrefrefrfw}.
 \begin{kor}\label{kgoprgfeegferfwef}
 If $\bD$ is \pcas, then weakly compact morphisms are automatically strongly compact.
 \end{kor}
 Thus in a \pcas{} category all our  compactness variants {from \cref{qrjiofoqdedeqw}} coincide.

Let now $\bD$ be an $\infty$-category admitting countable filtered colimits and $S:\bD^{\shp}\to \bD$ be the partial shape functor from \eqref{fwqeddcvsdvcs}. 
 If $\bD'$ is a full subcategory of $\bD$, then we can consider $\Ind^{\aleph_{1}} (\bD') $ as a full subcategory of $\Ind^{\aleph_{1}} (\bD) $.
In the following we will  construct the maximal \pcas{} full subcategory $\bD^{\casmbl}$ of $\bD$, see \cref{wregiojwoeprgerwfrewfwref} for a precise statement.
 Note that in general  we do not expect that the full subcategory  $\bD^{\shp}$ of $\bD$  is   \pcas{}  since it is not clear that the functor $S$ takes values in the full subcategory $\Ind^{\aleph_{1}} (\bD^{\shp}) $ of $\Ind^{\aleph_{1}} (\bD ) $. 
  We therefore define by {transfinite recursion} a decreasing family {$(\bD^{(i)})_{i}$ of full subcategories
  of $\bD$}. We set $\bD^{(0)}:=\bD^{\shp}$. {Let $i$ be an ordinal and assume that $\bD^{(j)}$ has been defined for all ordinals $j<i$. 
  \begin{enumerate}
  \item If $i$ is a successor ordinal, then $i=j+1$ for some ordinal $j$ and we define  $\bD^{(i)}$ to be the subcategory of $\bD^{(j)}$ on objects $A$ with $S(A)$ in $\Ind^{\aleph_{1}} (\bD^{(j)}) $.
  \item If $i$ is a limit ordinal, then we define $\bD^{(i)}:=\bigcap_{j<i} \bD^{(j)}$. \end{enumerate}
 Since $\bD$ is small there exists an ordinal $i'$ such that $\bD^{(i')}=\bD^{(i'+1)}$. We then define  the full subcategory  $$\bD^{\casmbl}:= \bD^{(i')}$$  of $\bD$.   This is independent of the choice of $i'$.}

\begin{lem}\label{ekopgwfwregwrgweg}
The  subcategory $\bD^{\casmbl}$ is closed under countable filtered colimits  and   \pcas.  If $\bD$ is stable, then so is  $\bD^{\casmbl}$.    
\end{lem}
\begin{proof}
 The subcategory  $\bD^{\shp}$ is  closed under countable filtered colimits by \cref{werojgpwertgwerrfwerfw}.
Since the functor $S$ preserves filtered colimits we can conclude by transfinite induction  that
 $ \bD^{(i)}$ is closed under countable filtered colimits {for every ordinal $i$}.
 This implies that $\bD^{\casmbl}$ is closed under countable filtered colimits.
By construction, if $A$ is in  $\bD^{\casmbl}$, then $S(A)\in\Ind^{\aleph_{1}} (\bD^{\casmbl}) $. Consequently, 
$\bD^{\casmbl}$ is  \pcas.  

For stability we argue similarly, using the second assertion of  \cref{werojgpwertgwerrfwerfw}.
\end{proof}

 Let $F:\bC\to \bD$ be a  countable  filtered colimit preserving  functor between  $\infty$-categories addmitting countable filtered colimits. 
 It essentially uniquely  extends to a countable filtered colimit preserving functor
 $\hat F $ fitting into \begin{equation}\label{wrefrefreferwffwrfw} \xymatrix{\bC\ar[r]^{F}\ar[d]^{y_{\bC}} &\bD \ar[d]^{y_{\bD}} \\\Ind^{\aleph_{1}} (\bC )  \ar[r]^{\hat F} &\Ind^{\aleph_{1}} (\bD ) }  \ .
\end{equation}
The condition that $F$ preserves countable filtered colimits  is equivalent to the condition that the natural comparison morphism  is an equivalence 
 \begin{equation}\label{fwewerfreferfwfref2}\lvert - \rvert_\bD \circ \hat F \stackrel{\simeq}{\to} F \circ \lvert - \rvert_\bC\ .\end{equation}  

\begin{ddd} \label{gijwreigoerfureoifuoerfwerfefwefwefrerfw}

We say that $F$ preserves shapes  if for every $A$ in $\bC^{\shp}$ the  object $\hat F(S_{\bC}(A))$ in 
$\Ind^{\aleph_{1}}(\bD)$ together with the map $$F(A)\stackrel{F(\eta_{A})}{\to}   F(|S_{\bC}(A)|_{\bC}) \stackrel{\simeq}{\to} |\hat F(S_{\bC}(A))|_{\bD}$$ represents the shape of $F(A)$.
 \end{ddd}
 Note that for a shape preserving functor $F$ we have  $$F(\bC^{\shp})\subseteq \bD^{\shp}\ .$$

 The following result provides examples of shape preserving functors.
 Assume that $$L:\bC\leftrightarrows \bD:R$$ is an adjunction.    \begin{lem}\label{kjogperwgefweff} If  $R$ preserves countable  filtered colimits, then
 $L$ preserves shapes.
 \end{lem}
 \begin{proof}
 We have an induced adjunction 
 $$\hat L:\Ind^{\aleph_{1}}(\bC)\leftrightarrows \Ind^{\aleph_{1}}(\bD):\hat R\ .$$
 The equivalence of functors $\bC^{\shp,\op}\times \Ind^{\aleph_{1}}(\bD)\to \Spc$
 \begin{align*}
\Map_{\Ind^{\aleph_{1}}(\bD)}(\hat L(S_{\bC}(-)),-)\simeq\Map_{\Ind^{\aleph_{1}}(\bC)}(S_{\bC}(-),\hat R(-))
&\simeq\Map_{ \bC}(-, |\hat R(-)|_{\bC})\\\ \simeq 
\Map_{ \bC}(-, R(|-|_{\bD}))&\:\:\simeq
\Map_{ \bD}(L(-), |-|_{\bD}) 
\end{align*}
sends  for every $A$ in $\bC^{\shp}$ the identity $\id_{\hat L(S_{\bC}(A))}$ to the map
$L(\eta_{A}):L(A)\to | \hat L (S_{\bC}(A))|_{\bD}$ which presents 
  $\hat L(S_{\bC}(A))$  as the shape of $L(A)$.
  \end{proof}

 Let $\Cat_{\infty}^{\cfil,\shp}$  be the subcategory of $\Cat_{\infty}$ of  small    $\infty$-categories admitting countable filtered colimits
and countable  filtered colimit  and shape   preserving functors and
$\Cat_{\infty}^{\casmbl}$ be {the}  full subcategory of $\Cat_{\infty}^{\cfil,\shp}$ of  \pcas{} categories. 
Similarly, we let   $\Cat_{\infty}^{\exa,\cfil,\shp}$  be the subcategory of $\Cat^{\exa}_{\infty}$ of  small  countably cocomplete  stable $\infty$-categories 
and exact, countable filtered colimit  and shape   preserving functors and
$\Cat^{\exa,\casmbl}_{\infty}$ be  {its full subcategory} 
of      \pcas{} categories. 

\begin{prop}\label{wregiojwoeprgerwfrewfwref} We have right Bousfield localizations 
$$\incl:\Cat_{\infty}^{\casmbl}\leftrightarrows  \Cat_{\infty}^{\cfil,\shp}:(-)^{\casmbl}$$
and 
$$\incl:\Cat_{\infty}^{\exa,\casmbl}\leftrightarrows  \Cat_{\infty}^{\exa,\cfil,\shp}:(-)^{\casmbl}\ .$$
\end{prop}
\begin{proof}
We first show that the construction $\bD\mapsto \bD^{\casmbl}$ extends to a functor.
In the  statement  below we use the following notation: If $\Phi:\cC\to \cD$ is a functor between $\infty$-categories and $\cD'$ is a full subcategory of $\cD$, then we write $\Phi(\cC)\subseteq \cD'$ for the assertion that the essential image of $\Phi$ is contained in $\cD'$.
  \begin{lem}\label{tgijoworegvferwfrewfw}
 If $F:\bC\to \bD$ preserves  shapes, then $F(\bC^{\casmbl})\subseteq  \bD^{\casmbl}$. 
   \end{lem}
 \begin{proof}
 We argue  {by transfinite induction that}  $F(\bC^{(i)})\subseteq \bD^{(i)}$.
For $i=0$ the assertion is true by assumption.
{Let $i$ be an ordinal.
\begin{enumerate}
\item  If $i$ is a successor ordinal, then $i=j+1$ for some ordinal $j$.  Then we have $F(\bC^{(j)})\subseteq \bD^{(j)}$ by induction hypothesis. Let $A$ be in $\bC^{(i)}$ so that $S_{\bC}(A)\in \Ind^{\aleph_1}(\bC^{(j)})$.
 Then $S_{\bD}(F(A))\simeq \hat F(S_{\bC}(A))\in \hat F(\Ind^{\aleph_{1}} (\bC^{(j)} ) )\subseteq \Ind^{\aleph_{1}} (\bD^{(j)} ) $.
 Hence $F(A)\in \bD^{(i)}$.
 \item If $i$ is a limit ordinal, then  clearly
 $$F(\bC^{(i)})=F(\bigcap_{j<i} \bC^{(j)})\subseteq \bigcap_{{j<i}} F(\bC^{(j)})\subseteq \bigcap_{{j <i}} \bD^{(j)}=\bD^{(i)}\ .$$
\end{enumerate}
}
%
 \end{proof}

In the following statement the superscripts $\exa$, $\cfil$, $ \shp$ stand for  exact, 
countable filtered colimit,  and shape   preserving functors respectively.
 Note that the inclusion $\bD^{\casmbl}\to \bD$  preserves countable filtered colimits and shapes
(and also finite colimits and limits in the stable case).
The following result completes the proof of \cref{wregiojwoeprgerwfrewfwref}.

\begin{lem} If $\bC$ is  \pcas{} and $\bD$ is an $\infty$-category admitting countable filtered colimits, then
 the inclusion $\bD^{\casmbl}\to \bD$ induces an equivalence 
$$\Fun^{\cfil,\shp}(\bC,\bD^{\casmbl})\stackrel{\simeq}{\to}\Fun^{\cfil,\shp}(\bC,\bD). $$
If $\bC$ and $\bD$ are in addition stable, then  the inclusion also induces an equivalence  
$$\Fun^{\exa,\cfil,\shp}(\bC,\bD^{\casmbl})\stackrel{\simeq}{\to}\Fun^{\exa,\cfil,\shp}(\bC,\bD) $$

\end{lem}
\begin{proof}
Let $F:\bC\to \bD$ be a  countable colimit and shape preserving functor.  We must show that 
$F$ takes values in $\bD^{\casmbl}$. But since $\bC^{\casmbl}=\bC$ this immediately follows from \cref{tgijoworegvferwfrewfw}.
The argument in the stable case is analoguous.
\end{proof}\end{proof}

\subsection{Compact approximations and exhaustions}\label{erwgjoiwpergwerfrefrewfw}

We consider an $\infty$-category $\bD$ admitting countable filtered colimits.  
In this section we characterize the existence of a shape  of an object $A$ of $\bD$ in terms of the existence of suitable exhaustions of $A$.  Conversely we discuss the consequences of the assumption that $\bD$ is \pcas{} to the existence of exhaustions of its objects.

  In the following we will say that a system  $(A_{n})_{n\in \nat}$ in $\bD$ is an approximation  of $A:=\colim_{n\in \nat }A_{n}$.  
 Let $?$ be in \{weakly, slp, strongly\}.

  \begin{ddd}\label{jigowergwerfewrfwerf}\mbox{}
\begin{enumerate}
\item A ?-compact approximation of $A$  in $\bD$ is an approximation   $(A_{n})_{n\in \nat}$ such that   the structure maps $A_{i}\to A$ are ?-compact for all $n$ in $\nat$.
\item A ?-compact exhaustion   of $A$  is an  approximation $(A_{n})_{n\in \nat}$ such that 
 the structure maps $A_{n}\to A_{n+1}$ are ?-compact  for all $n$ in $\nat$.   
\end{enumerate}\end{ddd}
Note that a ?-compact exhaustion is a ?-compact approximation.

 The following result gives a criterion when an approximation  represents the shape of an object $A$ of $\bD$.
 
\begin{prop} \label{qrjigowergwergrwewf}
The object $A$ admits a shape if and only if it admits an approximation $(A_{n})_{n\in \nat}$ such that the  canonical map
  \begin{equation}\label{adsvasdccdcaseccacsdc}\lim_{n\in \nat} \colim_{i\in I}\Map_{ \bD}(A_{n},Z_{i}) \to \Map_{\bD}(A,\colim_{i\in I}Z_{i})\end{equation} is an equivalence {for all systems $(Z_i)_{i \in I}$ in $ \bD $}. In this case $S(A)\simeq \colim_{n\in \nat} y(A_{n})$.
 \end{prop}
\begin{proof}
{We first assume that  $A$ admits a shape $(S(A),\eta_{A} )$. By  \Cref{jwoirthopgfrfwrgre} 
the map $\eta_{A}:|S(A)|\to A$ is an equivalence.  There exists   a system  $(A_{n})_{n\in \nat} $ in $\bD$ with \begin{equation}\label{gwerfwrefwrfrwerr}\colim_{n\in \nat} y(A_{n})\simeq S(A)\ .\end{equation} 
We apply $|-|$ and compose with $\eta_{A}$ in order to get an equivalence
$\colim_{n\in \nat}A_{n}\simeq A$, i.e, $(A_{n})_{n\in \nat}$ is an approximation of $A$.
Finally  \eqref{adsvasdccdcaseccacsdc} 
is the composition of the two equivalences 
  \begin{equation}\label{ehpkortzherthtrg} 
	\lim_{n \in \nat} \colim_{i \in I} \Map_{\bD} (A_n, Z_i) \simeq \Map_{\Ind^{\aleph_1}(\bD)}(\colim_{n \in \nat} y(A_n), \colim_{i \in I} y(Z_i)) \xrightarrow{\eta_{A}^{*}\circ\lvert-\rvert  } \Map_{\bD}(A,  \colim_{i \in I} Z_i)\ .
\end{equation}}For  the converse we assume that $(A_{n})_{n\in \nat}$ is an approximation of $A$ such that  \eqref{adsvasdccdcaseccacsdc} is an equivalence.
Then we set $S(A):=\colim_{n\in \nat} y(A_{n})  $
and define the equivalence $\eta_{A}: |S(A)|\simeq \colim_{n\in \nat}A_{n} \stackrel{\simeq}{\to}  A$.
We conclude that the second map in \eqref{ehpkortzherthtrg}  is an equivalence showing that $(S(A),\eta_{A})$ is a shape of $A$.
\end{proof}
 
 Let $(A_{n})_{n\in \nat}$ be an approximation of $A$. 
 
 \begin{kor}\label{wrijotgwegrerfwerf9} If 
   $\colim_{n\in \nat} y(A_{n})$ represents the shape of $A$, then $(A_{n})_{n\in \nat}$  is a weakly compact approximation.
\end{kor}
\begin{proof}   We  apply  $\pi_{0}$ to the equivalence \eqref{adsvasdccdcaseccacsdc}.   \end{proof}

The following theorem shows that the  converse  of \cref{wrijotgwegrerfwerf9}  holds under {a} stronger assumption.    

 \begin{theorem}\label{wkotpgsfvsfgfdgg}
If $(A_{n})_{n\in \nat}$ is a slp-compact exhaustion of  $A$, then $A$ admits a shape and $S(A)\simeq \colim_{n\in \nat}
y(A_{n})$. 
\end{theorem}
\begin{proof} We use the criterion \cref{qrjigowergwergrwewf}.
We consider a system $(Z_{i})_{i\in I}$ and set $Z:=\colim_{i\in I} Z_{i}$. By assumption the map of systems
$$(\colim_{i\in I}\Map_{ \bD}(A_{n},Z_{i}))_{n\in \nat^{\op}}\to (\Map_{\bD}(A_{n},Z))_{n\in \nat^{\op}}\ .$$
has the slp in the sense of \cref{qifjgioerfwefdqwefqwef}. 
Hence it follows from  \cref{bidfosdfivosdifhgnbeoiosdf} {below}
 that
  \begin{equation}\label{adsvasdccdcaseertertccacsdc}\lim_{n\in \nat} \colim_{i\in I}\Map_{ \bD}(A_{n},Z_{i}) \to\lim_{n\in \nat} \Map_{\bD}(A_{n},\colim_{i\in I}Z_{i})\end{equation}
  is an equivalence. 
  \end{proof}
  
  \begin{kor}\label{gkopwerfweferwfw}
If every object of $\bD$ admits an slp-compact exhaustion, then $\bD$ is \pcas.
\end{kor}

    \begin{theorem}\label{regijfowergffreferwfw} If $\bD$ is \pcas, then
 every object of $\bD$ admits a strongly compact exhaustion.
 \end{theorem}
 {
 	 \begin{proof}
 		We fix $A$   in $\bD$. Then there exists a system $(A_{n})_{n\in  \nat}$ in $\bD$ 
 		such that  $$\colim_{n\in \nat} y(A_{n})\simeq S(A)\ .$$ Let $\epsilon: S \circ \lvert - \rvert \to \id_{\bD}$ be the counit of the adjunction \eqref{fwqeddcvsdvcs1}. 
		Since $S$ is a right Bousfield localization the map 
		$ \epsilon_{S (\lvert y (A) \rvert)}: S ( \lvert  S ( \lvert y(A) \rvert )\rvert )\to S  (\lvert y(A) \rvert)  $ is an equivalence. Since the target and domain of $\epsilon $ commute with countable filtered colimits $\epsilon_{S (\lvert y( A) \rvert)}$ is  the colimit of the system of maps $ (\epsilon_{y(A_n)} :S(\lvert y(A_n)  {\rvert)} \to y(A_n))_{n \in \nat} $. The commutative diagram
 		$$\hspace{-1cm}\xymatrix{\Map_{\Ind^{\aleph_{1}}(\bD)}(S(A),S  (\lvert  S(  \lvert y(A) \rvert) \rvert)) \ar[rr]^-{\epsilon_{{S( \lvert y (A) \rvert)},\ast}}_-{\simeq} \ar[d]^{\simeq} && \Map_{\Ind^{\aleph_{1}}(\bD)}(S(A),S ( \lvert y(A) \rvert)) \ar[d]^{\simeq} \\ \lim_{n \in \nat}\colim_{k\in \nat}\Map_{\Ind^{\aleph_{1}}(\bD)}(y(A_n),S(|y(A_k)|))\ar[rr]^{\colim_k \epsilon_{y(A_k), \ast}} && \lim_{n \in \nat}\colim_{k\in \nat}\Map_{\Ind^{\aleph_{1}}(\bD)}(y(A_n),y(A_k))  } \ .$$
 		implies that the lower horizontal map is an equivalence.
		We interpret the family  $(\id_{y(A_n)})_{n}$  as an element of the right lower corner.
	It has  a pre-image	
		 $(f_{n})_{n}$ in the left lower corner.  Then $f_{n} $ is a preimage of $\id_{y(A_n)}$ under 
 		$$ \colim_{k{\in \nat} } \epsilon_{y(A_k), \ast}: \colim_{k\in \nat}\Map_{\Ind^{\aleph_{1}}(\bD)}(y(A_n),S(|y(A_k)
		|))\to\colim_{k\in \nat}  \Map_{\Ind^{\aleph_{1}}(\bD)}(y(A_n),y(A_k))\ . $$
		 It induces a  factorization
 		$$ \xymatrix{& S(|y(A_{k(n)})|) \ar[dr]^{\epsilon_{y(A_{k(n)})}}\\ y(A_n) \ar@{-->}[ur] \ar[rr] && y(A_{k(n)})  }  $$ for some  $k(n) \geq n+1$.  Therefore for all $n  $ in $ \nat$ the structure map $A_n \to A_{k(n)}$ is strongly compact by \cref{wegjiopewgfsg} and \eqref{sfdvojkogpvvsfdvsdvsdfv}.
		We can conclude that a cofinal subsystem of the system $(A_{n})_{n\in \nat}$ is a  strongly compact exhaustion of $A$.
 	\end{proof}
}

 {
	In the following we clarify the relation between weakly compact maps and slp compact maps.
	\begin{prop}\label{vniowsdfiojgdf}
		If $\bD$ admits finite limits and filtered colimits in $\bD$ commute with finite limits, then weakly compact maps are slp compact.
	\end{prop}
	\begin{proof}
		We start with a general criterion for a square of anima to be liftable: The solid part of a square of anima
		$$ \xymatrix{X \ar[r] & Y  \\ K \ar[u] \ar[r] & L \ar[u] \ar@{-->}[ul] } $$
		induces a point $\ast \to X^K \times_{Y^K} Y^L$ and a canonical comparison map $X^L \to X^K \times_{Y^K} Y^L$. The above diagram admits a dashed lift if and only if the solid diagram below admits a dashed lift
		$$ \xymatrix{& X^L \ar[d]\\ \ast \ar@{-->}[ru] \ar[r] & X^K \times_{Y^K} Y^L } . $$
		Let $f: B \to A$ be a weakly compact map in $\bD$ and $(Z_i)_{i \in I}$ a system   with colimit Z. We have to solve the lifting problem
		$$ \xymatrix{\colim_{i \in I} \Map_{\bD}(B, Z_i) \ar[r] & \Map_{\bD}(B,Z) \\ \colim_{i \in I} \Map_{\bD}(A,Z_i) \ar[r] \ar[u]^{f^\ast} & \Map_{\bD}(A, Z) \ar[u]^{f^{\ast}} \\ K \ar[r] \ar[u] & L \ar@{-->}[uul] \ar[u]}. $$
		for all finite anima $K $ and $L$. Due to the above criterion and the finiteness of $K$ and $L$, this is equivalent to solving the lifting problem
		\begin{equation}\label{grewwrefewrfrewfrfw2} \xymatrix{& & \colim_{i \in I} \Map_{\bD}(B, Z_i^L) \ar[d]\\ \ast \ar@{-->}[rru] \ar[r] & \colim_{i \in I}\Map_{\bD}(A, Z_i^K \times_{Z^K} Z^L) \ar[r]  & \colim_{i \in I}\Map_{\bD}(B, Z_i^K \times_{Z^K} Z^L) } . \end{equation} 
		In order to obtain such a lift we consider the diagram
		\begin{equation}\label{grewwrefewrfrewfrfw} \xymatrix{B \ar[d]^f \ar@{-->}[r] & Z_i^K \times_{Z_j^K} Z_j^L \ar[d] \\ A \ar[r]^-{!} & Z_i^K \times_{Z^K} Z^L } \ , 
\end{equation}  		where the marked map is induced by the given point in $\colim_{i \in I} \Map_{\bD}(A, Z_i^K \times_{Z^K} Z^L)$. We can now use  the weak compactness of $f$ and $$ \colim_{(i \to j) \in I_{i/}} Z_i^K  \times_{Z_j^K} Z_j^L \simeq Z_i^K \times_{Z^K} Z^L   $$
  in order to obtain the dashed  arrow in \eqref{grewwrefewrfrewfrfw}.		 		The composite
		$$\xymatrix{B\ar@{-->}[r]& Z_i^K \times_{Z_j^K} Z_j^L \ar[r]^-{\pr_{Z_j^L}} &Z_j^L} $$
		interpreted as a point in $ \colim_{i \in I}\Map_{\bD}(B, Z_i^L) $ provides an appropriate lift in \eqref{grewwrefewrfrewfrfw2}.
	\end{proof}
This provides another argument for a Theorem of Clausen and Lurie. 
\begin{theorem}[Clausen, Lurie, see {\cite[Thm. 2.2.11, Lemma 2.3.13]{som}}] \label{vnaefnauen}
	We assume that $\bD$ is   left exact, admits countable filtered colimits,  and filtered colimits in $\bD$ commute with finite limits. If $(A_n)_{n \in \nat}$ is a weakly compact exhaustion of $A$, then   $\colim_{ n \in \nat} y(A_n)$ represents the shape of $A$.
\end{theorem}
\begin{proof}
	This is a consequence of \cref{vniowsdfiojgdf} and \cref{wkotpgsfvsfgfdgg}.
\end{proof}
}

\subsection{slp maps between towers of anima} \label{qifjgioerfwefdqwefqwef}
The main result of this technical section is  \cref{bidfosdfivosdifhgnbeoiosdf} which goes into the proofs of \cref{wkotpgsfvsfgfdgg} and \cref{wreig90wregfwerf}.
\begin{ddd} \label{gwjeiorrfrefrfw}A map   $(X_n)_{n \in \nat^{\op}} \to (Y_n)_{n \in \nat^{\op}}$ in  $\Fun(\nat^{\op}, \Spc)$
 has the slp ({shifted lifting property})   if for every $n$ in $\nat$ the commutative  square
$$ \xymatrix{X_n \ar[r] & Y_n \\ X_{n+1} \ar[r]  \ar[u] & Y_{n+1} \ar[u]}$$ 
has the slp in the sense of  \cref{kotgpergoeprtkhgerth}.\end{ddd}
 \begin{theorem} \label{bidfosdfivosdifhgnbeoiosdf}
	If a map $f:(X_n)_{n \in \nat^{\op}} \to (Y_n)_{n \in \nat^{\op}}$ in $\Fun(\nat^{\op}, \Spc)$ satisfies the slp, then the induced map
	$$ \lim_{n \in \nat^{\op}} X_n \to \lim_{n \in \nat^{\op}} Y_n $$
	is an equivalence.
\end{theorem}
\begin{proof}
The idea is to show that the map induces a bijection on the level of $\pi_{0}$ and isomorphisms
on homotopy groups for all choices of base points.  For the argument we need some closure properties 
of the class of squares with the slp. As a preparation for their derivation 
  we 
 consider an adjunction \begin{equation}\label{vsdfvberfvvsdfv}  L: \bA \rightleftarrows \bB: R\end{equation}  between $\infty$-categories. Furthermore let  $f: A \to A^\prime$ be a map in $\bA$ and $B$ be an object in $\bB$.
\begin{lem} \label{pqtgvbnuklnkloojh}
	 The following are equivalent:
	\begin{enumerate}
		\item  \label{qjieorgwergwegwref}	The diagram $$ \xymatrix{& R(B) \\ A \ar[ur]^g \ar[rr]^f && A^\prime \ar@{-->}[lu]} $$ admits a dashed completion.
		\item\label{qjieorgwergwegwref1} The diagram $$ \xymatrix{& B \\ L(A) \ar[ur]^{counit \circ L(g)} \ar[rr]^{L(f)} && L(A^\prime) \ar@{-->}[lu]} $$ admits a dashed completion.
	\end{enumerate}
\end{lem}
\begin{proof}
	We assume Assertion \ref{qjieorgwergwegwref} and get the triangle from Assertion \ref{qjieorgwergwegwref1} by applying $L$ and post-composing the counit:
	$$ 
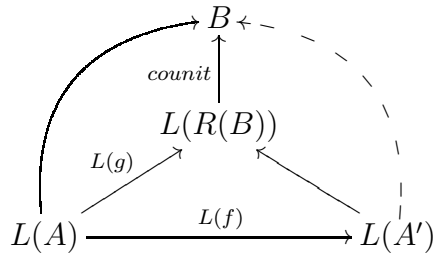
\xymatrix{& B&\\&L(R(B))\ar[u]^{counit}& \\ L(A) \ar[ur]^{ L(g)} \ar[rr]^{L(f)} \ar@/^1cm/[uur]&& L(A^\prime) \ar@/^-1cm/@{-->}[luu]\ar[ul]}\ .$$
	We now assume Assertion \ref{qjieorgwergwegwref1} 
	and apply  $R$ in order to get the desired triangle from  Assertion \ref{qjieorgwergwegwref}:
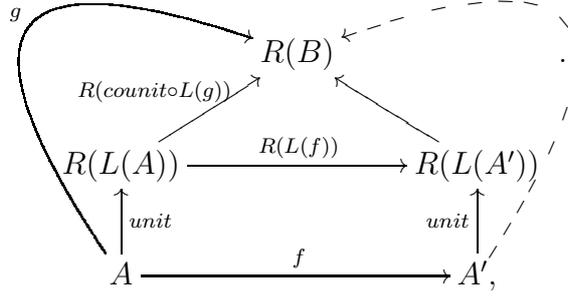
	$$  \xymatrix{& R(B) \\ R(L(A)) \ar[ur]^{R(counit \circ L(g))} \ar[rr]^{R(L(f))} && R(L(A^\prime)) \ar[lu] \\ A \ar@/^3cm/[uur]^g\ar[u]_{{unit}} \ar[rr]^{f} & &\ar@{-->}@/^-3cm/[uul] A^\prime, \ar[u]^{{
				unit}}}  \ .$$
	The left deformed triangle   commutes due to the triangle identities for the adjunction \eqref{vsdfvberfvvsdfv}.
\end{proof}
\begin{lem}
	Let 
	$$ \xymatrix{ X \ar[r]  & Y \\ X^\prime \ar[u] \ar[r] & Y^\prime \ar[u]^f \\ K \ar[r] \ar[u] & L \ar[u] } $$
	be a commutative diagram of anima such that the upper square is cartesian and the outer square is liftable. Then also the lower square
	is liftable.
\end{lem}
\begin{proof}
	Using \eqref{puighygbqrjkb} it suffices to show that there exists a commutative diagram
	$$ 
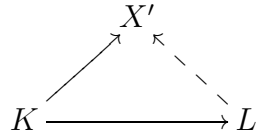
\xymatrix{& X^\prime \\ K \ar[ru] \ar[rr] & & L \ar@{-->}[lu]} $$
	in $\Spc_{/Y^{\prime}}$ (we drop the structure maps to $Y^{\prime}$). There exists an adjunction
	$$ f_\sharp: \Spc_{/Y^\prime} \leftrightarrows \Spc_{/Y} :f^\ast, $$
	where $f_\sharp$ is given by postcomposition with $f$ and $f^\ast$ is given by pullback along $f$.  We have an equivalence  $f^\ast X   \simeq X'$. By \cref{pqtgvbnuklnkloojh}  it suffices to solve the lifting problem
	$$ 
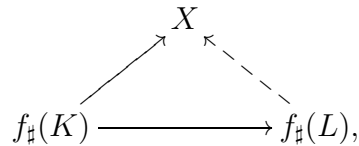
\xymatrix{&X \\ f_\sharp (K) \ar[ru] \ar[rr] & & f_\sharp(L), \ar@{-->}[lu]} $$
	and this   is possible by our assumption.
\end{proof}
For the following corollary it is convenient to identify a commutative square
$$ \xymatrix{X_0 \ar[r]^{f_0} & Y_0 \\ X_1 \ar[r]^{f_1} \ar[u] & Y_1 \ar[u]}$$
with a map $f: X \to Y$ in $\Fun(\Delta^{1,\op}, \Spc)$. We call $f$ liftable, if the associated square is liftable.
\begin{kor}\label{siubdsfvioasdnf}\mbox{}
	\begin{enumerate}
		\item The class of morphisms  in $\Fun(\Delta^{1,\op}, \Spc)$ satisfying the slp  is closed under pullbacks.
		\item \label{wekogpwegwerfwerfwerfw}The class of morphisms in $\Fun(\nat^{\op}, \Spc)$ satisfying the slp is closed under pullbacks.
	\end{enumerate}
\end{kor}

We consider  the adjunction
\begin{equation}\label{fqwqwojfpwedqwedqd} \Sigma_+: \Spc \rightleftarrows \Spc_{\ast/}: \Omega\ .\end{equation}  Let $f:(X,x) \to (Y,y)$ be a map of pointed anima.
\begin{lem}
	  The square
	\begin{equation}\label{fqwefwqeojfkqpwdqwdqwd}  \xymatrix{\Omega (X,x) \ar[r]^{\Omega f} & \Omega (Y,y) \\ K \ar[r] \ar[u] & L\ar[u]}  \end{equation}  is liftable if and only if the associated square
	$$ \xymatrix{X \ar[r]^{f} & Y \\ \Sigma_+ K \ar[r] \ar[u] & \Sigma_+ L\ar[u]}$$
	 is liftable.
\end{lem}
\begin{proof}
	We have an adjunction
	$$ (\Sigma_+)_{/{(Y,y)}}: \Spc_{/\Omega(Y,y)} \rightleftarrows (\Spc_{\ast})_{ /(Y,y)}: \Omega_{/(Y,y)}. $$ 	Here the non-obvious functor $ (\Sigma_+)_{/(Y,y) }$ is given by the composite
	$$ \Spc_{/\Omega(Y,y)}\xrightarrow{(\Sigma_+)_{/ \Omega(Y,y)}} (\Spc_{\ast})_{ /\Sigma_+( \Omega(Y,y))}\xrightarrow{counit_\sharp}( \Spc_{\ast})_{/(Y,y)}, $$
	where $counit_\sharp$ is postcomposition with the counit $\Sigma_+( \Omega(Y,y))\to (Y,y)$ of the adjunction \eqref{fqwqwojfpwedqwedqd}. In view of 
	\cref{pqtgvbnuklnkloojh} the square \eqref{fqwefwqeojfkqpwdqwdqwd} is liftable if and only if  the dashed completion 	$$ \xymatrix{& (X,x) \\ \Sigma_+ K \ar[ru] \ar[rr] & & \Sigma_+ L \ar@{-->}[lu]} $$
	exists in $(\Spc_{\ast})_{/(Y,y)}$. Such a completion exists if and only if the underlying diagram in $\Spc_{/Y}$ admits the corresponding dashed lift.
\end{proof}

We consider a  commutative  square \begin{equation}\label{fwqedqwdeqe} \xymatrix{(X_0,x_0) \ar[r] & (Y_0,y_0) \\ (X_1,x_1) \ar[r] \ar[u] & (Y_1,y_1) \ar[u]}   \end{equation} 
		  in $\Spc_{\ast/}$.
\begin{kor}\label{eiogqubiuosdhbafuiopyib}\mbox{}
	\begin{enumerate}
		\item  {If the underlying square of \eqref{fwqedqwdeqe} in $\Spc$ has the}
		  slp, then also
		$$ \xymatrix{\Omega  (X_0,x_{0}) \ar[r] & \Omega ( Y_0 ,y_{0})\\ \Omega ( X_1 ,x_{1})\ar[r] \ar[u] & \Omega ( Y_1 ,y_{1})\ar[u]}  $$
		 {has the} slp.
		\item \label{gjwergoiwerfwerfwerf}	If $f:(X_n)_{n \in \nat^{\op}} \to (Y_n)_{n \in \nat^{\op}}$  {has the} slp, then   for all compatible choices of base points $ x: {(x_{n})_{n\in \nat} }\to (X_n)_{n \in \nat^{\op}}, $ the induced map $$ ( \Omega (X_n,x_{n}))_{n \in \nat^{\op}} \to (\Omega  (Y_n,f(x_{n})))_{n \in \nat^{\op}} $$
		  {also has the} slp.
	\end{enumerate}

\end{kor}

\begin{rem}\label{ojiregofewrgreferwf} In this remark
we  recall the properties of the $\lim^1$-sequence. 
\begin{enumerate}
	\item  For every tower $(X_n)_{n \in \nat^{\op}}$ and a compatible system of base points $x:{(x_{n})_{n\in \nat} } \to (X_n)_{n \in \nat}$ there exists an exact sequence
\begin{equation}\label{eqwfqwedqe}  \ast \to {\lim_{n \in \nat^{\op}}}^1 \pi_{1}(X_n, x_n) \to \pi_0 (\lim_{n \in \nat^{\op}} X_n, {\lim_{n \in \nat^{\op}}x_{n}}) \to \lim_{n \in \nat^{\op}} \pi_0(X_n,x_n) \to \ast  \end{equation}	 
	of pointed sets. 
	\item \label{koerpgerwfqwefqwdweqd99} The functor ${\lim}^1$ admits a factorization
	$$ \xymatrix{\Fun(\nat^{\op}, \Group) \ar[rr]^{{\lim}^1} \ar[rd] && \Set \\ & \Pro(\Group) \ar@{-->}[ru]} \ .$$
	This follows from the formula ${\lim_{n \in \nat^{\op}}}^1 G_n \cong  \pi_0(\lim_{n \in \nat^{\op}} B G_n )$
	{for any system $(G_{n})_{n\in \nat^{\op}}$ in $\Fun(\nat^{\op},  \Group)  $}.
	\item \label{qerkogpfweqedwed}For any tower $(X_n)_{n \in \nat^{\op}}$ {in $\Spc$} the map $\pi_{0}(\lim_{n \in \nat^{\op}} X_n) \to \lim_{n \in \nat^{\op}}\pi_0(X_n) $ is surjective.
\end{enumerate}
\end{rem}
 
	We start with a computation on $\pi_0$. 
	
\begin{lem}\label{wjeitgowgwerfewrfwer99}
	If $(X_n)_{n \in \nat^{\op}} \to \underline{\ast}$ {has the} slp, then $(\pi_0(X_n))_{n \in \nat^{\op}}$ is isomorphic to $ { \underline{\ast}}$ as a pro-system.
	\end{lem}
	\begin{proof} First note that $ \pi_0(X_n)  $ is non-empty for all $n$ in $ \nat$ as the diagram
	$$ \xymatrix{X_n \ar[r]  & \ast \\ X_{n+1} \ar[u] \ar[r] & \ast \ar[u] \\ \emptyset \ar[r] \ar[u] &  \ast \ar[u] \ar@{-->}[luu]}  $$
	admits a dashed lift. Secondly, the map $\pi_0(X_{n+1}) \to \pi_0(X_n)$ is constant since for every map $S^0 \to X_{n+1}$ there exists a dashed lift
	$$ \xymatrix{X_n \ar[r]  & \ast \\ X_{n+1} \ar[u] \ar[r] & \ast \ar[u] \\ S^0 \ar[r] \ar[u] &  \ast. \ar[u] \ar@{-->}[luu]}  $$
	\end{proof}

	We can now start the actual proof of \cref{bidfosdfivosdifhgnbeoiosdf}.
	Let $(X_n)_{n \in \nat^{\op}} \to (Y_n)_{n \in \nat^{\op}}$ {have the} slp. Since slp maps are stable under pullback   by \cref{siubdsfvioasdnf}.\ref{wekogpwegwerfwerfwerfw} and an equivalence of anima is a map with trivial fibres  {at all base points}, we can reduce to the case where $(Y_n)_{n \in \nat^{\op}}  \simeq  \underline{\ast}$. As slp maps are stable under loops with respect to each compatible system of base points  by  {\cref{eiogqubiuosdhbafuiopyib}.\ref{gjwergoiwerfwerfwerf}}, it suffices to check that $\pi_0(\lim_{n \in \nat^{\op}} X_n) \cong  \ast$. Since  $\lim_{n \in \nat^{\op}} \pi_0(X_n) \cong *$  {by  \cref{wjeitgowgwerfewrfwer99}}, we know by  \cref{ojiregofewrgreferwf}.\ref{qerkogpfweqedwed} that $\pi_{0}(\lim_{n \in \nat^{\op}} X_n)$ is non-empty. By the ${\lim}^1$-sequence \eqref{eqwfqwedqe}  it suffices to check that for all compatible families of base points $x:  (x_{n})_{n\in \nat^{\op}} \to (X_n)_{n \in \nat^{\op}}$ we have ${\lim_{n \in \nat^{\op}}}^1 \pi_0( \Omega ( X_n,x_{n})) \cong  \ast$. This is true  {by \cref{ojiregofewrgreferwf}.\ref{koerpgerwfqwefqwdweqd99} since}    $(\Omega  (X_n,x_{n}))_{n \in \nat^{\op}}$  {has the} slp by \cref{eiogqubiuosdhbafuiopyib}.\ref{gjwergoiwerfwerfwerf} and therefore $(\pi_0(\Omega  (X_n,x_{n})))_{n \in \nat^{\op}} \simeq \underline{*}$ in $\Pro(\Group)$,  {again by   \cref{wjeitgowgwerfewrfwer99}}. 
\end{proof}

\subsection{Phantoms}\label{weokgpwegwerfwfrfrefwerf}


In this section we introduce the notion of a strong phantom morphism and analyse some of its properties and how they interact with the condition on the $\infty$-category of being \pcas. 
{For a different treatment of phantom maps see \cite[Appendix E]{Efimov_2024}}.
For the general definitions it suffices to  assume that $\bD$ is a pointed   $\infty$-category which admits all countable filtered colimits. Later we will add the assumption of stability.

\begin{ddd}\label{ijwqfofjewfqedfq9}
A morphism $ g:B\to C$  in $\bD$ is called  a strong  phantom morphism
if the composition $A\xrightarrow{f} B\xrightarrow{g}C$ vanishes for all weakly compact maps $f$.
 
A strong phantom object is  an object whose identity is a  strong phantom morphism.
\end{ddd}

Since the weakly compact maps {form} a two-sided ideal in $\bD$ also the strong phantom maps form   a two-sided ideal.

\begin{rem}
Classically a  morphism  $B\to C$ is called a phantom map if the composition $A\to B\to C$ vanishes for  {any}  map from a compact object $A$. A phantom object is an object whose identity morphism is a phantom map. If $\bD$
is compactly generated, then $\bD$ does  not have non-trivial phantom objects, and being a phantom morphism
is a restrictive condition. In the other extreme, if $\bD$ has only few (or even no) non-trivial compact objects, then
being a phantom morphism or object is not  very (or not at all) restrictive. We use the adjective strong in order to 
distinguish the stronger conditions introduced in \cref{ijwqfofjewfqedfq9} from the classical ones.

In \cref{ijwqfofjewfqedfq9} we defined {the} set of strong phantom morphisms as the  left annihilator of the set of weakly compact morphisms.
In principle one could also define bigger sets of phantoms as  left annihilators of the possibly smaller sets of slp compact  or strongly compact morphisms.  We will not consider these variants further, but note that
in a \pcas{} category all these variants coincide anyway by
\cref{kgoprgfeegferfwef}
\hB
\end{rem}

 \begin{rem}
{We consider a morphism  $f:B\to C$ in $\bD$. Following   \cite{Efimov_2024} it is called a pure monomorphism
if the fibre $A\to B$ of $f$ exists and is a strong phantom morphism.
Similarly, it is called a pure epimorphism, if the cofibre $C\to D$ exists and is a  strong phantom morphism.}

{If $\bD$ is stable, then in \cite{Bunke:2023ab} a fibre sequence $A\to B\to C$
was called a phantom extension if the map $B\to C$ is a pure epimorphism. 
For example, the projection map of a weakly quasi-diagonal extension of $C^{*}$-algebras induces a pure epimorphism in $E$-theory, see   \cref{ogjwerferwferwfwe} for details. Therefore weakly quasi-diagonal extensions of $C^{*}$-algebras become   
phantom extensions in $E$-theory. For further applications of phantoms in $C^{*}$-algebra theory we refer to \cite{Bunke:2023ab}.}

{Still assuming that $\bD$ is a stable, a morphism $A\to B$ is called a phantom equivalence if its cofibre (or equivalently its fibre) is a strong phantom object. For applications of this notion to the   Farrell-Jones conjecture or the coarse Baum-Connes conjecture  see  \cite{Bunke:2021uh}  or
\cite{Bunke:2024aavvvv}.
 \hB} \end{rem}

 In order to check whether a morphism $\phi:B\to C$ is a strong phantom it suffices to consider a weakly compact
 approximation. 
\begin{lem}\label{wekgowpefefwrf} 
If $(B_{n})_{n\in \nat}$ is a {weakly} compact approximation for $B$, then a morphism $\phi:B\to C$ is a strong phantom map if and only if the compositions $B_{n}\to B\xrightarrow{\phi} C $ vanish for all $n$ in $\nat$.
\end{lem}
\begin{proof}
If $\phi$ is a strong phantom map, then the compositions  $B_{n}\to B\xrightarrow{\phi} C $  clearly vanish.
For the converse we consider a weakly compact map $f:A\to B$ and must show that $\phi\circ f\simeq 0$. In view of \cref{jowtogpwgerferfwefewfref} it has a factorization
$A\to B_{n}\to B$ for some $n$ in $\nat$. Hence the composition $A\xrightarrow{f} B\xrightarrow{\phi} C$ has a factorization
$A\to B_{n}\to B\xrightarrow{\phi} C$ and therefore vanishes.
\end{proof}

From now one we assume that $\bD$ is stable and countably cocomplete.

\begin{ddd}\label{erijogoewrpgfrefw1}\mbox{}\begin{enumerate}
\item\label{erijogoewrpgfrefw} A sequence $(g_{n})_{n\in \nat}$ of morphisms $g_{n}:B\to C$ is called summable if 
there exists a morphism $\hat g:B\to \bigoplus_{\nat}C$ with the components $(g_{n})_{n\in \nat}$.
\item \label{ikwjdbnfjdjkidhfiwejfg} A morphism $g:B\to C$ is called a summable phantom map if the constant sequence $(g)_{n\in \nat}$  is summable. \end{enumerate}
\end{ddd}
In \cref{erijogoewrpgfrefw} we say that $\hat g$ witnesses the summability of the sequence, {and in \cref{ikwjdbnfjdjkidhfiwejfg} we also say that $g$ is summable.}
{\begin{rem} Assume that $(g_{n})_{n\in \nat}$ is a summable sequence of maps $g_{n}:B\to C$
with summability witnessed by $\hat g:B\to \bigoplus_{n\in \nat}C$. Then we can define the sum
 of the family as the composition
\begin{equation}\label{qererqr}{}^{\hat g}\sum_{n\in \nat} g_{n}:B\xrightarrow{\hat g}  \bigoplus_{n\in \nat}C\xrightarrow{\mathrm{fold}}C\ .
\end{equation} 
We added the superscript $\hat g$ in order to indicate the possible dependence of the sum on the choice of the witness.
{One can check that  the sum of a summable sequence is well-defined up to strong phantoms.}
 \hB
\end{rem}}

\begin{lem}\label{regjiewerggre9}
A summable map is a strong phantom map.
\end{lem}
\begin{proof}
Let $g:B\to C$ be a summable map with summability  witnessed by $\hat g$.
We consider any weakly compact map  $f:A\to B$.  By \cref{jowtogpwgerferfwefewfref}  we get a factorization of the composition $A\xrightarrow{f}B\xrightarrow{\hat g}\bigoplus_{\nat}C$
over  a map $A\to \bigoplus_{i=0}^{k}C$ for some $k$. This implies that the composition $g\circ f:A\to B\xrightarrow{\hat g} \bigoplus_{\nat}C\xrightarrow{\pr_{n}}C$ vanishes for $n>k$.    
 \end{proof}
\begin{rem}
Let $(g_{n})_{n\in \nat} $ be a summable sequence of maps from $B$ to $C$.
Then one could ask whether this implies that the maps $g_{n}$ are strong phantoms for almost all $n$.
For any weakly compact map $f:A\to B$ we indeed have a factorization
of $\hat g\circ f$ over $A\to B\to \bigoplus_{n=0}^{k_{f}} C$ which shows that $g_{n}\circ f\simeq 0$ for all $n$ in $\nat$ with $n> k_{f}$. This does not suffice to conclude that the maps $g_{n}$ are phantom for large $n$ since $k_{f}$ may depend {on} $f$. \hB
\end{rem}

 We consider a sequence of maps $(g_{n})_{n\in \nat}$ from $B$ to $C$ in $\bD$.
\begin{prop}\label{iujgofffsdafdfaf} If $B$ admits a weakly compact  approximation and  all but finitely many members of the sequence $(g_{n})_{n\in \nat}$ are strong phantom maps, then the sequence is summable. In particular every strong phantom map out of $B$ is summable.
  \end{prop}
\begin{proof}
	Without loss of generality we can assume that each member of the sequence $(g_n)_{n \in \nat}$ is a strong phantom. Let $(B_{i})_{i \in \nat}$ be  a weakly compact approximation for $B$. Then for every $k$ in $\nat$, the compositions
	$B_{i}\to B\xrightarrow{g_{k}} C$  vanish for all $i$ in $\nat$ and the map $g_k: B \to C$ is the map induced by a commutative diagram
	$$ \xymatrix{ B_0 \ar[r] \ar[rrd]^0 & B_1 \ar[r] \ar[rd]^0 & B_2 \ar[r] \ar[d]^0  & B_3 \ar[dl]_0 \ar[r] &  \ar[dll]_0 \cdots \\ & & C }. $$
	Note that the main information is hidden in the fillers of the triangles.
	Thus $g_k$ is also the map induced on colimits of a  commutative diagram $\hat g_k$
	$$ \xymatrix{ B_0 \ar[r] \ar[d] & \cdots \ar[r] \ar[d] & B_{k-1} \ar[r] \ar[d]  & B_k \ar[r] \ar[d]^{0} & B_{k+1} \ar[r] \ar[d]^{0} & \cdots \\0 \ar[r]&  \cdots \ar[r] & 0 \ar[r]  & C \ar[r]^{\id_{C}} & C \ar[r]^{\id_{C}} & \cdots }. $$
	Here we set $B_{-1} = 0$. The family of diagrams $(\hat g_k)_{k\in \nat}$ induces a map into the product of the lower horizontal diagrams, which yields the commutative diagram
	$$ \xymatrix{ B_0 \ar[r] \ar[d] & \cdots \ar[r] \ar[d] & B_{k-1} \ar[r] \ar[d]  & B_k \ar[r] \ar[d] & B_{k+1} \ar[r] \ar[d] & \cdots \\\bigoplus_{\{0\}}C \ar[r]&  \cdots \ar[r] & \bigoplus_{\{0, \cdots, k-1\}} C \ar[r] & \bigoplus_{\{0, \cdots, k\}} C \ar[r] & \bigoplus_{\{0, \cdots, k+1\}} C \ar[r] & \cdots }. $$
	The map induced on colimits $B \to \bigoplus_{\nat} C$  by this diagram witnesses the summability of the sequence $(g_n)_{n \in \nat}$.
\end{proof}

\begin{prop}\label{weokrgpwregrefrfw9}
We consider  a composition $A\xrightarrow{f} B\xrightarrow{g}C$ of two strong phantom maps. 
 If $A$   admits a  weakly compact exhaustion, then $g\circ f\simeq 0$.
\end{prop}
\begin{proof}
 Let $(A_{n})_{n\in \nat}$ be a weakly compact exhaustion of  $A$ with   structure maps $\alpha_{n}:A_{n}\to A_{n+1}$. We then  have a Milnor sequence of abelian groups
$$0\to {\lim}^{1}_{n\in \nat^{\op}} [A_{n},\Omega B]\to [A,B]\to \lim_{n\in \nat^{\op}} [A_{n},B]\to 0\ .$$
The image of $f$ in $[A,B]$ under the second map vanishes so that
$f$ comes from an element $\hat f$ in $   {\lim}^{1}_{n\in \nat^{\op}} [A_{n},\Omega B]$.
We represent the ${\lim}^{1}$-term of a system $(G_{n})_{n\in \nat^{\op}}$ of abelian groups with structure maps $\gamma_{n}:G_{n+1}\to G_{n}$
as a quotient of $\prod_{n\in \nat} G_{n}$ by the image of the endomorphism $\id-s$, where $s$ sends
 a family $(g_{n})_{n\in \nat}$ to the family $(\gamma_{n}(g_{n+1}))_{n\in \nat}$. In this way $\hat f$ 
  is represented by a family $(\hat f_{n})_{n\in \nat}$ with $\hat f_{n}$ in $[A_{n},\Omega B]$.
We also have a Milnor sequence
$$0\to {\lim}^{1}_{n\in \nat^{\op}} [A_{n},\Omega C]\to [A,C]\to \lim_{n\in \nat^{\op}} [A_{n},C]\to 0\ .$$
The element $g\circ f$ is represented by the element in  ${\lim}^{1}_{n\in \nat^{\op}} [A_{n},C]$
represented by the family $(g\circ \hat f_{n})_{n\in \nat}$.
The family $(g\circ \hat f_{n+1}\circ \alpha_{n})_{n\in \nat}$ represents the same element in ${\lim}^{1}_{n\in \nat^{\op}} [A_{n},C]$.
But this is the zero family since $g$ is a strong phantom map and $\alpha_{n}$ is weakly compact for every $n$ in $\nat$.
\end{proof}

\begin{kor}
If $B$ is a strong phantom object admitting a weakly compact exhaustion, then $B\simeq 0$. In particular, if $\bD$ is  \pcas, then every strong phantom object is trivial.
\end{kor}
\begin{proof}
We have $\id_{B}\simeq \id_{B}\circ \id_{B}\simeq 0$ by   \cref{weokrgpwregrefrfw9}.
If $\bD$ is \pcas, then we use in addition that every object admits even a strongly compact  exhaustion by \cref{regijfowergffreferwfw}.
\end{proof}

\begin{kor}\label{weijogperfwrefregdghfg}
Let $f:B\to C$ be an equivalence and $\phi:B\to C$ be a strong phantom map. If $B$ admits a weakly compact exhaustion, then $f-\phi$ is an equivalence.
\end{kor}\begin{proof}{Since $\phi f^{-1} \phi f^{-1}\simeq 0$ and $ f^{-1} \phi f^{-1}\phi \simeq 0$ by \cref{weokrgpwregrefrfw9}, we can conclude}
$$(f-\phi)^{-1}\simeq f^{-1}+f^{-1}\circ \phi\circ f^{-1}\ .$$
 \end{proof}

Let $F:\bC\to \bD$ be a functor between  countably cocomplete stable $\infty$-categories,
and consider a strong phantom map $\phi:B\to C$.

\begin{prop}\label{wiotgjwogirwefwefw}
If   $F$ preserves  countable sums and $B$ admits a weakly compact approximation,  
then $F(\phi)$ is a strong phantom map.
\end{prop}
\begin{proof}
The phantom $\phi$ is summable by \cref{iujgofffsdafdfaf}.
Let $\hat \phi:B\to \bigoplus_{\nat }C$ be a witness of summability.
Then $$ F(B)\stackrel{F(\hat \phi)}{\to }
F( \bigoplus_{\nat }C)\stackrel{\simeq}{\leftarrow} \bigoplus_{\nat }F(C)$$
witnesses summability of $F(\phi)$.  By \cref{regjiewerggre9} we conclude that $F(\phi)$ is a strong phantom map.
\end{proof}

{We finally have the following characterization of phantom maps in  a compactly assembled  {stable} $\bD$.}\begin{prop}
\label{vbeirjgowjgj}
	  A map $f: D \to E$ is a strong phantom
	   if and only if for all countable colimit preserving functors $ F: \bD \to \Sp $ the map {$\pi_0 (F(f))$ of abelian groups is {the zero map.}}
\end{prop}
\begin{proof}
	If $f$ is strong phantom and $F: \bD \to \Sp$ a countable colimit preserving functor, then by \Cref{wiotgjwogirwefwefw} also $F(f)$ is a strong phantom{. Since the sphere spectrum}    $\mathbb{S}$ is compact in $\Sp$ the map $\pi_{{0}}(F(f)) = \pi_{{0}}(\map( \mathbb{S}, F(f)))$ is $0$.
	
	{For the converse} we consider a compact map $g: C \to D$ and a map $f: D \to E$ satisfying the vanishing condition. Since $g$ is compact, it admits a lift $\hat{g}$ along
	$$ \lvert - \rvert: \map_{\Ind^{\aleph_1}(\bD)}(y(C), S(D)) \to \map_{\bD}(C,D) \ .$$
	{Since} {the functor}  ${F(-):=}\map_{\Ind^{\aleph_1}}(y(C), S(-)): \bD \to \Sp$ preserves countable colimits we have {$\pi_{0}(F(f))(\hat g)=0$}.    This finally implies that $f{\circ} g \simeq \lvert S(f)\circ \hat{g} \rvert \simeq 0$ and therefore $f$ is a strong phantom.
\end{proof}

\section{$C^{*}$-algebras}\label{wrejgiowergerfwfwef}

\subsection{Semi-projectivity and shapes for  $G$-$C^{*}$-algebras}\label{erijgioepgwegerfwef}

   \newcommand{\alg}{\mathrm{alg}}
   
   In this section we recall  the theory of  of semi-projective maps between $C^{*}$-algebras from \cite{zbMATH03996430}. We thereby generalize the classical constructions  in a straightforward manner  from the
   non-equivariant to the equivariant case.  For  notational simplicity {we} will formulate the theory only  for separable $C^{*}$-algebras $G\nCalg_{\sepa}$, but   the material has a straightforward version in the general case with the clear exception of the existence of shape systems.  In order to do  homotopy theory in $G\nCalg_{\sepa}$ 
 we form
 the Dwyer-Kan localization $L_{h}:G\nCalg_{\sepa}\to G\nCalg_{\sepa,h}$   at homotopy equivalences. 
We interpret
 further results of \cite{zbMATH03996430} in this context. We will see that semi-projective maps have  properties that  are analogous to properties of compact maps considered in \cref{werogjwpegorefkmlerwgwergwerg}.  But the $\infty$-category 
 $G\nCalg_{\sepa,h}$
   does not fit precisely into the framework of  \cref{werogjwpegorefkmlerwgwergwerg} since  it it not known  to  admit all countable filtered colimits. Note that the colimits $Z:=\colim_{i\in I}Z_{i}$ below are always interpreted in $G\nCalg_{\sepa}$ and not in $G\nCalg_{\sepa,h}$. In order to resolve the lack of countable filtered colimits we will introduce the functor $\asGc:G\nCalg_{\sepa}\to\AsGc$ in
  \cref{irthjzgjdiogjhjdicnmjksdisjf}.

 Let $Z$ be in $G\nCalg_{\sepa}$ and $(I_{n})_{n\in \nat}$ be an increasing family of invariant ideals in $Z$.  The family $(I_{n})_{n\in \nat}$ is then called a presentation of the invariant direct limit ideal   $$I:=\overline{\bigcup_{n\in \nat}I_{n}}$$     
of $Z$.
The following definition generalizes  \cite[Def. 2.10]{zbMATH03996430} to the equivariant case. 
 
  \begin{ddd}\label{weorjgopwerfrefwref2}
 A morphism  $f:A\to B$   in $G\nCalg_{\sepa}$  is called semi-projective if for any
 $Z$ in $G\nCalg_{\sepa}$,  increasing family of invariant ideals $(I_{n})_{n\in \nat}$ with direct limit ideal $I$, and  the bold part of the diagram
 $$\xymatrix{A\ar@{..>}[r] \ar[d]^{f} &Z/I_{n}  \ar@{->}[d]  \\ B\ar[r] & Z/I } $$  there exists $n$ in $\nat$  and the dotted completion.
 \end{ddd}
 
 The semi-projective maps {form} a two-sided ideal in $G\nCalg_{\sepa}$.

    \begin{ddd}An algebra $A$ in $G\nCalg_{\sepa}$ is called  semi-projective if $\id_{A}$ is  semi-projective.  
    \end{ddd}

 \begin{constr}\label{egjweogrefwerfw}{\em 
 In this construction   we present the typical source of  semi-projective maps.
  Let $B$ be a $G$-$C^{*}$-algebra with norm $\|-\|_{B}$ and $(b_{k})_{k\in K}$ be a finite family of elements in $B$.
 Let furthermore $(p_{l})_{l\in L}$ be a finite family of non-commutative polynomials in variables
 $(x_{k,g})_{(k,g)\in K\times G}\cup (x_{k,g}^{*})_{(k,g)\in K\times G}$. We assume that this family contains the  {monomials} $x_{k,e}$ for all $k$ in $K$ in order to ensure that the set of relations used below is admissible in the sense of \cite[Def. 1.1]{zbMATH03996430}.
 We write $p_{l}(b)$ in $B$ for the evaluation of $p_{l}$ at  $x_{k,g}=gb_{k}$ for all $k$ in $K$ and $ g$ in $G$. 
 Let $(r_{l})_{l\in L}$ be a family of elements of $\R$ such that
 $\|p_{l}(b)\|_{B}<r_{l}$ for all $l$ in $L$. We can then define the universal $C^{*}$-algebra
 $A$ (in the sense of \cite[Def. 1.2]{zbMATH03996430}) generated by a family of elements $(b'_{k,g})_{(k,g)\in K\times G}$ subject to the relations $\|h^{*}p_{l}(b')\|\le r_{l}$ for all $h$ in $G$ and $l$ in $L$, where $h^{*}p(b')$ is the evaluation of $p$ at the family $(b'_{k,hg})_{(k,g)\in K\times G}$.
 The $C^{*}$-algebra $A$ has a canonical  $G$-action which extends the map $(h,b'_{k,g})\mapsto b'_{k,hg}$ on the generators.
 We furthermore have a  natural homomorphism $f:A\to B$ in $G\nCalg$ fixed by the conditions 
 $f(
 b'_{k,g})=gb_{k}$ for all $k$ in $K$ and $g$ in $G$.\hB } \end{constr}
 
 \begin{ddd}
 The ideal of  explicit semi-projective maps is
the ideal in $G\nCalg$  generated by maps 
  arising from the construction  \cref{egjweogrefwerfw}.  
 \end{ddd}

 \begin{prop}[{\cite[Prop. 2.35]{zbMATH03996430}}]\label{tokgprhhehtheth}
 An explicit semi-projective map  is  semi-projective.
 \end{prop}
 \begin{proof}
 Since the semi-projective maps form a two-sided ideal it suffices to show that the maps $f:A\to B$ constructed in \cref{egjweogrefwerfw} are semi-projective.
  In this argument we will add an index in order to indicate to which $C^{*}$-algebra a norm belongs.
 Let $(I_{n})_{n\in \nat}$ be an increasing family of  invariant ideals  of a $G$-$C^{*}$-algebra $Z$ with direct limit ideal $I$. We let $\iota :Z\to Z/I$ and $\iota_{j}:Z \to Z/I_{j}$ for $j$ in $\nat$ denote the projections.   We consider a map  $h:B\to Z/I
$ in $G\nCalg$. We can choose elements $\tilde b_{k}$ in $Z $ for all $k$ in $K$ such that
$h(b_{k})=\iota(\tilde b_{k})$.
Since $$\lim_{j\in \nat}\|  \iota_{j}(p_{l}(\tilde b))\|_{Z/I_{j}} =\|\iota(p_{l}(\tilde b))\|_{Z/I}= \|h(p_{l}(b))\|_{Z/I}\le  \| p_{l}(b)\|_{B}< r_{l}$$
(see e.g. \cite[Prop. 7.15]{KKG} for an argument for the first equality)
there exists $n$ in $\nat$  
such that $\|\iota_{n}(p_{l}(\tilde b))\|_{Z/I_{n}}\le r_{l}$ for all $l$ in $L$.
Then also $\|\iota_{n}(p_{l}(g\tilde b))\|_{Z/I_{n}}\le r_{l}$ for all $g$ in $G$, where $g\tilde b:=(g\tilde b_{k})_{k\in K}$.
By the universal property of $A$ there exists an equivariant  
 map $\tilde f:A\to Z/I_{n}$  fixed by $\tilde f(b'_{k,g})=  g\iota_{n}(\tilde b_{k})$.
Hence we get the desired square $$\xymatrix{A\ar@{..>}[r]^{\tilde f} \ar[d]^{f} &Z/I_{n} \ar@{->}[d] \\ B\ar[r]^{{h}} & Z/I } $$  in $G\nCalg_{\sepa}$.
 \end{proof}

 \begin{rem}\label{okprthrtertegtrgrg}
 In this remark we introduce the localization of the category of separable $G$-$C^{*}$-algebras at the homotopy equivalences.
  The category $G\nCalg $ of $G$-$C^{*}$-algebras has a topological enrichment by compactly generated topological spaces such that for any compact topological space $W$ and $A,B$ in $G\nCalg $ the topolgical  hom-space $\uHom_{G\nCalg }(A,B)$ is determined by 
 \begin{equation}\label{cdcadsnckladscasdcadc}
\Hom_{\Top} (W,\uHom_{G\nCalg }(A,B))\cong \Hom_{G\nCalg }(A,C(W, B))\ .
\end{equation} 
 Using the topological enrichment we can talk about paths of morphisms  or  homotopies between homomorphisms. In particular we get a notion of homotopy equivalences.
 
 We {now restrict to separable $G$-$C^{*}$-algebras and} let  \begin{equation}\label{cdcadsnckladscasdcadc13}L_{h}:G\nCalg_{\sepa}\to G\nCalg_{h,\sepa}\end{equation} 
 denote the  Dwyer-Kan localization of $G\nCalg_{\sepa}$ at the set of homotopy equivalences.   The $\infty$-category $G\nCalg_{h,\sepa}$ is the homotopy coherent nerve of   the Kan-enriched simplicial category 
 obtained by replacing the topological {hom}-spaces $\uHom_{G\nCalg_{\sepa}}(A,B)$ by their
 singular complexes {$\sHom_{G\nCalg_{\sepa}}(A,B)$}. By definition of  
  a Dwyer-Kan localization for any $\infty$-category ${\bC}$, we have an equivalence
\begin{equation}\label{fqwoeijoqwijdqewdqwed} L_{h}^{*}:\Fun(G\nCalg_{\sepa,h},\bC)\stackrel{\simeq}{\to} \Fun^{h}(G\nCalg_{\sepa},\bC)\ ,
\end{equation}	
where the superscript $h$ indicates homotopy invariant functors (see \cref{fjioqwefdqwedewdqwed}.\ref{fiqewpoqdqded}).
 
 Since for every $A$ in $G\nCalg_{\sepa}$ the  tensor product $A\otimes - $ preserves homotopy equivalences the Dwyer-Kan localization \eqref{cdcadsnckladscasdcadc13} is symmetric monoidal.  For any symmetric monoidal $\infty$-category pull-back along the symmetric monoidal refinement of $L_{h}$ induces  an equivalence 
 \begin{equation}\label{fqwoeijoqwijdqewdqwe332323d} L_{h}^{*}:\Fun_{\otimes/\lax}(G\nCalg_{\sepa,h},\bC)\stackrel{\simeq}{\to} \Fun_{\otimes/\lax}^{h}(G\nCalg_{\sepa},\bC)\ ,
\end{equation}	
where the subscript $\otimes /\lax$ indicates symmetric monoidal or lax symmetric monoidal functors.
 
The $\infty$-category $G\nCalg_{\sepa,h}$ is pointed and left-exact,  and the tensor product on this category is bi-left exact. 
A morphism $f:A\to B$ in $G\nCalg_{\sepa}$ is called a Schochet fibration  \cite{zicki}
if the induced map  $$\uHom_{G\nCalg_{\sepa}}(C,A)\to \uHom_{G\nCalg_{\sepa}}(C,B)$$ of topological $\Hom$-spaces   is a Serre fibration
of topological spaces for every $C$ in $G\nCalg_{\sepa}$. A cartesian square $$\xymatrix{A\ar[r]\ar[d]&B\ar[d]\\C\ar[r]&D}$$
in $G\nCalg_{\sepa}$ is called Schochet fibrant if $B\to D$ (and hence also $A\to C$) are Schochet fibrations. 
A functor is called Schochet exact if it sends Schochet fibrant cartesian squares to cartesian squares.
The functor $L_{h}$ is  Schochet exact
and pull-back along $L_{h}$ induces for every left-exact $\infty$-category $\bC$ an equivalence
 \begin{equation}\label{ghiwoergerfefrefw}L_{h}^{*}:\Fun^{\lex}(G\nCalg_{\sepa,h},\bC)\stackrel{\simeq}{\to} 
 \Fun^{h,\mathrm{Sch}}(G\nCalg_{\sepa},\bC) \ ,
\end{equation} 
where the {superscripts $\mathrm{Sch}$ and $\lex$ indicate Schochet exact and finite limit-preserving} functors, {respectivley}. For the arguments see
\cite[Sec. 3]{keb}.  \hB
 \end{rem}

 Let $f:A\to B$ be a  
 map in $G\nCalg_{\sepa}$, and let 
 $(Z_{n})_{n\in \nat} $ be a   system  in $G\nCalg_{\sepa}$ with $Z:=\colim_{n\in \nat}Z_{i}$.
 Recall the shifted lifting property  (slp) from  \cref{kotgpergoeprtkhgerth}.

     \begin{prop} [{\cite[Thm. 3.1 \& 3.3]{zbMATH03996430}}]\label{wrtkohpgertwfrewfwrf}
   {If $f$ is semi-projective, then  the} square of anima
$$     \xymatrix{\colim_{n\in \nat}\Map_{G\nCalg_{\sepa,h}}(A,Z_{n})\ar[r]  &\Map_{G\nCalg_{\sepa,h}}(A,Z)  \\ \colim_{n\in \nat}\Map_{G\nCalg_{\sepa,h}}(B,Z_{n}) \ar[r] \ar[u]& \Map_{G\nCalg_{\sepa,h}}(B,Z)\ar[u]} $$
has the slp for $\emptyset\to *$ and $S^{0}\to *$.
\end{prop}
   \begin{proof}
   We first show the slp for $*\to 0$.  We must show that 
 for any map 
 $B\to Z$  in   $G\nCalg_{\sepa}$ there exists $j$ in $\nat$ and a homotopy commutative diagram $$\xymatrix{A\ar@{..>}[r] \ar[d]^{f} &Z_{j} \ar@{->}[d] \\ B\ar[r] &\colim_{n\in \nat}Z_{n} } $$ 
 in $G\nCalg_{\sepa}$. We can copy the proof of \cite[Thm. 3.1]{zbMATH03996430}. We simply  observe that all the algebras constructed in this proof  carry canonical $G$-actions, and that the morphisms are equivariant.

We now show the slp for   $S^{0}\to *$. 
We let $\alpha_{n}:Z_{n}\to Z$  and $\alpha_{m,n}:Z_{n}\to Z_{m}$ for $m\ge n$  denote the structure maps of the system.
We consider 
two maps  $$\tilde h,\tilde h:B\to Z_{n}$$
 such that $\alpha_{n}\circ \tilde h:B\to Z$ is homotopic to $ \alpha_{n}\circ \tilde h':B\to Z$.
  We must show that there exists $m$ in $\nat$ with $m\ge n$
 such that $A\xrightarrow{f}B\xrightarrow{\tilde h} Z_{n}\xrightarrow{\alpha_{m,n}} Z_{m}$ and
 $A\xrightarrow{f}B\xrightarrow{\tilde h'} Z_{n}\xrightarrow{\alpha_{m,n}}  Z_{m}$ are homotopic.
 To this end we can again copy the argument for 
   \cite[Thm. 3.3]{zbMATH03996430} word by word.
\end{proof}

For explicit semi-projective maps we can show a much stronger lifting property.
Let $(Z_{i})_{i\in I}$ be a countable filtered system in $G\nCalg_{\sepa}$  with $Z:=\colim_{i\in I}Z_{i}$.

\begin{prop}\label{xjnbcuiohwedfiohgeriosdfylk}  If  $f:A\to B$ is an explicit semi-projective homomorphism of $G$-$C^{*}$-algebras, then the square of anima $$\xymatrix{\colim_{i\in I}\Map_{G\nCalg_{\sepa, h}}(A,Z_{i})\ar[r] &\Map_{G\nCalg_{ \sepa,h}}(A,Z)  \\ \colim_{i\in I}\Map_{G\nCalg_{\sepa, h}}(B,Z_{i})  \ar[r] \ar[u]_{f^{*}}&\Map_{G\nCalg_{ \sepa,h}}(B,Z)\ar[u]_{f^{*}}  } \ .$$  has the slp.\end{prop}
\begin{proof}
 It is easy to see that  the class of maps $f:A\to B$ for which the statement holds true is a two-sided ideal in $G\nCalg_{\sepa}$. It therefore suffices to consider the case that $f:A\to B$ arises from \cref{egjweogrefwerfw}.
We consider a map  $K\to L$  of finite anima and the bold part of 
$$\xymatrix{\colim_{i\in I}\Map_{G\nCalg_{\sepa,h}}(A,Z_{i})\ar[r] &\Map_{G\nCalg_{\sepa,h}}(A,Z)  \\\colim_{i\in I}\Map_{G\nCalg_{\sepa,h}}(B,Z_{i})  \ar[r] \ar[u]_{f^{*}}&\Map_{G\nCalg_{\sepa,h}}(B,Z)\ar[u]_{f^{*}}\\K\ar[u]\ar[r]&L\ar[u] \ar@{-->}[uul] } \ .$$
We must find a lift of the outer square indicated  by the dashed arrow.
In view of the description of the mapping spaces in $G\nCalg_{h}$ given in \cref{okprthrtertegtrgrg}
this problem can be modelled by a   diagram in $\sSet$ 
\begin{equation}\label{dascasdcadscadscacdcdscrrrrr}\xymatrix{\colim_{i\in I} {\sHom}_{G\nCalg}  (A,Z_{i})\ar[r] & {\sHom}_{G\nCalg} (A,Z)  \\\colim_{i\in I}{\sHom}_{G\nCalg} (B,Z_{i})  \ar[r] \ar[u]_{f^{*}}& {\sHom}_{G\nCalg}  (B,Z)\ar[u]_{f^{*}}\\ \bar{K} \ar[u]^{(1)}\ar[r]^{\bar{g}}&\bar{L}\ar[u]^{h} \ar@{-->}[uul] }
\end{equation}
upon applying the localization $\sSet\to \Spc$, where $\bar{K} \to \bar{L}$ is an inclusion of  finite simplicial sets; see \cref{gerihgowerf}.   The bold part is strictly commutative, and we are looking for a dashed arrow such that
the lower left triangle commutes and the upper right triangle commutes up to a homotopy $\beta $ with the special property that $g^{*}\beta$ is constant.  This implies after going to $\Spc$ that
the  fillers of the two triangles compose to the filler of the square.

We denote by $g:X\to Y$  the inclusion of second countable compact Hausdorff spaces given by $\lvert \bar{g} \rvert: \lvert \bar{K} \rvert \to \lvert \bar{L} \rvert$.  
Using \eqref{cdcadsnckladscasdcadc} we can rewrite the lifting problem  \eqref{dascasdcadscadscacdcdscrrrrr}  in  the form  of the commutative diagram
\begin{equation}\label{vfdvfvrwgwrewgrg}\xymatrix{ A\ar[dd]^{f}\ar@{..>}[r]&\ar[dd]^(0.3){\alpha_{Y,k}}C(Y)\otimes Z_{k}\ar[dr]^(0.4){r_{Z_{k}}}  & C(X)\otimes Z_{j}\ar[d]^{\alpha_{X,k,j}}\\ &&C(X)\otimes Z_{k}\ar[d]^{\alpha_{X,k}}\\B\ar[uurr]^(0.3){l}\ar[r]^-{h}&C(Y)\otimes Z\ar[r]^{r_{Z}}&C(X)\otimes Z&}\ .
\end{equation}  
  Since $X = \lvert \bar{K} \rvert$ and $K$ is a finite simplicial set, the map  marked by $(1)$ factorizes over a finite stage of the colimit
and gives rise to the map $l$.
The maps denoted by $r$ with decorations are induced by restriction along $g$,
and the maps $\alpha$ with decorations are induced by structure maps of the system $(Z_{i})_{i\in I}$ and its colimit.  The value of $k$ is still open. 

The following lemma provides the diagram \eqref{vfdvfvrwgwrewgrg}  together with the homotopy $\beta$ as desired.

\begin{lem}
There exists $k$ in $I$ with $j\le k$ and a dotted arrow in \eqref{vfdvfvrwgwrewgrg} such that the left square commutes up to homotopy $\beta$, 
the other square with left upper corner commutes strictly, and $r_{Z,*}\beta$ is constant.
\end{lem}
\begin{proof}

Let $b$ be in $B$. We claim that for every $\epsilon$ in $(0,\infty)$ there exists $k$ in $I$ with $j\le k$
and an element $\phi$ in $C(Y)\otimes Z_{k}$ such that
$\|\alpha_{Y,k}(\phi)-h(b)\| \le \epsilon$ and $r_{Z}(h(b))= \alpha_{X,k}(r_{Z_{k}}(\phi))$.
By the Choi-Effros theorem \cite{choi-effros} we can choose a completely positive {contractive} split $s:C(X)\to C(Y)$ such that $r\circ s=\id_{C(X)}$.
We extend $s$ as $s_{Z}:=s\otimes \id_{Z}$.
 Since $C(Y)\otimes -$ preserves filtered colimits we can first choose
 $k$  in $I$ with $j\le k$ and $\phi'$ in $C(Y)\otimes Z_{k}$  such that  $\|\alpha_{Y,k}(\phi')-h(b)\| \le \epsilon/3$. 
 Then we set $$\phi:=\phi'-s_{Z_{k}}(r_{Z_{k}}(\phi'))+s_{Z_{k}}( \alpha_{X,k,j} (l(b)))$$
 in $C(Y)\otimes Z_{k}$.
We then have \begin{eqnarray*}
r_{Z_{k}} (\phi)&=&r_{Z_{k}}( \phi')-r_{Z_{k}}( s_{Z_{k}}(r_{Z_{k}}(\phi')))+ r_{Z_{k}} (s_{Z_{k}} (\alpha_{X,k,j} (l(b))))\\&=&\alpha_{X,k,j} (l(b))
\end{eqnarray*}

 We further have
 \begin{eqnarray*}
\alpha_{X,k} (\alpha_{X,k,j}(l(b)))-\alpha_{X,k}(r_{Z_{k}}(\phi'))&=& r_{Z}(h(b))-\alpha_{X,k}(r_{Z_{k}}(\phi'))
\\&=& r_{Z}(\alpha_{Y,k}(\phi'))-\alpha_{X,k}(r_{Z_{k}}(\phi'))+ r_{Z}(h(b)-\alpha_{Y,k}(\phi'))\\&=&r_{Z}(h(b)-\alpha_{Y,k}(\phi'))
\end{eqnarray*}
and therefore
$$\|\alpha_{X,k} (\alpha_{X,k,j}(l(b))- r_{Z_{k}}(\phi'))\|\le \epsilon/3\ .$$
Hence we can choose $k$ sufficiently large such that
$$\| \alpha_{X,k,j}(l(b))- r_{Z_{k}}(\phi'))\|\le 2\epsilon/3\ .$$
Then 
$\|\phi-\phi'\|\le 2\epsilon/3$ and therefore $\|\alpha_{Y,k}(\phi)-h(b)\|\le \epsilon$.

 We now construct the dotted arrow. We fix $\epsilon$ in $(0,\infty)$ dependent on $A$ such that the following works. 
 For every $G$-orbits of generators (with representative $a_{i}$)
 we choose $\phi_{i}$ for $f(a_{i})$ as above. Then we get the dotted arrow. 
 We let $\psi_{i,t}$ be the straight interpolation between $h(f(a_{i}))$ and $\alpha_{Y,k}(\phi_{i})$.
 Then we get a homotopy $\beta:A\to C(Y\times I)\otimes Z$ whose
 value on $a_{i}$ is $\psi_{i,t}$.
 Since $      r_{Z}(\alpha_{Y,k}(\phi_{i}))= \alpha_{X,k}(r_{Z_{k}}(\phi))=\alpha_{X,j}(l(f(a_{i})))=r_{Z}(f(h(a_{i})))$
 we see that $r_{Z,*}\beta$ is the constant homotopy.
  \end{proof}
  This completes the proof of \cref{xjnbcuiohwedfiohgeriosdfylk}
  \end{proof}

\begin{rem}\label{gerihgowerf}
Let $K$ and $L$ be a finite anima and $\bar M$ and $\bar N$ be Kan complexes. We consider a    commutative square \begin{equation}\label{ndfgbdfgbdfgndgf} \xymatrix{K\ar[d]\ar[r]^{(1)}&L\ar[d]\\ \ell \bar M\ar[r]&\ell \bar N}\end{equation} in anima, where $\ell:\sSet\to \Spc$ is the canonical functor. We claim that there exists a commutative diagram
\begin{equation}\label{ndfgbdfgbdfgndgf1}\xymatrix{\bar K\ar[d]\ar[r]^{(2)}&\bar L\ar[d]\\ \bar M\ar[r]&\bar N}\end{equation} in $\sSet$ that upon applying $\ell$ goes to the diagram above  such that $\bar K$ and $\bar L$ are finite and $(2)$ is an inclusion.

 By representing the finite anima as finite simplicial sets and applying Kan's $\mathrm{Sd}/\mathrm{Ex}$-adjunction,
 we find a map $\bar K\to \bar L$ between finite simplicial complexes that is sent by $\ell$ to the map marked by (1) above. 
 Since $\bar N$ is a Kan complex, the commutativity of the square \eqref{ndfgbdfgbdfgndgf} yields a homotopy filling the square \eqref{ndfgbdfgbdfgndgf1}. Modifing the  right upper corner in  \eqref{ndfgbdfgbdfgndgf1}, using a mapping cylinder construction, we obtain the desired  commutative square.
 \hB \end{rem}

 Let $A$ be in $G\nCalg$.  
 \begin{ddd}[{\cite[Def. 4.1]{zbMATH03996430}}]\label{rwogpwegrewfrefrfwf}
 A  shape system for $A$ is a  system $(A_{n})_{n\in \nat}$ in $G\nCalg$
 such that $A\cong \colim_{n\in \nat} A_{n}$ and $A_{n}\to A_{n+1}$ is  semi-projective for every $n$ in $\nat$.
 If all these  maps are explicit  semi-projectives  then  this {is called} an explicit  shape system.
 \end{ddd}
 
 By \cref{tokgprhhehtheth} an explicit shape system is in particular a shape system.
 Our main motivation for introducing the notion of explicit shape systems is the following consequence of \cref{xjnbcuiohwedfiohgeriosdfylk}.

 \begin{kor}\label{jgriowergwrefrefrefwfrfwf} If $(A_{n})_{n\in \nat}$ is an explicit shape system for $A$ and $(Z_{i})_{i\in I}$ is a  system in $G\nCalg$ with $\colim_{i\in I} Z_{i}\cong Z$, then
 the map of systems
 $$(\colim_{i\in I }\Map_{G\nCalg_{h}}(A_{n},Z_{i}))_{n\in \nat^{\op}}\to  (\Map_{G\nCalg_{h}}(A_{n},Z))_{n\in \nat^{\op}}$$
 has the slp in the sense of \cref{gwjeiorrfrefrfw}.
  \end{kor}

 We consider   the full subcategory   $G\nCalg_{\sepa}$ of $G\nCalg$ of separable $G$-$C^{*}$-algebras.

 \begin{prop}[{\cite[Thm. 4.3]{zbMATH03996430}}]\label{ojkergperwerge}
Every separable $G$-$C^{*}$-algebra admits an explicit shape system.
  \end{prop}
  \begin{proof}
  One can copy the proof of  \cite[Thm. 4.3]{zbMATH03996430} which is based on  \cref{egjweogrefwerfw}.     Let us nevertheless provide a sketch.  Let $A$ be  in $G\nCalg_{\sepa}$.
  We consider a countable  subset $(a_{k})_{k\in \nat}$  of $A$ such that $(ga_{k})_{(k,g)\in \nat\times G}$ is dense in $A$. 
  We then consider the set of generators $(a_{k,g})_{(k,g)\in \nat\times G}$.
   Every non-commutative  polynomial  $p$ over $\Q[i]$  in variables $(x_{k,g})_{(k,g)\in \nat\times G}\cup (x^{*}_{k,g})_{(k,g)\in \nat\times G} $   gives rise to a  relation $\|  p(a)\|\le \|\ev(p(a))\|_{A}$, where $\ev(p(a))$ is the  image of the evaluation of $p$ at $(a_{k,g})_{(k,g)\in \nat\times G}$ in $A$ under the map which sends $a_{k,g}$ to $ga_{k}$.       The set  $R$ of these polynomials is countable and $A$ is isomorphic to the $G$-$C^{*}$-algebra generated by the set $(a_{k,g})_{(k,g)\in \nat\times G}$ with respect to these relations and the $G$-action extending
   $h a_{k,g}=a_{g,hg}$. We let $ (p_{n})_{n\in \nat}$ be an indexing of  
$R$. We let $R_{n}$ be the {union of the} subset of those relations in the segment of the sequence for $ k=0,\dots,n$ which only involve the variables
   $a_{k,g}$ for $k=0,\dots,n$ and $g$ in $G$ {with the set of monomials $\{x_{k,e}\mid k=0,\dots,n\}$.}  We then define $A_{n}$ as the universal $G$-$C^{*}$-algebra
   generated by  $(a_{k,g})_{(k,g)\in \{0,\dots,n\}\times G}$ subject to the relations 
   $\|h^{*}p_{l}(a)\|\le \|\ev(p_{l}(a))\|_{A}+1/n$ for all $l$ in $R_{n}$ and $h$ in $H$. The connecting maps $A_{n}\to A_{n+1}$ are explicit semi-projective since they arise from \cref{egjweogrefwerfw}. Furthermore, 
 we have an isomorphism $A\cong \colim_{n\in \nat}A_{n}$.
    \end{proof}
Explicit shape systems can be chosen functorially for $\nat$-indexed systems.  Let $(A_{n})_{n\in \nat}$ be a system in $G\nCalg_{\sepa}$.
\begin{prop}\label{gjiowergrwfwerfrf}
There exists  commutative
diagram 
 $$\xymatrix{A_{0,0}\ar[d]\ar[r] &A_{1,0} \ar[d]\ar[r]&A _{2,0}\ar[r] \ar[d]&A_{3,0}\ar[r]\ar[d]&\dots\\
 A_{0,1}\ar[d]\ar[r] &A_{1,1} \ar[d]\ar[r]&A_{2,1}\ar[r] \ar[d]&A_{3,1}\ar[r]\ar[d]&\dots  \\
 A_{0,2}\ar[d]\ar[r] &A_{1,2} \ar[d]\ar[r]&A_{2,2}\ar[r] \ar[d]&A_{3,2} \ar[r]\ar[d]&\dots\\\vdots&\vdots&\vdots&\vdots& & \\A_{0}\ar[r]&A_{1}\ar[r]&A_{2} \ar[r]&A_{3}\ar[r]&&\dots}$$  in $G\nCalg_{\sepa}$ such that the lower line is the colimit of the diagram taken vertically  and $(A_{n,k})_{k\in \nat}$ is an explicit shape system for  $A_{n}$ every $n$ in $\nat$.  \end{prop}
\begin{proof}
We construct the diagram inductively using the argument for \cref{ojkergperwerge}.
Assume that we have constructed $A_{n,k}$ and the maps  for $n<n_{0}$ and all $k$ in $\nat$.
We then construct the explicit shape system $(A_{i}')_{i\in \nat} $ of $A_{n_{0}}$ as in the proof of  \cref{ojkergperwerge}.
We thereby make sure that the countable subset  of $A_{n_{0}}$ used in the construction
contains the images of  the generators  {of the} $A_{n_{0}-1,i}$ for all $i$ in $\nat$.
In order to define the diagram
$$\xymatrix{A_{n_{0}-1,0}\ar[r]\ar[d]&A_{n_{0},0}\ar[d]\\A_{n_{0}-1,1}\ar[r]\ar[d]&A_{n_{0},1}\ar[d]\\A_{n_{0}-1,2}\ar[r]\ar[d]&A_{n_{0},2}\ar[d]\\\vdots&\vdots
}$$ we choose an increasing  sequence of integers $(k(i))_{i\in \nat}$
and set $A_{n_{0},i}:=A'_{k(i)}$. We thereby choose $k(i)$ so large such that
the image of  every  generator  $x$ of $A_{n_{0}-1,i}$ in $A_{n_{0}}$ is a generator $x'$  of $A'_{k(i)}$  such that 
 $\|x'\|_{A'_{n_{0},k(i)}}<\|x\|_{A_{n_{0}-1,i}}$. We can then define $A_{n_{0}-1,i}\to A_{n_{0},i}$ such that it sends $x$ to $x'$.
 \end{proof}

\subsection{Axiomatic characterization of $\AsGc$ and $\EsG$}\label{gjkowperefwfrefrfrefw}
The classical equivariant $E$-theory functor $\esGn$ in \eqref{fqwefewdqwedeerererwefqf}
is a functor from separable $G$-$C^{*}$-algebras to an additive category constructed in  \cite[Def. 6.8]{Guentner_2000}. The functor   $\asG$ in \eqref{fqwefewdqwedeerererwefqf}  to the ordinary asymptotic morphism category  \cite[Def. 2.13]{Guentner_2000} constitutes an intermediate step in the construction of $E$-theory.
In the present paper we will consider  refinements of both functors 
to functors with values in $\infty$-categories. We think of these refinements as means to do homotopy theory with separable $G$-$C^{*}$-algebras reflecting some additional invariance properties. 
While $L_{h}:G\nCalg_{\sepa}\to G\nCalg_{\sepa,h}$ from \cref{okprthrtertegtrgrg} is the correct framework for just doing homotopy theory, the functor  $\asGc:G\nCalg_{\sepa}\to \AsGc$ provides the correct version of homotopy theory
 which is also compatible with filtered colimits. If one adds Morita invariance and exactness to the picture, then
$\esG:G\nCalg_{\sepa}\to \EsG$ will be the appropriate home. 
This point of view suggests to first characterize these functors axiomatically
in terms of universal properties. We then prove existence  by providing constructions. 
Note that the constructions yield additional information about these functors which are not obvious from the axiomatic characterization alone, e.g. the property of being \pcas{} asserted by our main   \cref{jifofqweewf9}.

Before we state the characterizations  \cref{wrejgioweferfw} and  \cref{wrejgioweferfw1} we recall in \cref{fjioqwefdqwedewdqwed} all the $C^{*}$-algebraic notions
necessary to formulate the universal properties. 
We will show that an   {exact and countable sum-preserving functor is automatically   filtered colimit preserving and apply this   to $\esG$.}
{At the end of this section  we} explain   how the functors {$\asGc$ and $\esG$} can be extended {as $s$-finitary functors} to all $C^{*}$-algebras in a formal way.

 In  \cref{twhkogprthtrerge},  \cref{twhkogprthtrerge1} and \cref{irthjzgjdiogjhjdicnmjksdisjf} we
present  three independent  constructions of an equivariant $E$-theory functor $\esG$.  The verification that all three constructions yield   equivalent functors which  satisfy the desired universal properties will be completed in
  \cref{xvbnfiohregiohsdfpojsdf}. The functor $\asGc$ will be constructed in \cref{irthjzgjdiogjhjdicnmjksdisjf} 
 and again constitutes an intermediate {step} towards $E$-theory.

Since in this section we also want to formulate  the universal properties for the extension of the functors to all $G$-$C^{*}$-algebras  $G\nCalg$ {(as opposed to separable algebras $G\nCalg_{\sepa}$)}  we will describe the necessary $C^{*}$-algebraic notions in this generality.
 
 We equip $G\nCalg$ with the symmetric monoidal structure $\otimes$ induced by the maximal tensor product.
 The choice of the maximal tensor product (in contrast  e.g. to the minimal one) is important since we need that the tensor product preserves exact sequences and 
 filtered colimits in each argument.
 The tensor product of two separable $C^{*}$-algebras is again separable so that $\otimes $ restricts to $G\nCalg_{\sepa}$. We furthermore note that $\otimes$ is bi-continuous for the topological enrichment of $G\nCalg$.

 If $(A_{i})_{i\in I}$ is  a family  in $G\nCalg$, then we can form the categorical product $\prod_{i\in I}A_{i}$ in $G\nCalg$. The elements of this product are uniformly bounded families $(a_{i})_{i\in I}$ with $a_{i}$ in $A_{i}$\footnote{Note that this description is only correct for discrete groups $G$ which is our standing assumption.}. The sum $\bigoplus_{i\in I}A_{i}$ is the $G$-$C^{*}$-subalgebra
 of the product generated by families $ (a_{i})_{i\in I}$ with finitely many non-zero entries.
 We have canonical inclusions $\epsilon_{i}:A_{i}\to \bigoplus_{i\in I}A_{i}$ for all $i$ in $I$.
 But note that the sum is not the categorical  coproduct of the family. The latter   would be represented 
 by the free product $\ast_{i\in I} A_{i}$.
 A  sum or coproduct of a countable  family of separable $G$-$C^{*}$-algebras is again a separable $G$-$C^{*}$-algebra. In contrast, an infinite  product of non-trivial $C^{*}$-algebras is never separable.

\begin{rem}\label{fjioqwefdqwedewdqwed} 
 We start with a list of properties which may be satisfied by   functors defined on $G\nCalg$   or at least on its subcategory  $G\nCalg_{\sepa}$ of separable $G$-$C^{*}$-algebras.
\begin{enumerate}
 \item \label{fiqewpoqdqded}homotopy invariance: The enrichment of $G\nCalg$ in topological spaces described  in \cref{okprthrtertegtrgrg} yields a notion of homotopy {equivalence}.  A morphism $f:A\to B$ is a homotopy equivalence if there exists a morphism $g:B\to A$ such that the compositions 
 $f\circ g$ and $g\circ f$ are homotopic to the identities of $A$ and $B$, respectively. A functor defined on $G\nCalg$ or $G\nCalg_{\sepa}$ is called homotopy invariant if it sends homotopy equivalences to equivalences. 
\item\label{fiqewpoqdqded1} $K_{G}$-stability: We define the separable $G$-$C^{*}$-algebra $K_{G}:=K(L^{2}(G)\otimes \ell^{2})$ with the 
  $G$-action induced by the left-regular representation on $L^{2}(G)$. A morphism $f:A\to B$ in $G\nCalg$ is called a $K_{G}$-equivalence   if $f\otimes \id_{K_{G}}:A\otimes K_{G}\to B\otimes K_{G}$ is a homotopy equivalence. 
A functor   is called  $K_{G}$-stable,  if it sends $K_{G}$-equivalences   to equivalences.
Since $K_{G}$ is separable the notion of $K_{G}$-stablility  is also  well-defined for functors only defined on separable algebras.
\item \label{fiqewpoqdqded2} exactness:    A   functor $F$ to a pointed category  is called  exact, if {it} sends the zero algebra to the zero object and 
 exact sequences of $G$-$C^{*}$-algebras to fibre sequences.
More precisely, writing an exact sequence of $G$-$C^{*}$-algebras as a commutative  square
$$\xymatrix{A\ar[r]\ar[d]&B\ar[d]\\0\ar[r]&C} $$ which is cartesian and cocartesian, we require
that $F(0)\simeq 0$ and that  the resulting commutative square
$$\xymatrix{F(A)\ar[r]\ar[d]&F(B)\ar[d]\\0\ar[r]&F(C)} $$ 
is cartesian.
\item countable sum-preserving:     A functor $F$  is called countable sum-preserving
if for any countable family $(A_{i})_{i\in I}$ in $G\nCalg$ the coproduct $\coprod_{i\in I} F(A_{i})$ in the target of $F$  exists and  the canonical map
$$\sqcup_{i\in I} F(\epsilon_{i}):\coprod_{i\in I} F(A_{i})\to F(\bigoplus_{i\in I}A_{i})$$   is an equivalence.
Since a countable sum of separable $G$-$C^{*}$-algebras is again separable, the condition of being countable sum preserving is also well-defined for functors only defined on separable $G$-$C^{*}$-algebras.
\item \label{qijfoqrfqwfefqewqef} countable filtered colimit preserving: For every system $(A_{n})_{n\in \nat}$ in $G\nCalg$  the colimit $\colim_{n\in \nat} F(A_{n})$ in the target of $F$ exists and the canonical map
$$\colim_{n\in \nat} F(A_{n})\to F(\colim_{n\in \nat} A_{n})$$ is an equivalence. Since the colimit  of a  system $(A_{n})_{n\in \nat}$ of separable $G$-$C^{*}$-algebras is again separable this property makes also sense for functors only defined on  separable $G$-$C^{*}$-algebras.
\item\label{wergoijoergferfegwegwerg} $s$-finitariness: This property is only defined for functors $F$ defined on all of $G\nCalg$. 
  A functor $F$ is called  s-finitary if    for every $A$ in $G\nCalg$  the colimit   $\colim_{A'\subseteq_{\sepa} A} F(A')$ exists in the target of $F$ and  the canonical morphism $$\colim_{A'\subseteq_{\sepa} A} F(A')\to F(A)$$ is an equivalence. Here $A'\subseteq_{\sepa} A$ is the poset of $G$-invariant separable subalgebras $A'$ of $A$. {The functor  $F$} is s-finitary  if and only if $F$ preserves $\aleph_1$-filtered colimits. 
   {In order to see this note} that  the realization functor $$\Ind_{\aleph_1}(G\nCalg_{{\sepa}}) \to G\nCalg  $$
   {defined as the left Kan-extension of the inclusion $G\nCalg_{\sepa}\to G\nCalg$ along $y:G\nCalg_{\sepa}\to\Ind_{\aleph_1}(G\nCalg_{{\sepa}})$}
    is an equivalence  with inverse
  \[ 
  	G\nCalg \to\Ind_{\aleph_1}(G\nCalg_{{\sepa}})\ , \quad A \mapsto \colim_{ A^\prime \subseteq_{\sepa} A} y(A^\prime)\ .
  \]

\end{enumerate}
 \end{rem}

 \begin{ddd}\label{wrejgioweferfw1}
The    functor  $\asGc$  is the initial
functor $$\asGc:{G}\nCalg_{\sepa}\to \AsGc$$ to an $\infty$-category admitting countable filtered colimits   
which is homotopy invariant and preserves filtered colimits.
\end{ddd}
The universal property required in \cref{wrejgioweferfw1} {asserts} that precomposition
with $\asGc$ induces for every $\infty$-category {$\bC$} admitting countable filtered colimits an equivalence 
$$(\asGc)^{*} :\Fun^{\cfil}( \AsGc ,\bC)\stackrel{\simeq}{\to} \Fun^{h, \cfil}( G\nCalg_{\sepa},\bC)\ ,$$
where the superscript $\cfil$ indicates countable filtered colimit preserving functors {(for $h$ see \cref{okprthrtertegtrgrg})}. 
 It is clear that \cref{wrejgioweferfw1} characterizes  the functor $\asGc$ essentially uniquely.  The construction of  this functor will be completed with \cref{uqwighfuihvsdiohjfhjkbvdxc}.

\begin{ddd}\label{wrejgioweferfw}
The   equivariant $E$-theory functor for separable $C^{*}$-algebras is the initial
functor $$\esG:\nCalg_{\sepa}\to \EsG$$ to a countably cocomplete stable $\infty$-category
which is homotopy invariant, $K_{G}$-stable, exact and countable sum-preserving.
\end{ddd}
The universal property required in \cref{wrejgioweferfw} {asserts} that precomposition with $\esG$ induces for every countably cocomplete stable $\infty$-category $\bC$ an equivalence
 $$(\esG)^{*} :\Fun^{\colim_{\omega}}( \EsG ,\bC)\stackrel{\simeq}{\to} \Fun^{ h,K_{G},\exa, \oplus}( G\nCalg_{\sepa},\bC)\ ,$$
 where the superscripts $\colim_{\omega}$, $ h$, $K_{G}$, $\exa$,
and $\oplus$ indicate the corresponding property of functors: countable colimit preserving, homotopy invariant, $K_{G}$-stable, exact and countable sum preserving.
It is  again clear that \cref{wrejgioweferfw} characterizes the  functor $\esG$  essentially uniquely. 
The statements in \cref{xvbnfiohregiohsdfpojsdf} complete the proof of existence.

{Note that the $\infty$-categories $\EsG$ and $\AsGc$ have canonical symmetric monoidal structures  
such that the functors $\esG$ and $\asGc$ have canonical symmetric monoidal refinements
which are characterized by analogous universal properties. For details we refer to the  subsequent subsections.}

Note that in \cref{wrejgioweferfw} we  require the condition of being countable sum preserving since
this property can also be formulated for the classical functor  $\esGn$.
The more natural property of being  countable filtered colimit preserving  is not expected for 
 the classical functors since  the functor {$\EsG\to \ho(\EsG)\simeq \EsGn$}
does not preserve filtered colimits. For the same reason in the case of $\asGc$  {there} is no nice characterization of the classical functor $\asG$ by a universal property.

The following general observation will ensure that $\esG$ nevertheless preserves all countable filtered colimits.
 \begin{prop}\label{weijgwoefrewfwdffsfd}If $F:G\nCalg_{\sepa}\to \bC$ is a homotopy invariant, exact and  countable sum preserving functor to a countably cocomplete stable $\infty$-category $\bC$, then
$F$ preserves countable filtered colimits. \end{prop}
\begin{proof}
We  adapt   the proof given in  \cite[Prop. 2.6]{MR2193334} to the $\infty$-categorical context and add some details.
It suffices to
show that the functor $F$  preserves $\nat$-indexed colimits. 

Let $(A_{n})_{n\in \nat}$ be a system in $G\nCalg_{\sepa}$ with structure maps $\alpha_{n}:A_{n} 
 \to A_{n+1}$ for all $n$ in $\nat$ and $\colim_{n\in \nat}A_{n}\simeq A$ with structure maps $\iota_{n}:A_{n}\to A$.
 Then we define the telescope
 $T$ in $G\nCalg_{\sepa}$ as the subalgebra of
 families $(f_{n})_{n\in \nat}$ in $\prod_{n\in \nat} C([n,n+1],A_{n})$ which satisfy the following conditions:
  \setlist[enumerate]{itemsep=-0.4cm}\begin{enumerate} \item \label{ewijgowgfrfwfwrrf9}$f_0(0)=0$
 \item  \label{ewijgowgfrfwfwrrf91}
 $\alpha_{n}(f_{n}(n+1))=f_{n+1}(n+1)$ for all $n$ in $\nat$ \item  \label{ewijgowgfrfwfwrrf92}
 $\lim_{t\to \infty} \|f_{n(t)}(t)\|=0$, 
 where $n(t)$ in $\nat$  is such that $t\in (n(t),n(t)+1]$ for every $t$ in $[0,\infty)$. 
\end{enumerate}

The evaluation    $e:(f_{n})_{n\in \nat}\mapsto (f_{n}(n+1))_{n\in \nat}$ fits into 
 an exact sequence \begin{equation}\label{dvacdscweqcqwcsdac}
 0\to \bigoplus_{n\in \nat} S( A_{n})\to T\xrightarrow{e} \bigoplus_{n\in \nat} A_{n}\to 0   \end{equation}
in $G\nCalg_{\sepa}${, where $S(B):=C_{0}((0,1), B)$ for any $B$ in $G\nCalg$.}
We apply $F$ and use that it is exact and preserves {countable} sums in order to get a fibre sequence
\begin{equation}\label{vwepojpocdcsdcs}F(T)\xrightarrow{F(e)} \bigoplus_{n\in \nat}  F(A_{n}) \xrightarrow{1-s} \bigoplus_{n\in \nat}  F(A_{n})\xrightarrow{
} \Sigma F(T)\ , 
\end{equation}  
where $s$ sends a family $(x_{n})_{n\in \nat}$ to $(F(\alpha_{n-1})(x_{n-1}))_{n\in \nat}$, and where we set $F(\alpha_{-1})(x_{-1}):=0$. 
The fibre sequence  exhibits $\Sigma  F(T)$    as the colimit of the system $(F(A_{n}))_{n\in \nat}$ (compare with \cite[Def. 1.6.4]{neem}).

\begin{rem}In order to see the formula for the map $1-s$  in the middle one can consider the map of exact sequences $$\xymatrix{0\ar[r]&S(A)\ar[r]\ar@{=}[d]&C_{0}((0,1],A)\ar[r]^-{\ev_{1}}\ar[d]&A\ar[d]^{i_{1}}\ar[r]&0\\0\ar[r]&S(A)\ar[r]&C([0,1],A)\ar[r]^-{\ev_{0}\oplus \ev_{1}}&A\oplus A\ar[r]&0}\ ,$$ where $i_{1}$ is the inclusion of the second summand.
Applying $F$  and using the contractability of $C_{0}((0,1],A)$ we get the map of fibre sequences
$$\xymatrix{  0\ar[r]\ar[d]&F(A)\ar[d]^{F(i_{1})}\ar[r]^{-\id_{F(A)}}&F(A)\ar@{=}[d] \\F(C([0,1],A))\ar[r]&F(A)\oplus F(A)\ar[r]^-{\delta}&F(A)}\ ,$$ from which we read off that the map  $\delta\circ F(i_{1}) =- \id_{F(A)}$.  If we replace $C_{0}((0,1],A)$ by $C_{0}([0,1),A)$  and $i_{1}$ by the embedding $i_{0}$ of the first summand, then a similar argument shows  that $\delta\circ F(i_{0})=\id_{F(A)}$. 
Hence $\delta=\pr_{0}-\pr_{1}$. 

We could interpret $C([0,1],A)$ as the analogue of a factor $ C([n,n+1],A_{n})$ of $T$.
Applying an analogous reasoning to \eqref{dvacdscweqcqwcsdac} we can  read off the map $1-s$ as claimed.
\hB
\end{rem}
The fibre sequence \eqref{vwepojpocdcsdcs} also identifies   $F(T)$ with $\colim_{n\in \nat} \Omega F(A_{n})$.

We have a homomorphism $c:T\to S(A){\cong}  C_{0}((0,\infty),A)$ which sends $(f_{n})_{n\in \nat}$ to the function
$t\mapsto \iota_{n(t)}(f_{n}(t))$. It is easy to see that the following diagram commutes:
$$\xymatrix{\bigoplus_{n\in \nat} \Omega F(A_{n})\ar[r]^{1-s}\ar@{..>}[dr]_{0}&\bigoplus_{n\in \nat}  \Omega F(A_{n})\ar[d]^{\sum_{n\in \nat}\Omega F(\iota_{n})}\ar[r]& F(T)\ar[d]^{F(c)}\\& \Omega F(A)&\ar[l]^{\simeq } F(S(A)) }\ .$$     The lower horizontal map is an equivalence since
 $F$ is homotopy invariant and reduced, see the argument leading to \eqref{adsvacadqwe}.
 The diagram   identifies $F(c)$ with the desuspension of the canonical map $$\colim_{n\in \nat} F(A_{n})\to  F(\colim_{n\in \nat} A_{n})\ .$$ In order to show that it is an equivalence we  
 let $\tilde T$ be the subalgebra  of  $\prod_{n\in \nat} C([n,n+1],A_{n})$  consisting of the families 
satisfying the conditions \cref{ewijgowgfrfwfwrrf9} and \cref{ewijgowgfrfwfwrrf91}  from above and the weakening    of \cref{ewijgowgfrfwfwrrf92} requireing $$\lim_{s,t\to \infty,s\le t} \|\alpha_{n(t),n(s)}( f_{n(s)}(s))-\alpha_{n(t)}(f_{n(t)}(t))\|=0$$ with $\alpha_{n,m}:=\alpha_{n-1}\circ\dots\circ \alpha_{m}:A_{m}\to A_{n}$  for integers $m,n$ with $m\le n$.  We have a map
 $$\tilde T\to A\ ,  \quad (f_{n})_{n\in \nat}\mapsto \lim_{t\to \infty } \iota_{n(t)}(f_{n(t)}(t))\ ,$$
We now consider
 the map of exact sequences
 $$
\xymatrix{
0\ar[r]&T\ar[d]^{c}\ar[r]&\tilde T\ar[r]\ar[d]^{\tilde c}&A\ar[r]\ar@{=}[d]&0\\
0\ar[r]&C_{0}((0,\infty),A)\ar[r]&C_{0}((0,\infty],A)\ar[r]^-{\ev_{\infty}}&A\ar[r]&0}
\ ,$$
where $\tilde c$ is given by the same formula as $c$.
Since $\tilde T$ and $C_{0}((0,\infty],A)$ are contractible, the map of fibre sequences obtained by applying $F$ reduces to the square
$$\xymatrix{F(A)\ar[r]^-{\simeq}\ar@{=}[d] & \Sigma F(T)\ar[d]^-{\Sigma F(c)} \\ F(A) \ar[r]^-{\simeq} & \Sigma F(S(A))} $$
 which shows that $F(c)$ is an equivalence.
\end{proof}

\begin{rem}  
The analogue of \cref{weijgwoefrewfwdffsfd} for functors $F:G\nCalg\to \bC$ 
is also true with the same proof. \hB
\end{rem}

The following is a consequence of 
 \cref{weijgwoefrewfwdffsfd}. 
 \begin{kor} \label{cgrgregeffweeeee}
 The equivariant $E$-theory functor $\esG:G\nCalg_{\sepa}\to \EsG$   preserves all countable filtered colimits.
\end{kor}

  {In order to complete the discussion of the equivariant $E$-theory    we will  discuss   induction and restriction functors and crossed products and adjunctions relating them in \cref{kogregwergwerg} and  \cref{kogregwergwerg1}.}

{In the following we introduce the 
 s-finitary extension of the functors  $\asGc$  and  $\esG$ to all $G$-$C^{*}$-algebras and state their universal properties.} 


 
 \begin{ddd}\label{iuhiowerfwefewf}We define the presentable stable $\infty$-categories \begin{equation}\label{feqwewdwedqedewdqwd} \AGc:=\Ind_{\aleph_{1}}(\AsGc)\ , \qquad \EG:=\Ind_{\aleph_{1}}(\EsG)\ .
\end{equation}\end{ddd}  
Note that here $\aleph_{1}$ is a subscript and $\Ind_{\aleph_{1}}(\bC)$ indicates the $\aleph_{1}$-ind completion of $\bC$, i.e., the $\infty$-category obtained from $\bC$ by adding freely all $\aleph_{1}$-filtered colimits.
\footnote{This ind completion should not be confused with $\Ind^{\aleph_{1}}(\bC)$ which stands for freely adding   countable (not countably!) filtered colimits.} We let $y':\bC\to \Ind_{\aleph_{1}}(\bC)$ the 
 canonical embedding which preserves countable  filtered colimits.
 \begin{ddd}\label{wokjgwpegregfwrefrewfgdgh} We define the functors
\begin{equation}\label{oqjofpqwefqwefqwdq}\aGc:  G\nCalg\to \AGc\ , \qquad \eGG:G\nCalg\to \EG
\end{equation} {as}  left Kan-extensions 
$$\xymatrix{G\nCalg_{\sepa}\ar[dr]_{\incl}\ar[r]^{\asGc}&\AsGc\ar[r]^{y'}&\AGc\\&G\nCalg\ar@{..>}[ur]_{\aGc}&}\ , \qquad \xymatrix{G\nCalg_{\sepa}\ar[dr]_{\incl}\ar[r]^{\esG}&\EsG\ar[r]^{y'}&\EG\\&G\nCalg\ar@{..>}[ur]_{\eG}&}\ .$$
\end{ddd}

In the statement below  the superscript $\mathrm{sfin}$ stands for  the condition of being $s$-finitary, see \cref{fjioqwefdqwedewdqwed}.\ref{wergoijoergferfegwegwerg}. The following proposition states the universal propertes of {the functors described in \cref{wokjgwpegregfwrefrewfgdgh}.}
\begin{prop}\label{oopwerfewrsfg} \mbox{} \begin{enumerate}\item Pull-back along $\aGc$ induces for every   $\infty$-category $\bC$  {with filtered colimits} an equivalence 
\begin{equation}\label{fasdfoiuqiofafds2}(\aGc)^{*} :\Fun^{\mathrm{fil}}(\AGc,\bC)\stackrel{\simeq}{\to}\Fun^{h,\cfil ,\mathrm{sfin}}(G\nCalg,\bC)\ .
\end{equation}
\item \label{oigopregpwerferwf9} Pull-back along $\eGG$ induces for every cocomplete stable $\infty$-category $\bC$ an equivalence
\begin{equation}\label{fasdfoiuqiofafds}(\eGG)^{*} :\Fun^{\colim}(\EG,\bC)\stackrel{\simeq}{\to}\Fun^{h,K_{G},\exa,\oplus,\mathrm{sfin}}(G\nCalg,\bC)\ .
\end{equation}  \end{enumerate}
\end{prop}
\begin{proof} This follows immediately from \cref{wrejgioweferfw1} or \cref{wrejgioweferfw} {and} the universal property of $\Ind_{\aleph_{1}}$. See also \cite[Sec. 3]{KKG} or \cite[Sec. 8]{keb} for more detailed arguments.
\end{proof}
The symmetric monoidal structure on $\AsGc$ or $\EsG$ (note that it preserves countable colimits in each argument) induces an entry-wise filtered colimit preserving symmetric monoidal structure on the $\aleph_{1}$-$\Ind$-completions $\AGc$ or  $\EG$.

\begin{prop} \label{9gwergwerfw} \mbox{} \begin{enumerate}\item 
The functor $\aGc$ has a symmetric monoidal refinement such that
the pull-back along $\aGc$ induces for every symmetric monoidal   $\infty$-category $\bC$ {admitting filtered colimits}  {{with} an entry-wise filtered colimit preserving tensor product}
an equivalence
\begin{equation}\label{fasdfoiuqioewwewefafds}(\aGc)^{*} :\Fun_{\otimes/\mathrm{lax}}^{{\mathrm{fil}}}(\AGc,\bC)\stackrel{\simeq}{\to} \Fun_{\otimes/\mathrm{lax}}^{ h, \cfil,\mathrm{sfin}}(G\nCalg,\bC)\ .
\end{equation} 

\item \label{giigopwerfwergweg}The functor $\eGG$ has a symmetric monoidal refinement such that
the pull-back along $\eGG$ induces for every  cocomplete symmetric monoidal  stable $\infty$-category $\bC$ with bi-cocontinuous tensor product 
an equivalence
\begin{equation}\label{fasdfoiuqiofafds4}(\eGG)^{*} :\Fun_{\otimes/\mathrm{lax}}^{\colim}(\EG,\bC)\stackrel{\simeq}{\to} \Fun_{\otimes/\mathrm{lax}}^{h,K_{G},\exa,\oplus,\mathrm{sfin}}(G\nCalg,\bC)\ .
\end{equation}   \end{enumerate}

\end{prop}
\begin{proof}[Sketch]
The arguments are similar as the corresponding parts of \cite[Sec. 8]{keb} or the proof of \cite[Prop. 3.8]{KKG}.
\end{proof}

 \subsection{Construction of ${\EsGght}$ as a localization at $E$-equivalences}\label{twhkogprthtrerge}
In this section we construct a functor $$\esGght:G\nCalg_{\sepa}\to \EsGght$$ in analogy to the constructions  of the $\infty$-categorical version of $KK$-theory  in \cite{LN}  and equivariant $KK$-theory 
in \cite{KKG} as the Dwyer-Kan localization of $G\nCalg_{\sepa}$ at the classical $E$-theory equivalences.
 In order to distinguish the functor constructed in the present section from the functors constructed in the subsequent sections \cref{twhkogprthtrerge1} and \cref{irthjzgjdiogjhjdicnmjksdisjf} we will momentarily  add the subscript $GHT$ in order to make clear that its construction is based on the work of Guentner, Higson and Trout  \cite{Guentner_2000}.

Recall that $$\esGn:G\nCalg_{\sepa}\to  \EsGn $$ denotes the classical equivariant $E$-theory functor introduced in \cite[Def. 6.8]{Guentner_2000}.  A morphism $f:A\to B$ in $G\nCalg_{\sepa}$  is called a classical $E$-theory equivalence if $\esGn (f):\esGn(A)\to  \esGn(B)$ is an isomorphism in $\EsGn$.

\begin{ddd}\label{wetrokgpwergfrefwerfwrefgsf}
We define $$\esGght:G\nCalg_{\sepa}\to  \EsGght$$ as the Dwyer-Kan localization
of $G\nCalg_{\sepa}$ at the classical $E$-theory equivalences.
\end{ddd}

{By definition, pull-back along $\esGght$ induces for any $\infty$-category $\bC$ an equivalence
 \begin{equation}\label{gqjiofwefqwefqwefqew}(\esGght)^{*}:\Fun(\EsGght,\bC)\stackrel{\simeq}{\to} \Fun^{W_{\mathrm{GHT}}}(G\nCalg_{\sepa},\bC)\ ,
\end{equation}
where $W_{\mathrm{GHT}}$ stands for the full subcategory of functors which send classical $E$-theory equivalences to equivalences.}
In \cref{gjerogrgesrgsg} below we will show that   $\esGght$ {also} has the universal property required by \cref{wrejgioweferfw}.
 
Note that the classical $E$-theory functor $\esGn$ is homotopy invariant  (by construction) and $K_{G}$-stable
(see  \cite[Prop. 6.10]{Guentner_2000}).
\begin{kor}
The functor $\esGght
$ is homotopy invariant and  $K_{G}$-stable. \end{kor}
By the universal property of the functor $\esGght$ we get the dashed arrow  in the following square
  \begin{equation}\label{fweddqwedqed}
 \xymatrix{G\nCalg\ar[r]^-{\esGn}\ar[d]^{\esGght}&\EsGn \\\EsGght\ar[r]^{\ho}\ar@{-->}[ur]&\ho \EsGght\ar@{..>}[u]_{\phi} }
\end{equation} such the corresponding upper triangle commutes. {Since {taking the} the homotopy category is part of the adjunction
\[ 
	\ho: \Cat \rightleftarrows \Cat_1 :\iota \ ,
\]
where $\iota$ is the inclusion of ordinary categories into all categories and  $\EsGn$ is an  ordinary category, we automatically get the dotted arrow  $\phi$ and the other commutative triangle.} Note that $\phi$ is conservative {by construction}.
We will show in \cref{jwegoeferfrefwerf} below that $\phi$ is an equivalence.  

In order to discuss  finite  limits in $\EsGght$  we closely follow \cite[Sec. 2]{KKG} or \cite[Sec. 3]{keb}.  First of all, the category $G\nCalg_{\sepa}$ is pointed and admits all finite limits. In order to descend these limits to the localization $\EsGght$ we use that $G\nCalg_{\sepa}$
 has  the structure of a category of fibrant objects \cite{Brown_1973}, see also \cite[Def. 1.1]{Uuye:2010aa}, \cite[Def. 2.7]{KKG}, or  \cite[Def 7.4.12  and 7.5.7]{Cisinski:2017} for an $\infty$-categorical version) such that:
\begin{enumerate}
\item The fibrations are the surjective homomorphisms.
\item The weak equivalences are the classical $E$-theory equivalences. 
\end{enumerate}
A similar structure for the classical $KK$-theory equivalences has first been considered in \cite{Uuye:2010aa}.
The existence of the  structure of a category of fibrant objects  implies that $\esGght$ precisely inverts the classical $E$-equivalences.   
Using  \cite[Prop. 7.5.6]{Cisinski:2017}  we can further conclude  (see also \cite[Prop. 3.16]{keb} for similar arguments):
 \begin{kor}\label{wergjiowerferfwefwref}
The category $\EsGght$ is pointed, left-exact and the functor $\esGght$ is exact. Moreover, every fibre sequence in $\EsGght$ is  induced up to equivalence by a short exact sequence in $G\nCalg_{\sepa}$.
\end{kor}

\begin{prop}\label{jwegoeferfrefwerfst}
The $\infty$-category $\EsGght$ is stable.   \end{prop}
\begin{proof}
For the moment let $F:G\nCalg_{\sepa}\to \bC$ be any exact and homotopy invariant functor.
For $A$ in $G\nCalg_{\sepa}$   the suspension $S(A)$  fits into an  exact sequence
$$\xymatrix{S(A)\ar[r]\ar[d]&C_{0}([0,1),A)\ar[d]^{\ev_{0}}\\ 0\ar[r]&A}\ .$$
Since $\ev_{0}$ is surjective, using the  exactness of $F$  we get the  cartesian square
 $$\xymatrix{F(S(A))\ar[r]\ar[d]&F(C_{0}([0,1),A)) \ar[d]^{\ev_{0}}\\ 0\ar[r]&F(A)}\ .$$
 By the homotopy invariance of $F$ we have an equivalence  $F(C_{0}([0,1),A))\simeq 0$ and the square  above provides an equivalence
 \begin{equation}\label{adsvacadqwe}
 F(S(A))\simeq \Omega F(A)\ .
\end{equation} 
  Let \begin{equation}\label{cjdsiocjoiejfoidsaj}
0\to I\to A\stackrel{f}{\to} B\to 0
\end{equation}   be an exact sequence in $G\nCalg_{\sepa}$. Then using the mapping cone construction we get a homomorphism of short exact sequences   \begin{equation}\label{fqwehwiedqwedwoediwqodi}
\xymatrix{0\ar[r]&I\ar[d]^{\iota_{f}}\ar[r]&A\ar[d]^{\sim}_{h}\ar[r]^{f}&B\ar@{=}[d]\ar[r]&0\\0\ar[r]&C(f)\ar[r]&Z(f)\ar[r]^{\tilde f}&B\ar[r]&0\\&S(B)\ar[u]^{\partial_{f}}&&&}\ ,
\end{equation}
where  $h$ is a homotopy equivalence and we added the map $\partial_{f}$ which is part of the Puppe sequence associated to $f$  (see e.g.  \cite[(5.1)]{keb} for notation and more details).
We apply $F$ and get the map
\begin{equation}\label{fqwefqwdsafsdfdsf}
\xymatrix{ \Omega F(B) \ar@{=}[d]\ar[r]^{d_{f}}&F(I)\ar[d]_{\simeq}^{F(\iota_{f})}\ar[r]&F(A)\ar[d]^{\simeq}_{F(h)}\ar[r]^{F(f)}&F(B)\ar@{=}[d]\\\ar@{-}[dr]_{\eqref{adsvacadqwe}}^{\simeq}\Omega F(B)\ar[r]  &F(C(f))\ar[r]&F(Z(f))\ar[r]^{F(\tilde f)}&F(B)\\&F(S(B))\ar[u]_{F(\partial_{f})}&&}\ .
\end{equation} 
of  horizontal  fibre sequences. The   commutativity of the left  triangle is given by the usual Puppe sequence argument showing
that  $F (\partial_{f})$  represents the fibre
of $F(C(f))\to F(Z(f))$.

We now specialize to the homotopy invariant and exact functor $F=\esGght$.
  We can conclude that $\phi(d_{f})$ (with $\phi$ as in \eqref{fweddqwedqed}) induces the boundary  maps in the long exact sequences in classical $E$-theory \cite[Thm. 6.20]{Guentner_2000}.

Following \cite[Sec. 4]{MR750677} we use the  Toeplitz extension 
\begin{equation}\label{eqwdqwewefqewfewfq}
0\to K\to \cT_{0}\stackrel{\pi}{\to} S(\C)\to 0
\end{equation}  in order to  introduce Bott periodicity  (see e.g. \cite[(6.1), (6.3) and  (6.4)]{keb} for details).
We  form tensor product of the Toeplitz extension  with $A$ in $G\nCalg$ and apply $\esGght$ in order to get a fibre sequence \begin{equation}\label{ewffdfsvgdfgser}
\Omega^{2} \esGght(A)\stackrel{\beta_{A}}{\to} \esGght(A)\to \esGght(A\otimes \cT_{0})\xrightarrow{\esGght(\id_{A}\otimes \pi)} \Omega \esGght(A)\ ,
\end{equation}
where $\beta_{A}:=d_{\id_{A}\otimes \pi}$ is an instance of $d_{f}$ in \eqref{fqwefqwdsafsdfdsf}.
Here we used   $K_{G}$-stability in order to remove tensor factors $K$, the isomorphism $\C\otimes A\cong A$ and the equivalences $\Omega^{2}\esGght(A)\simeq \Omega \esGght(S(A))$ and \eqref{adsvacadqwe} in order to rewrite the terms appropriately.  By  \eqref{fqwefqwdsafsdfdsf} we know that $\phi(\beta_{A})\simeq \esGn(\partial_{\id_{A}\otimes \pi})$ and therefore coincides with the Bott periodicity isomorphism  from \cite[Prop. 6.16]{Guentner_2000} (which is based on \cite[Sec. 4]{MR750677}). Since $\phi$ is conservative we conclude that   $\beta_{A}$ is an equivalence.
The family $(\beta_{A})_{A\in G\nCalg_{\sepa}}$ implements an equivalence of endofunctors $\beta:\Omega^{2}\stackrel{\simeq}{\to} \id_{\EsGght}$ of the left-exact $\infty$-category $\EsGght$
which shows that $\EsGght$ is a stable $\infty$-category.
\end{proof}

\begin{prop}\label{jwegoeferfrefwerf}
The functor  $\phi:\ho\EsGght\to \EsGn$ in \eqref{fweddqwedqed} is an equivalence.
\end{prop}\begin{proof}
The functor $\ho\circ \esG:G\nCalg_{\sepa}\to \ho\EsGght$ is a  homotopy invariant and $K_{G}$-stable.
Since $\esGght$ is exact and $\EsGght$ is stable by \cref{jwegoeferfrefwerfst} the composition $\ho\circ \esGght$
 is a half-exact functor to an additive category. In view of the universal property of $\esGn$
formulated in the introduction of \cite{Guentner_2000}  we get an inverse functor $\psi:\EsGn\to \ho\EsGght$ such that
$$\xymatrix{&G\nCalg_{\sepa}\ar[dl]_{\ho\circ \esGght}\ar[dr]^{\esGn}\\\EsGn\ar[rr]^{\psi}&&\ho \EsGght}$$
commutes. Using the universal properties of both functors $\esGn$ and $\ho\circ \esGght$ we can then conclude that
$\phi$ and $\psi$ are inverse to each other.
 \end{proof}

\begin{lem}\label{rokgpgwergwerg9}
The stable $\infty$-category $\EsGght$ is countably cocomplete and the functor $\esGght$  is countable sum preserving.
\end{lem}
\begin{proof}
\cite[Prop. 7.1]{Guentner_2000} states that the functor 
$$\esGn :G\nCalg\to \EsGn$$ sends countable sums to coproducts. 
Hence by \cref{jwegoeferfrefwerf}  the same applies to $\ho\circ \esGght$. Since $\ho$ is conservative and preserves  coproducts
we conclude that $\esGght$ also sends countable sums to coproducts.
In particular, $\EsGght$ admits countable coproducts. In fact, they   are represented by countable sums  in $G\nCalg_{\sepa}$. Since $\EsGght$ is stable  the existence of countable coproducts  implies that $\EsGght$ is countably cocomplete.
 \end{proof}

\begin{prop} \label{gjerogrgesrgsg}The precomposition by $\esGght$ induces for every countably cocomplete stable $\infty$-category $\bC$ an equivalence
\begin{equation}\label{fqewiufhiuefhiuewfqwefqwefqwddqdwd}(\esGght)^{*} :\Fun^{\colim_{\omega}}( \EsGght ,\bC)\stackrel{\simeq}{\to} \Fun^{h,K_{G},\exa, \oplus}( G\nCalg_{\sepa},\bC)\ .
\end{equation} 
\end{prop} 
\begin{proof} The argument is similar to the proof of \cite[Prop. 2.23]{KKG}.
By definition the precomposition by $\esGght$ induces for any $\infty$-category $\bC$ an equivalence
$$\Fun( \EsGght ,\bC)\stackrel{\simeq}{\to} \Fun^{W_{{\mathrm{GHT}}}}( G\nCalg_{\sepa},\bC)\ ,$$
where the superscript $W_{{\mathrm{GHT}}}$ stands for functors which invert the classical $E$-theory equivalences.
We now assume that $\bC$ is stable. We claim that there is an inclusion 
\begin{equation}\label{rtokpgpergertgetrg9}
\Fun^{\mathrm{h},K_{G},\exa}( G\nCalg_{\sepa},\bC)\subseteq  \Fun^{{W_\mathrm{GHT}}}( G\nCalg_{\sepa},\bC)\ .
\end{equation}
 So assume that $F:G\nCalg_{\sepa}\to \bC$ is homotopy invariant, $K_{G}$-stable and exact. Then $\ho\circ F$ is a homotopy invariant, $K_{G}$-stable and half-exact functor  to the additive category $\ho\bC$. In view of the universal property of $\esGn:G\nCalg_{\sepa}\to \EsGn$  formulated in the introduction of \cite{Guentner_2000} 
 we have a factorization 
 $$\xymatrix{G\nCalg_{\sepa}\ar[dr]_{\esGn}\ar[r]^-{F}&\bC\ar[r]^-{\ho}&\ho\bC\\&\EsGn\ar@{..>}[ur]}\ .$$
In particular, $\ho\circ F$ sends classical $E$-theory equivalences to isomorphisms. Since $\ho$ is conservative, the functor $F$ sends classical $E$-theory equivalences to  equivalences. This shows the claim.

We now build the following diagram
\begin{equation}\label{gegergsfgfdgsdfgsdfgsfdg}
\xymatrix{\Fun^{\colim_{\omega}}( \EsGght,\bC)\ar@{..>}[r]_-{\simeq}^-{(\esGght)^{*}}\ar[d]&\Fun^{\mathrm{h},K_{G},\exa,\oplus}( G\nCalg_{\sepa},\bC) \ar[d]\\\Fun^{\mathrm{ex}}( \EsGght,\bC) \ar[d]\ar@{-->}[r]_-{\simeq}^-{(\esGght)^{*}}&\Fun^{\mathrm{h},K_{G},\exa}( G\nCalg_{\sepa},\bC) \ar[d]^{\eqref{rtokpgpergertgetrg9}}\\\Fun( \EsGght ,\bC)\ar[r]_-{\simeq}^-{(\esGght)^{*}}&\Fun^{{W_\mathrm{GHT}}}( G\nCalg_{\sepa},\bC)}\ .
\end{equation}
For the lower square we only need that $\bC$ is a stable $\infty$-category.
Since  $\esG$ is exact the restriction of $(\esGght)^{*}$ to exact functors induces the dashed  functor which is fully faithful by construction. Using 
the last statement of  \cref{wergjiowerferfwefwref} one checks that it is essentially surjective.

If $\bC$ is in addition 
countably cocomplete, then  we also get the upper square.
Since $\esGght$ is countable sum {preserving}    we get the further restriction to a fully faithful functor  represented by  the dotted arrow. Since countable coproducts in $\EsGght$ are represented by countable sums in $G\nCalg_{\sepa}$ we conclude that the dotted arrow is also essentially surjective.
\end{proof}

 {
\begin{prop}
The stable $\infty$-category $\EsGght$ has a canonical  bi-exact symmetric monoidal structure
 and  $\esGght:G\nCalg\to \EsGght$ canonically  refines to a   symmetric monoidal functor.
The tensor product on $\EsGght$ preserves countable colimits in each variable.
Pull-back along the symmetric monoidal refinement of $\esGght$ induces an equivalence 
\begin{equation}\label{fasdfoiuqiofafdfffs2}(\esGght)^{*} :\Fun_{\otimes/\mathrm{lax}}^{\colim_{\omega}}(\EsGght,\bC)\stackrel{\simeq}{\to} \Fun_{\otimes/\mathrm{lax}}^{h,K_{G},\exa,\oplus }(G\nCalg_{\sepa},\bC)\ ,
\end{equation}  
\end{prop}
\begin{proof}
  One checks that $\esGght$ from \cref{wetrokgpwergfrefwerfwrefgsf}
 is a symmetric monoidal Dwyer-Kan localization. 
 To this end we  observe  that for every $A$ in $G\nCalg_{\sepa}$ the functor $A\otimes- :G\nCalg_{\sepa}\to G\nCalg_{\sepa}$ preserves classical $E$-theory equivalences.
 This follows from \cite[Thm. 6.21 ]{Guentner_2000} stating that the maximal tensor product
 descends to the classical equivariant $E$-theory. 
 Since $\esGght\circ (A\otimes -):G\nCalg_{\sepa}\to \EsGght$ is homotopy invariant, $K_{G}$-stable,
 exact and  countable sum preserving  we can conclude that the tensor product on $\EsGght$ is bi-exact and preserves countable coproducts. It therefore preserves all countable colimits.
  This argument is similar to the corressponding argument for  
  equivariant $KK$-theory \cite[Prop. 2.20]{KKG}. 
   \end{proof}}

\begin{rem}\label{jiogpgergrwgw9}
In \cref{xvbnfiohregiohsdfpojsdf} we provide an alternative arguments for \cref{jwegoeferfrefwerf} and \cref{gjerogrgesrgsg}
which do not rely on the universal property of $\esGn$ as stated in  \cite{Guentner_2000}. The careful reader may notice that the latter is stated in the introduction of  \cite{Guentner_2000} but  the details of its verification are    left    to the reader.
\hB
\end{rem}

 \subsection{Construction of ${\EsGs}$  by forcing the universal properties}\label{twhkogprthtrerge1}
In this section we describe {the} construction of {a} functor \begin{equation}\label{gerfwerfsfvf}\esGs:G\nCalg_{\sepa}\to \EsGs
\end{equation}
by forcing the universal property required in \cref{wrejgioweferfw}. We  follow the line of argument of \cite{keb}  in which the   non-equivariant case is worked out. Our main motivation for this second construction is to provide a purely homotopy theoretic  description of equivariant $E$-theory which is independent of \cite{Guentner_2000}.

{We define  the functor in \eqref{gerfwerfsfvf} by}   the following chain of symmetric monoidal Dwyer-Kan localizations
\begin{align}
\hesG:G\nCalg_{\sepa}&\xrightarrow{L_{h}}  G\nCalg_{\sepa,h}\xrightarrow{R_{K_{G}}  L_{\hat K_{G}}} S_{K_{G}}G\nCalg_{\sepa,h}
\label{weokgpwerfferwfwerf}\\ & \hspace{2cm}\xrightarrow{L_{\exa{,\oplus}}}  S_{K_{G}}G\nCalg_{\sepa,h,\exa{,\oplus}} \xrightarrow{\Omega^{2}}  { \EsGs}
\nonumber \end{align}
 \setlist[enumerate]{itemsep=-0cm} \begin{enumerate}
 \item \label{rgiwejogerfwef} The functor $L_{h}$ is the symmetric monoidal  Dwyer-Kan localization at the homotopy equivalences, see \cref{okprthrtertegtrgrg}. The resulting $\infty$-category $G\nCalg_{\sepa,h}$ is left-exact, pointed and has countable coproducts which are represented by countable free products  of  separable $G$-$C^{*}$-algebras. The functor $L_{h}$ sends Schochet fibrant squares  to cartesian squares {and preserves countable coproducts.} The arguments are the same as in the non-equivariant case discussed in  \cite[Sec. 3]{keb}.
 \item \label{kohpjktrpgretgert} The functor $R_{K_{G}}L_{\hat K_{G}}$ is the symmetric monoidal  Dwyer-Kan localization at the $K_{G}$-equivalences, see \cref{fjioqwefdqwedewdqwed}.\ref{fiqewpoqdqded1}.  It is the composition of a left and a right Bousfield localization which we will discuss in detail in  \cref{tgokptegerferfwefwf}. The resulting $\infty$-category $ S_{K_{G}}G\nCalg_{\sepa,h}
$ is left-exact and semi-additive. {It furthermore   admits countable coproducts which are represented by countable free products  of algebras, but also by countable sums. Since the tensor product with a separable algebra
preserves countable sums,
  $R_{K_{G}}L_{\hat K_{G}}$ preserves countable coproducts.}
\item\label{wgkowperfewrfwer} The functor $L_{\exa{, \oplus}}$ is the symmetric monoidal  left exact {and countable sum-preserving} Dwyer-Kan localization (see \cref{gferopfwerfwerfsdf}) {generated by} the collection of morphisms
the functor $R_{K_{G}}L_{\hat K_{G}}(L_{h}(\iota_{f}))$ (see  \eqref{fqwehwiedqwedwoediwqodi} for $\iota_{f}$) for all exact sequences  
\eqref{cjdsiocjoiejfoidsaj}. The resulting $\infty$-category $ S_{K_{G}}G\nCalg_{\sepa,h,\exa}$
is left-exact, semi-additive {and admits countable coproducts, and the functor $L_{\exa,\oplus}$ preserves countable coproducts}. The composition
$$G\nCalg_{\sepa}\xrightarrow{L_{\exa,{\oplus}}\circ R_{K_{G}} \circ L_{\hat K_{G}}\circ L_{h}}  S_{K_{G}}G\nCalg_{\sepa,h,\exa{,\oplus}}$$
is homotopy invariant, $K_{G}$-stable, exact  and {sends countable sums to countable coproducts}, see  \cref{fjioqwefdqwedewdqwed}.\ref{fiqewpoqdqded2}.
The arguments are  {similar} as in the non-equivariant case disussed in \cite[Sec. 5]{keb}.
\item The left exact functor $\Omega^{2}$ is the right-adjoint of a symmetric monoidal  right Bousfield localization 
\begin{equation}\label{vefhwiuvecwe}\incl:  {\EsGs} \leftrightarrows S_{K_{G}} G\nCalg_{\sepa,h,\exa{,\oplus}} :\Omega^{2}\ ,
\end{equation} 
where ${ \EsGs}$ is the full subcategory of $S_{K_{G}} G\nCalg_{\sepa,h,\exa{,\oplus}}$  of group objects (see  \cref{hfifsfgsdfg}). The   $\infty$-category {$ \EsGs$} is stable
and the  composition
\begin{equation}\label{vasdvadscdscasdcasdcasc}
 {\esGs}: G\nCalg_{\sepa}\xrightarrow{\Omega^{2}\circ L_{\exa{,\oplus}}\circ R_{K_{G}}\circ L_{\hat K_{G}}\circ L_{h}} {\EsGs}
\end{equation} 
is homotopy invariant, $K_{G}$-stable, {exact, and countable sum-preserving.}
The arguments are the same as in the non-equivariant case \cite[Sec. 7]{keb}, and the functor \eqref{vasdvadscdscasdcasdcasc} is the {countable sum-preserving} analogue of the $E$-theory functor $\es :\nCalg_{\sepa}\to \EE_{\sepa}$ from \cite[Def. 7.9]{keb}. 
  \end{enumerate}

\begin{rem}\label{tgokptegerferfwefwf} {In this long remark we show the assertions stated in  \cref{kohpjktrpgretgert} of the list above.} The details of the localization $R_{K_{G}}L_{\hat K_{G}}:G\nCalg_{\sepa,h}\to S_{K_{G}}\nCalg_{\sepa,h}$  at the $K_{G}$-equivalences differ slightly from
the non-equivariant case considered in \cite[Sec. 4]{keb}.      
The arguments rely on the following analytic fact: If $H\to H'$ is an isometric inclusion of $\infty$-dimensional separable Hilbert spaces, then the induced corner inclusion $K(H)\to K(H')$ is a homotopy equivalence.

\begin{lem}\label{gtrbergtrergetrgertg}
If  $H\to H'$ is an equivariant isometric inclusion of {non-zero}   separable Hilbert spaces with unitary  $G$-action, then
the map $K(H)\to K(H')$ is a $K_{G}$-equivalence.
\end{lem}
\begin{proof}
If $H$ is a Hilbert space with a unitary $G$-action $\rho$, then we have an equivariant unitary  isomorphism
$L^{2}(G)\otimes H\cong L^{2}(G)\otimes \underline{H}$ sending $\delta_{g}\otimes h$ to
$\delta_{g}\otimes \rho(g^{-1})h$, where $\delta_{g}$ in $L^{2}(G)$ is the basis element corresponding to $g$ in $G$ and $\underline{H}$ denotes $H$ with the trivial $G$-action. 

Assume now that $H\to H'$ is an equivariant isometric inclusion of   separable Hilbert spaces with unitary  $G$-action. The isomorphisms above and the isomorphism $K_{G}\cong K(L^{2}(G))\otimes K(\ell^{2})$ induce the vertical isomorphisms
in $$\xymatrix{K_{G}\otimes K(H)\ar[r]\ar[d]^{\cong} &K_{G}\otimes K(H') \ar[d]^{\cong} \\K(L^{2}(G))\otimes K(\ell^{2}\otimes \underline{H}) \ar[r] &K(L^{2}(G))\otimes K(\ell^{2}\otimes \underline{H'}) } \ .$$ Since the corner inclusion $K(\ell^{2}\otimes \underline{H})\to K(\ell^{2}\otimes \underline{H'})$ is a homotopy equivalence, we see that 
   $K_{G}\otimes K(H)\to K_{G}\otimes K(H')$ is a homotopy equivalence
and $K(H)\to K(H')$ is a $K_{G}$-equivalence. 
\end{proof}

\begin{kor} The functor $R_{K_{G}}L_{\hat K_{G}}$ inverts the {maps} 
$K(H)\otimes A\to K(H')\otimes A$ for any  $A$ in $G\nCalg_{\sepa}$.  
\end{kor}

In the non-equivariant case, we have a left upper corner inclusion $\epsilon: \C\to K:=K(\ell^{2})$ sending $1$ in $\C$ to a one-dimensional projection in $K$.   Then the stabilized corner inclusion
$\epsilon\otimes \id_{K}:K\cong \C\otimes K\to K\otimes K$ is a homotopy equivalence. This turns $K$ into an idempotent algebra in $\nCalg_{\sepa,h}$. By the general theory \cite[Sec, 4.8.2]{HA}, the idempotent algebra $K$ gives rise to a symmetric monoidal  left Bousfield localization  $$ L_{K}:\nCalg_{\sepa,h}\leftrightarrows L_{K} \nCalg_{\sepa,h}:\incl$$  {with} left-adjoint $L_{K}:=K\otimes -$.   In the equivariant case for compact  $G$,  using the subspace of $G$-invariant functions on $G$, there still  exists a left upper corner inclusion
$\C\to K_{G}$ and we again get a left Bousfield localization
$$ L_{K_{G}}:G\nCalg_{\sepa,h}\leftrightarrows L_{K_{G}} G\nCalg_{\sepa,h}:\incl\ .$$

If $G$ is non-compact, then we {must} proceed in two steps. 
We first consider the corner inclusion \begin{equation}\label{sdfvewrvsfdvsfdvsfv}\epsilon:\C\to \hat K_{G}:=K((L^{2}(G)\oplus \C)\otimes \ell^{2} )
\end{equation} 
sending $1$ to the one-dimensional $G$-invariant projection in $\hat K_{G}$ onto the subspace generated by $
(0\oplus 1)\otimes e_{0}$, where $e_{0}$ is the first basis vector of $\ell^{2}$.

\begin{lem}\label{hertgetrgreteg} The map $$\epsilon \otimes \id_{\hat K_{G}}:\hat K_{G}\cong \C\otimes \hat K_{G}\to \hat K_{G}\otimes \hat K_{G}$$ is a homotopy equivalence.
\end{lem}
\begin{proof}

One can find a commutative diagram of horizontal equivariant unitary inclusions of Hilbert spaces
$$\xymatrix{(L^{2}(G)\oplus \C)\otimes \ell^{2} \ar[rrr]^-{x\mapsto  ((0{\oplus }1)\otimes e_{0})\otimes  x }\ar@{=}[d] &&&((L^{2}(G)\oplus \C)\otimes \ell^{2} ) \otimes ((L^{2}(G)\oplus \C)\otimes \ell^{2} ) \ar[d]^{\cong} \\(L^{2}(G)\oplus \C)\otimes \ell^{2} \ar[rrr] &&&(L^{2}(G)\oplus \C)\otimes H } $$
where the lower horizontal map is induced by an isometric inclusion $\ell^{2}\to H$ of separable Hilbert spaces with trivial $G$-action. We can conclude that the map
 $\epsilon \otimes \id_{\hat K_{G}}:\hat K_{G}\cong \C\otimes \hat K_{G}\to \hat K_{G}\otimes \hat K_{G}$   is
 isomorphic to the induced map $K(L^{2}(G)\oplus \C)\otimes K(\ell^{2})\to K(L^{2}(G)\oplus \C)\otimes K(H)$
  and therefore a homotopy equivalence. 
  \end{proof}
  It follows from \cref{hertgetrgreteg} that $\hat K_{G}$  is an idempotent algebra in $G\nCalg_{\sepa,h}$. 
 \begin{kor} 
   We have a symmetric monoidal left Bousfield localization
\begin{equation}\label{vwervwevdfvsfvwevfvs}L_{\hat K_{G}}:G\nCalg_{\sepa,h}\leftrightarrows L_{\hat K_{G}}G\nCalg_{\sepa,h}:\incl
\end{equation} 
with $L_{\hat K_{G}}=\hat K_{G} \otimes -$. \end{kor}

Note that $\hat K_{G}$ represents the tensor unit of    $L_{\hat K_{G}}G\nCalg_{\sepa,h}$.

We now consider the corner inclusion $\alpha:K_{G}\to \hat K_{G}$ induced by the inclusion $L^{2}(G)\to L^{2}(G)\oplus \C$.  {By \cref{gtrbergtrergetrgertg},} we know that $$\alpha\otimes \id_{K_{G}}:K_{G}\otimes K_{G}\to \hat K_{G}\otimes K_{G}$$ is a homotopy equivalence and therefore an equivalence in $L_{\hat K_{G}}G\nCalg_{\sepa,h}$. Thus
$K_{G}$ becomes an idempotent coalgebra  in $L_{\hat K_{G}}G\nCalg_{\sepa,h}$.  
\begin{kor} 
 We have   a symmetric monoidal right Bousfield localization
$$\incl:  S_{  K_{G}}G\nCalg_{\sepa,h}\leftrightarrows   L_{\hat K_{G}}G\nCalg_{\sepa,h}:R_{K_{G}}$$
with right-adjoint $R_{K_{G}}:=K_{G}\otimes  -$.
\end{kor} 

We now consider the composition
of Dwyer-Kan localizations
\begin{equation}\label{sfdvfsdvwerv}R_{K_{G}} L_{\hat K_{G}}:G\nCalg_{\sepa,h}\to   S_{K_{G}}  G\nCalg_{\sepa,h}\ .
\end{equation} 

\begin{lem}  The composition $R_{K_{G}} L_{\hat K_{G}}$    precisely inverts the $K_{G}$-equivalences.  \end{lem} 
\begin{proof}
If $f:A\to B$  is  a  $K_{G}$-equivalence in $G\nCalg_{h}$, then by definition $K_{G}\otimes  f$  is a homotopy equivalence.
But then  also 
$\hat K_{G}\otimes K_{G}\otimes f$  is a homotopy equivalence and $R_{K_{G}}L_{\hat K_{G}}(f)$ is an equivalence. 

If $ R_{K_{G}}L_{\hat K_{G}}(f)$ is an equivalence, then $\hat K_{G}\otimes K_{G}\otimes f$
is a homotopy equivalence. Since also
$ \alpha\otimes\id_{K_{G}}$ is a homotopy equivalence, we conclude from 
$$\xymatrix{K_{G}\otimes A\ar[rr]^{\alpha\otimes \id_{K_{G}\otimes A}}\ar[d]^{K_{G}\otimes f} &&\hat K_{G}\otimes K_{G}\otimes A \ar[d]^{\hat K_{G}\otimes K_{G}\otimes f} \\ K_{G}\otimes A\ar[rr]^{\alpha\otimes \id_{K_{G}\otimes B}} && \hat K_{G}\otimes K_{G}\otimes B} $$
 that
$K_{G}\otimes f$ is a homotopy equivalence.
Consequently, $f$ is a $K_{G}$-equivalence.
\end{proof}

\begin{lem}
The category $S_{K_{G}}G\nCalg_{\sepa,h}$ is semi-additive, admits countable sums, and the functor
$R_{K_{G}}L_{\hat K_{G}}\circ L_{h}$ sends countable sums to countable coproducts.
\end{lem}
\begin{proof}
{The composition $$ {G}\nCalg_{\sepa}\xrightarrow{L_{h}}   G\nCalg_{\sepa,h}\xrightarrow{R_{K_{G}}L_{\hat K_{G}}}S_{K_{G}}G\nCalg_{\sepa,h}$$ in particular inverts left-upper corner inclusions $A\to A\otimes K$. This implies, as in \cite[Sec. 4]{keb}, that
it sends the canonical maps $*_{i\in I}A_{i} \to \oplus_{i\in I} A_{i}$ for countable families
$(A_{i})_{i\in I}$ in $G\nCalg_{\sepa}$
to equivalences. This fact for finite $I$ implies the
semi-additivity of $S_{K_{G}}G\nCalg_{\sepa,h}$. It furthermore implies that 
this category admits countable coproducts which are represented by 
countable sums in  $G\nCalg_{\sepa}$, and that the functor 
$R_{K_{G}}L_{\hat K_{G}}\circ L_{h}$ sends countable sums to countable coproducts.} \end{proof} This completes the explanation  of the details for  \cref{kohpjktrpgretgert}.  
 \hB
\end{rem}
 {
\begin{rem}\label{gferopfwerfwerfsdf}
	In order to obtain {the} left exact {and countable sum-preserving Dwyer-Kan} localization {$L_{\exa,\oplus}$} we must form the closure  of the generating set $W$ of morphisms $R_{K_{G}}L_{\hat K_{G}}(L_{h}(\iota_{f}))$  mentioned in \cref{wgkowperfewrfwer}  under 2-out-of-3, pull-backs{, and countable sums.} 
	Then for every left-exact {semi-additive} $\infty$-category $\bC$ admitting countable  coproducts we have by construction the equivalences
	\begin{align}\label{ewfqweijdoqfsfg} \Fun^{\mathrm{lex},{\coprod_{\omega}}}(S_{K_{G}}G\nCalg_{\sepa,h,\exa{,\oplus} },\bC)\xrightarrow[\simeq]{{L_{\exa, \oplus}}^{*}}
	\Fun^{\mathrm{lex},{\coprod_{\omega}}, W}(S_{K_{G}}G\nCalg_{\sepa,h }, \bC) \\ \xrightarrow[\simeq]{(R_{K_G} L_{\hat K_{G}} L_h)^{\ast}} \Fun^{h,K_{G},\exa+\mathrm{Sch}{, \oplus}}(G\nCalg_{\sepa}, \bC)
	\ , 		
	\end{align} 
	where the superscript $\mathrm{Sch}$ stands for Schochet
	exactness  (see \cref{okprthrtertegtrgrg}), the symbol $W$ {for} functors inverting maps in $W$, and $\coprod_{\omega}$  and $\oplus$ indicate functors  preserving countable coproducts, or sending countable sums to countable coproducts, respectively. 
	\hB
\end{rem}
}

\begin{rem}
\label{hfifsfgsdfg}
Since  $ S_{K_{G}}G\nCalg_{\sepa,h,\exa{,\oplus}}$ is a semi-additive $\infty$-category,
 every  object $A$  in this category  is   a commutative monoid. The object $A$ is a  group objects $A$   if  the shear map
$$A\times A\xrightarrow{(a,a')\mapsto (a+a',a)} A\times A$$ is an equivalence. 

 The counit 
$\Omega^{2}\to \id$ of the right Bousfield localization \eqref{vefhwiuvecwe} is induced for every   $A$ in $G\nCalg_{\sepa}$ by the boundary maps $ { \beta_{A}}$ of the fibre sequences (compare with \eqref{ewffdfsvgdfgser})
$$  \Omega^{2}({\esGs}(A))\stackrel{{ \beta_{A}}}{\to} {\esGs}(A) $$   associated to the exact sequences obtained by forming the tensor product of the Toeplitz sequence \eqref{eqwdqwewefqewfewfq} with $A$.
\hB
\end{rem}

{The first assertion in the following  corollary is the equivariant analogue of  \cite[(7.8)]{keb} and an immediate consequence of the construction  and \eqref{ewfqweijdoqfsfg} (also using \cite[Lem. 5.2]{keb} in order to get rid of the Schochet exactness condition). 
The second follows from \cref{weijgwoefrewfwdffsfd}.
For the remaining two assertions
we use that
$\esGs$ is a composition of symmetric monoidal Dwyer-Kan localizations.}

{
 \begin{kor} \label{gjerogrgesrgseg}\mbox{}
 \setlist[enumerate]{itemsep=-0.4cm}\begin{enumerate}
\item \label{wokegwpergerwgwreg3}
The precomposition by $\esGs$ induces for every countably cocomplete stable $\infty$-category $\bC$ an equivalence
\begin{equation}\label{rferwiufherufihweriufhwreiufhi}(\esGs)^{*} :\Fun^{\colim_{\omega}}( \EsGs  ,\bC)\stackrel{\simeq}{\to} \Fun^{h,K_{G},\exa, \oplus}( G\nCalg_{\sepa},\bC)\ .
\end{equation}  
\item\label{kophertherth9} $\EsGs$ admits countable colimits and $\esGs$ preserves countably filtered colimits.
\item The stable $\infty$-category $\EsGs$ has a canonical  bi-exact symmetric monoidal structure preserving countable colimits in each variable and  $\esGs:G\nCalg\to \EsGs$ canonically  refines to a   symmetric monoidal functor.
 \item\label{orkewgpoerfwerfwref} {
 Pull-back along the symmetric monoidal refinement of $\esGs$ induces an equivalence 
\begin{equation}\label{fasdfoiuqiofafdfffs1}(\esGs)^{*} :\Fun_{\otimes/\mathrm{lax}}^{\colim_{\omega}}(\EsGs,\bC)\stackrel{\simeq}{\to} \Fun_{\otimes/\mathrm{lax}}^{h,K_{G},\exa,\oplus }(G\nCalg_{\sepa},\bC)\ .
\end{equation}  }
\end{enumerate}
\end{kor}}

We next prepare the argument showing that $\esGs$ is   itself  a Dwyer-Kan localization.
Let $$\bA\stackrel{\Phi}{\to} \bB\stackrel{\Psi}{\to} \bC$$ be a sequence of Dwyer-Kan localizations of $\infty$-categories.
In general we do not expect that the composition $\Psi\circ \Phi:\bA\to \bC$ is a Dywer-Kan localization. But the following simple Lemma provides a sufficient condition for this to be the   case.
\begin{lem}\label{okpgwerwreg9}
If $\Phi$ is essentially surjective on morphisms, then
$\Psi\circ \Phi:\bA\to \bC$ is a Dwyer-Kan localizations.
\end{lem}
\begin{proof}
Let $W$ be the class of morphisms which are inverted by $\Psi\circ \Phi$.
Then we want to show that this composition is the Dwyer-Kan localization at $W$. Let $\Theta:\bA\to \bD$ 
be the Dwyer-Kan localization at $W$. Then by the universal property
of $\Theta$ we have a triangle
$$\xymatrix{&\bA\ar[ddl]_{\Theta}\ar[dr]^{\Phi}&&\\&&\bB\ar[dr]^{\Psi}&\\ \bD\ar@{..>}[rrr]^{\Delta}&&&\bC}\ .$$
In order to construct an inverse of $\Delta$ we consider 
$$\xymatrix{&\bA\ar[ddl]_{\Theta}\ar[dr]^{\Phi}&&\\&&\ar@{-->}[dll]^{\Sigma}\bB\ar[dr]^{\Psi}&\\ \bD &&&\ar@{..>}[lll]_{\Gamma}\bC}\ .$$
Since $\Theta$ inverts all morphisms which are inverted by $\Phi$ we get the dashed arrow $\Sigma$ from the universal property of $\Phi$. We now use the universal property of $\Psi$ in order to get the dotted arrow. 
Let $f$ be a morphism in $\bB$ which is inverted by $\Psi$. 
We must show that $\Sigma(f)$ is an equivalence.
 By our assumption $f\simeq \Phi(\tilde f)$ for some morphism in
$\bA$. Since $\Psi(f)\simeq\Psi(\Phi(\tilde f))$ is an equivalence, so is $\Theta(\tilde f)\simeq \Sigma(f)$.
We therefore get the dotted arrow. In order to show $\Delta$ and $\Gamma$ are inverse to each other we 
argue with the following diagrams:
$$\xymatrix{&\bA\ar[ddl]_{\Theta}\ar[ddr]^{\Theta}\ar[d]^{\Phi}&\\&\bB\ar[d]^{\Psi}&\\\bD\ar[r]^{\Gamma}&\bC\ar[r]^{\Delta}&\bD}\ , \quad  
\xymatrix{&\bA\ar[dd]_{\Theta}\ar[dr]^{\Phi}\ar[dl]_{\Phi}&\\\bB\ar[d]^{\Psi}\ar@{-->}[rd]^{\Sigma}&&\bB\ar[d]_{\Psi}\\\bC\ar[r]^{\Gamma}&\bD\ar[r]^{\Delta}&\bC}\ .
$$
The left diagram implies   by the universal property of $\Theta$ that 
the composition $\bD\to \bC\to \bD$ is equivalent to the identity of $\bD$.
In the right diagram we first use the universal property of $\Phi$ in order to see that 
  $\Psi\simeq \Delta\circ \Sigma$. We then use the universal property of
  $\Psi$ in order to conclude that $\Delta\circ \Gamma$ is 
 is equivalent to the identity of $\bC$.
 \end{proof}

The following results helps to check the
assumption of \cref{okpgwerwreg9}

\begin{lem}\label{thrkoprthetrherg}
If $\Phi:\bA\to \bB$ is a left-exact Dwyer-Kan localization between left-exact $\infty$-categories, then
$\Phi$ is essentially surjective on morphisms.
\end{lem}
\begin{proof}
The class $W$ of morphisms inverted by $\Phi$ satisfies a calculus of fractions.
\end{proof}

\begin{prop}\label{kopghrthrhertegtrg}
The functor $\esGs:G\nCalg_{\sepa}\to  \EsGs$ is a symmetric monoidal Dwyer-Kan localization.
\end{prop}
\begin{proof}
By construction
$\esGs$ in \eqref{weokgpwerfferwfwerf} is a composition of the Dwyer-Kan localization $L_{h}$ and a sequence of 
left-exact Dwyer-Kan localizations. Furthermore $L_{h}$ is obviously 
essentially surjective on morphisms.
Using   \cref{okpgwerwreg9}
 and \cref{thrkoprthetrherg}
 finitely many times we conclude that $\esGs$
 is a Dwyer-Kan localization.
 Since we already know that the tensor product on $\nCalg_{\sepa}$ descends through
 $\esGs$ it is clearly a symmetric monoidal Dwyer-Kan localization. 
\end{proof}

\subsection{Construction of ${\EsGf}$ as the universal finitary invariant} \label{irthjzgjdiogjhjdicnmjksdisjf}
 In this section we  construct {the} functor $$\esGf:G\nCalg_{\sepa}\to \EsGf$$ as a composition
 \begin{equation}\label{wrgwgrefrefrewff} G\nCalg_{\sepa}\xrightarrow{\asGc}  \AsGc \xrightarrow{} \EsGf\end{equation}
 The first  is the universal homotopy invariant and countable filtered colimit preserving  functor (see \cref{uqwighfuihvsdiohjfhjkbvdxc}), and the 
 second is a Dwyer-Kan  localization forcing $K_{G}$-stability and stability {using functors $R_{K_{G}}L_{\hat K_{G}}$ and $\Omega^{2}$  similarly as in} \eqref{weokgpwerfferwfwerf}. But note that we do not enforce exactness. {In fact, exactness   turns}    out to be  an automatic property by \cref{ifqerwfqewdfewdewdqedf}. In \cref{ewrogjpwergwerfwerfwrfw}, using the shape  theory  \cite{Dadarlat_1994},  we  observe   that $ \AsGc$ provides an $\infty$-categorical enhancement of the asymptotic morphism category $\AsG$ introduced in \cite{Guentner_2000}.
 This comparison is relevant for our construction since we want to import the Bott map represented by an asymptotic morphism in $\AsG$ and the verification of its properties  from the classical theory.

We construct the functor $\asGc$ as the composition \begin{equation}\label{543t445ts} \asGc: G\nCalg_{\sepa}\xrightarrow{L_{h}}  G\nCalg_{\sepa, h} \xrightarrow{y} \Ind^{\aleph_{1}}(G\nCalg_{\sepa, h}) \xrightarrow{R} \AsGc\ ,
\end{equation}
 where $L_{h}$ is the Dwyer-Kan localization at the homotopy equivalences (explained in some detail in  \cref{okprthrtertegtrgrg}),
 $y$ has the usual meaning   as  introduced at the beginning of \cref{ergkoperwgwerfwefref}, and $R$ is a Dwyer-Kan localization {which} we  will describe in detail in the following. 
 
 We abbreviate the composition of the first two maps in  \eqref{543t445ts} by     $yL_{h}:=y\circ L_{h}$. For any system $(Z_{n})_{n\in \nat}$ in $G\nCalg_{\sepa}$ we  have a  canonical map
 \begin{equation}\label{wefwefwefqewf}\colim_{n\in \nat} yL_{h}(Z_{n})\to  yL_{h}(\colim_{n\in \nat} Z_{n})
\end{equation}  in $\Ind^{\aleph_{1}}(G\nCalg_{\sepa, h})$. We   denote by
  $W_{R}$  the set of the maps \eqref{wefwefwefqewf} for all  systems $(Z_{n})_{n\in \nat}$ in $G\nCalg_{\sepa}$. Then we define the functor $$R:\Ind^{\aleph_{1}}(G\nCalg_{\sepa, h})\to \AsGc$$ as the Dwyer-Kan localization  
  at $W_{R}$.
 
 We have already seen in \cref{okprthrtertegtrgrg} that $L_{h}$ is a symmetric monoidal Dwyer-Kan localization for the maximal tensor product on $G\nCalg_{\sepa}$. Its target  is a pointed and left-exact symmetric monoidal $\infty$-category 
  with a bi-left exact tensor product.  The countable  ind-completion  $\Ind^{\aleph_{1}}(G\nCalg_{\sepa, h})$
 of a pointed left exact $\infty$-category  is again pointed and left-exact. 
It has countable filtered colimits and these commute with finite limits. Furthermore  the symmetric monoidal structure of $G\nCalg_{\sepa, h}$ induces a bi-left exact  symmetric monoidal structure  on $\Ind^{\aleph_{1}}(G\nCalg_{\sepa, h})$ and a canonical symmetric monoidal refinement
of $y$. This tensor product preserves {countable} filtered colimits in each argument.
These structures are used to interpret \cref{wreig90wregfwerf}.\ref{regkopwergerferfrewfw} below.

 \begin{theorem}\label{wreig90wregfwerf}\mbox{}\begin{enumerate}
  \item \label{gokpwergrwefwrfwref} $R$ is  the right-adjoint of a right Bousfield localization
  \begin{equation}\label{erfwerfewrferwfw}L: \AsGc\leftrightarrows\Ind^{\aleph_{1}}(G\nCalg_{\sepa, h}):R\ .
\end{equation}  \item  \label{gokpwergrwefwrfwref111} The essential image of $L$ is spanned by the objects   $\colim_{n \in \nat} L_{h}(yA_n)$ for all explicit shape systems $(A_n)_{n \in \nat}$.
\item\label{regkopwergerferfrewfw2} The $\infty$-category $\AsGc$ is pointed, left-exact and    has all  countable filtered colimits. {In  $\AsGc$  countable filtered colimits  commute} {with} finite limits. The functor   $R$ preserves countable filtered colimits. 
 \item\label{regkopwergerferfrewfw} $R$ is a symmetric monoidal Dwyer-Kan localization and the induced tensor product on $\AsGc$ is bi-left exact and preserves countable filtered colimits in each argument.
 \end{enumerate}
\end{theorem}
\begin{proof}

In order to show Assertion \ref{gokpwergrwefwrfwref} we use the following well-known criterion to detect Bousfield localizations.

\begin{rem}\label{ijewgowegwergerfref}
Let $R:\bC\to \bD$ be a Dwyer-Kan localization generated by a set $W$ of morphisms   in $\bC$. 
We say that an object $A$ in $\bC$ is $W$-local if the functor $\Map_{\bC}(A,-):\bC\to \Spc$ sends the  morphisms in $W$ to equivalences.  A morphism $\epsilon :A'\to A$ in $\bC$  such that
$R(\epsilon)$ is an equivalence and  $A'$  is $W$-local is called a $W$-local replacement of $A$.

If every object of   $\bC$  admits a $W$-local replacement, then   $R$ is a right Bousfield localization.   Furthermore, for any $W$-local $A$ and any $B$ in $\bC$ the map
\begin{equation}\label{wefqdqsdvdvsdfv}R:\Map_{\bC}(A,B)\stackrel{\simeq}{\to} \Map_{\bD}(R(A),R(B) ) \end{equation} 
is an equivalence. 
If $\bC$ is left exact, then $R$ is automatically a left-exact Dwyer-Kan localization and $\bD$ is again left exact.
Similarly, if $\bC$ is pointed, then so is $\bD$.

Assume that $f:A\to B$ is a morphism in $\bC$ such that $R(f)$ is an equivalence and $\epsilon:B'\to B$ is a $W$-local replacement. Then $\delta:B'\to A$ is a $W$-local replacement 
of $A$, where $\delta$ is a preimage of $\epsilon$ under the equivalence
$f_{*}:\Map_{\bC}(B',A)\stackrel{\simeq}{\to}\Map_{\bC}(B',B)$. \hB
\end{rem}

We first claim that all objects of the form $yL_{h}(A)$ for  some $A$  in $G\nCalg_{\sepa}$  admit $W_{R}$-local replacements. By \cref{ojkergperwerge} we can choose an explicit shape system $(A_{n})_{n\in \nat}$ for $A$. We claim that then  $$\epsilon:\colim_{n\in \nat} yL_{h}(A_{n})\to yL_{h}(A)$$ is 
a $W_{R}$-local replacement. First of all the morphism $\epsilon$ is of the form \eqref{wefwefwefqewf} and therefore belongs to $W_{R}$ and is inverted by $R$. It remains to show that
$\colim_{n\in \nat} yL_{h}(A_{n})$ is $W_{R}$-local. To this end we consider a countable filtered system $(Z_{i})_{i\in I}$
in $G\nCalg_{\sepa}$  and set  $Z:=\colim_{i\in I} Z_{i}$. Then we must show that the map \eqref{wefwefwefqewf}  induces an equivalence 
$$\Map (\colim_{n\in \nat} yL_{h}(A_{n}),\colim_{i\in I} yL_{h}(Z_{i}))\stackrel{\simeq}{\to} \Map (\colim_{n\in \nat} yL_{h}(A_{n}),yL_{h}(Z))\ ,$$
 where we abbreviated $\Map:=\Map_{ \Ind^{\aleph_{1}}(G\nCalg_{\sepa, h})}$.
Expanding definitions this map is equivalent to the map
\begin{equation}\label{fewdwedeqwdqewdqwed}\lim_{n\in \nat^{\op}}\colim_{i\in I} \Map_{G\nCalg_{\sepa,h}} (  A_{n} , Z_{i})\to \lim_{n\in \nat^{\op}}\Map_{G\nCalg_{\sepa,h}}  ( A_{n} ,Z)\ .
\end{equation}
We now use that the map of systems
$$(\colim_{i\in I}  \Map_{G\nCalg_{\sepa,h}} (  A_{n} , Z_{i}))_{n\in \nat^{\op}}\to (\Map_{G\nCalg_{\sepa,h}}  ( A_{n},  Z))_{n\in \nat^{\op}}$$
has the slp by \cref{jgriowergwrefrefrefwfrfwf}.  Therefore  the map \eqref{fewdwedeqwdqewdqwed}
 is an equivalence  by  \cref{bidfosdfivosdifhgnbeoiosdf}. This finishes the proof of the claim. 

A general object of $\Ind^{\aleph_{1}}(G\nCalg_{\sepa, h})$ is equivalent to one of the form
$\colim_{n\in \nat} yL(A_{n})$ for some system $(A_{n})_{n\in \nat}$ in $G\nCalg_{\sepa}$.
Since $\colim_{n\in \nat} yL_{h}(A_{n})\to yL_{h}(\colim_{n\in \nat }A_{n})$ is of the form \eqref{wefwefwefqewf} and is therefore  sent to an equivalence  by $R$, and since $yL_{h}(\colim_{n\in \nat }A_{n})$ has a $W_{R}$-local replacement  by the claim, also $\colim_{n\in \nat} yL_{h}(A_{n})$  has a $W_{R}$-local replacement as explained at the end of  \cref{ijewgowegwergerfref}.
Thus every object of $\Ind^{\aleph_{1}}(G\nCalg_{\sepa, h})$ has a $W_{R}$-local replacement and we can conclude by  \cref{ijewgowegwergerfref} that $R$ is a right Bousfield localization. 
The argument also provides Assertion \ref{gokpwergrwefwrfwref111} by inspection.

We now show Assertion \ref{regkopwergerferfrewfw2}. 
First of all, $\AsGc$ is a right Bousfield localization of the   pointed left exact $\infty$-category $\Ind^{\aleph_{1}}(G\nCalg_{\sepa, h})$  and therefore  {is} again pointed and left-exact.
Since the latter
 has all countable filtered colimits  
we also  know that $\AsGc$ has all countable filtered colimits.
 In order to show that $R$ preserves countable filtered colimits we use that there exists a best countable filtered colimit preserving approximation ${\epsilon}: y_{!}y^{*} R\to R$ of $R$, where $y_{!}$ is the left Kan-extension functor along
$y$ in \eqref{543t445ts}. 
 It suffices to show that this transformation is an equivalence on all objects. Any object of $\Ind^{\aleph_{1}}(G\nCalg_{\sepa,h})$ can be presented in the form
$\colim_{n\in \nat} yL_{h}(A_{n})$ for some system $ (A_{n})_{n\in \nat}$ in $G\nCalg_{\sepa}$.
We then calculate
$$\hspace{-0.6cm}(y_{!}y^{*} R)(\colim_{n\in \nat} yL_{h}(A_{n}))\simeq \colim_{n\in \nat} (y_{!}y^{*} R)(  yL_{h}(A_{n}))
\stackrel{!}{\simeq}  \colim_{n\in \nat}  R(  yL_{h}(A_{n}))\stackrel{?}{\to}   R(\colim_{n\in \nat}  yL_{h}( A_{n}))\ ,$$
where for the equivalence $!$  we use that {$\epsilon_y: (y_! y^\ast R) \circ y \to R \circ y$ is an equivalence, which is true as $y$ is fully faithful.}
It remains to show that the map $?$
is an equivalence. 
By \cref{gjiowergrwfwerfrf} there exists a system $(A_{n,k})_{(n,k)\in \nat\times \nat}$ in $ G\nCalg_{\sepa}$
such that
$(A_{n,k})_{k\in \nat}$ is an explicit shape system for $A_{n}$ for all $n$ in $\nat$.
We then  consider the commutative diagram
 $$\hspace{-1cm}\xymatrix{\colim_{n \in \nat} R(yL_{h}(A_n)) \ar[r]^{?} & R(\colim_{n \in \nat} yL_{h}(A_n)) \ar[r]_{(2)}^{\simeq} & R(yL_{h}( \colim_{n \in \nat} A_n)) \\ \colim_{n \in \nat} R(\colim_{k \in \nat}yL_{h}(A_{n,k})) \ar[u]^{(1)}_{\simeq} \ar[r]^-{\simeq}_-{(5)} & R(\colim_{(n,k) \in \nat^2} yL_{h}(A_{n,k})) \ar[r]_{(4)}^{\simeq} \ar[u] & R(yL_{h}(\colim_{(n,k) \in \nat^2} A_{n,k})) \ar[u]^{(3)}_{\simeq}}  \ .$$
   We must show that the morphism $?$ is an equivalence, but this follows from the diagram once the other equivalences have been justified.
 The morphism (1) is an equivalence since  it is induced by the maps $\colim_{k \in \nat}yL_{h}(A_{n,k})\to
 yL_{h}(A_{n})$ which belong to $W_{R}$ for all $n$ in $\nat$ and {are} therefore sent to equivalences by $R$. A similar reasoning applies to the maps $(2)$ and $(4)$. For $(3)$ we observe that $\colim_{(n,k) \in \nat^2} A_{n,k}\cong  \colim_{n\in \nat} A_{n,n}\cong \colim_{n\in \nat}A_{n}$ in $G\nCalg_{\sepa}$ by a cofinality argument. In order to see that $(5)$ is an equivalence we use the following remark. 
 
  	\begin{rem}\label{rgokpergwefrfw}
  		In the context of a right Bousfield localization as in  \cref{ijewgowegwergerfref}, assume that $(A_{n})_{n\in \nat}$ is a system consisting of $W$-local objects 
  		such that $\colim_{n\in \nat } A_{n}$ exists. Let $ s:(A_n)_{n \in \nat} \to (\colim_{k \in \nat} A_k)_{n \in \nat} $ be the structure map of the colimit. Then the map $R(s):(R(A_n))_{n \in \nat} \to (R(\colim_{k \in \nat} A_k))_{n \in \nat}$ is the structure map for the colimit of $(R(A_n))_{n \in \nat}$. This means that the canonical map
  		$$\colim_{n\in \nat} R(A_{n})\stackrel{\simeq}{\to} R(\colim_{n\in \nat} A_{n})$$
  		exists and is an equivalence. \hB
  	\end{rem}

In order to apply this remark for the equivalence $(5)$
we observe 
that 
 $\colim_{k\in \nat}yL_{h}(A_{n,k})$ is $ W_{R}$-local for every $n$ in $\nat$.
 
 In order to see that  countable filtered colimits in $\AsGc$ commute with finite limits {we employ the following general {fact.} Let $\bD$ be {an $\infty$-}category admitting finite limits and countable filtered colimits. 
	\begin{lem} \label{vnejieogfjeohjfnaff}{The realization functor $ \lvert - \rvert: \Ind^{\aleph_1}(\bD) \to \bD$ commutes with finite limits if and only if  
	  countable filtered colimits commute with finite limits in $\bD$.} \end{lem} \begin{proof} {We first  assume that the realization functor commutes with finite limits.}
  Given a diagram $D: I \times K \to \bD$ with $ I $ countable filtered and $K$ finite, we have to verify that the canonical map 
 	\begin{equation} \label{treijehgirf}
 		\colim_{i \in I} \lim_{k \in K} D(i,k) \to \lim_{k \in K} \colim_{ i \in I} D(i,k) 
 	\end{equation}
 	is an equivalence. Since filtered colimits commute with finite limits in anima, the canonical map
 	\begin{equation}  \label{bhjwaejrofgian}
 		\colim_{i \in I} y(\lim_{k \in K} D(i,k)) \simeq \colim_{i \in I} \lim_{k \in K} y(D(i,k)) \to \lim_{k \in K} \colim_{i \in I} y(D(i,k))
 	\end{equation}
 	in $\Ind^{\aleph_1}(\bD)$ is an equivalence. Applying the left exact functor $\lvert - \rvert: \Ind^{\aleph_1}(\bD) \to \bD$ to the equivalence \eqref{bhjwaejrofgian} we {obtain \eqref{treijehgirf} which is therefore also an equivalence.}
		
	{The converse direction is obvious.} \end{proof}
 
 To verify the condition \Cref{vnejieogfjeohjfnaff} for $\AsGc$ we consider the diagram
 \begin{equation}\label{ertgrogkeroptgetrg}\xymatrix{\Ind^{\aleph_{1}}(\Ind^{\aleph_{1}}(G\nCalg_{\sepa,h}))\ar[rr]^-{\Ind^{\aleph_{1}}(R)}\ar[d]^{|-|'}&&\Ind^{\aleph_{1}}(\AsGc)\ar[d]^{|-|''}\\\Ind^{\aleph_{1}}(G\nCalg_{\sepa,h}) \ar[rr]^-{R}&&\AsGc} \ ,
\end{equation}
which commutes, because $R$ preserves countable filtered colimits. We want to show that $\lvert - \rvert''$ preserves finite limits. Since $\Ind^{\aleph_1}(R) $ is a left exact localization (it is even a right Bousfield localization) and therefore admits an apropriate calculus of fractions, the functor $\lvert - \rvert''$ is left exact if and only if the composite $\lvert - \rvert'' {\circ} \Ind^{\aleph_1}(R) $ is left exact. {By the commutativity of  \eqref{ertgrogkeroptgetrg}  this} is equivalent to $R{\circ}  \lvert - \rvert'$ being left exact. Finally, both $R$ and $ \lvert - \rvert' $ are left exact as right Bousfield localizations. {For}  $\lvert - \rvert'$ this follows from the fact that $\Ind^{\aleph_1}(G\nCalg_{ \sepa,h})$ is compactly assembled by \Cref{neivmsdhjghdikvhjf}. }

\begin{rem}
	We will show in \cref{vsiowerfsdfuiohsdfoijasdf1} below that $\AsGc$ is compactly assembled, {for which the main {ingredient} is the right Bousfiled localization $R$. The above argument can actually be directly adapted to show that in every compactly assembled category $\bD$ admitting finite limits, {countable} filtered colimits commute with finite limits. The crucial point is to replace $R$ in the above argument by the right Bousfield localization $\lvert-\rvert: \Ind^{\aleph_1}(\bD) \to \bD$.}  \hB
\end{rem}

For Assertion \ref{regkopwergerferfrewfw} we must first show for any $\hat A$ in $\Ind^{\aleph_{1}}(G\nCalg_{\sepa,h})$ that $\hat A\otimes -$ sends the morphisms in $W_{R}$ to morphisms which are inverted by $R$. Thus let $(Z_{n})_{n\in \nat}$ be a system in $G\nCalg_{\sepa}$ and $f$ be the canonical map \eqref{wefwefwefqewf}. We must show that  the morphism 
$R(\hat A\otimes f )$ is an equivalence.  
   Since $R$ (by Assertion \ref{regkopwergerferfrewfw2}) and the tensor product in $ \Ind^{\aleph_{1}}(G\nCalg_{\sepa,h})$ preserve countable  filtered colimits  
  and  this category is generated under countable filtered colimits by the image of $y$
  we can assume that $\hat A\simeq y(A)$ for some $A$ in $G\nCalg_{\sepa,h}$. We now argue using the following commutative diagram:
$$ \xymatrix{  y(A)\otimes \colim_{ n \in \nat} y(Z_n)   \ar[r]^{ y(A)\otimes  f } \ar[d]^{\simeq}& \ar[d]^{\simeq}   y(A) \otimes y(\colim_{n \in \nat}  Z_n)  \\ \colim_{ n \in \nat} y(A\otimes Z_n  )  \ar[r] \ar[dr]^{!!} & y (A \otimes \colim_{ n \in \nat} Z_n )\\   &  y (\colim_{ n \in \nat}( A\otimes Z_n )) \ar[u]^{!}_{\simeq}}\ .  $$
The two upper vertical equivalences reflect the fact that $y$ is symmetric monoidal (and also that $\otimes$ preserves filtered colimits for the left one), the morphism $!!$ belongs to $W_{R}$ and is therefore  inverted by $R$, and the equivalence $!$ follows from the fact that  the maximal tensor product in $G\nCalg_{\sepa}$  preserves countable filtered colimits. We can conclude that
$R$ also inverts the upper horizontal {maps}.

The bi-left exactness of the tensor product on 
$ \Ind^{\aleph_{1}}(G\nCalg_{\sepa,h})$ implies the  bi-left exactness of the tensor product on  $\AsGc$. Indeed, using that $R$ is a left-exact Dwyer-Kan localization, for every $A$ in  $\Ind^{\aleph_{1}}(G\nCalg_{\sepa,h})$,  the functor
$R(A)\otimes -: \AsGc\to  \AsGc$ is left exact if and only if its pull-back  
  $$R(A)\otimes R(-)\simeq R(A\otimes -):\Ind^{\aleph_{1}}(G\nCalg_{\sepa,h})\to  \AsGc$$ along $R$  is left-exact. But this is the case since
the tensor product on  $\Ind^{\aleph_{1}}(G\nCalg_{\sepa,h})$ is bi-left exact and $R$ is left-exact.
Also note that it suffices to consider the functors $R(A)\otimes - $ for all $A$ in $\Ind^{\aleph_{1}}(G\nCalg_{\sepa,h})$ since $R$ is essentially surjective.

Finally the tensor product in $ \AsGc$ preserves countable filtered colimits in each variable since it is given by the formula $(A,B)\mapsto R(L(A)\otimes L(B))$ and all functors in this formula (note that $\otimes$ on the right hand side is the tensor product in $\Ind^{\aleph_{1}}(G\nCalg_{\sepa,h})$ {and $L$ is the left-adjoint of $R$ as in \eqref{erfwerfewrferwfw}})
{preserve} countable filtered colimits. 

This finishes the proof of \cref{wreig90wregfwerf}.
	 \end{proof}

We now  show that the functor $\asGc$ from \eqref{543t445ts} satisfies the universal property required in \cref{wrejgioweferfw1}. The superscript $\cfil$ indicates countable filtered colimit preserving functors.  
 \begin{prop}\label{uqwighfuihvsdiohjfhjkbvdxc} The functor
 $\asGc:G\nCalg_{\sepa}\to \AsGc$ is homotopy invariant and preserves countable  filtered colimits.
 Moreover, for any $\infty$-category $\bC$ admitting countable filtered colimits the pull-back along {$\asGc$} induces an equivalence
\begin{equation}\label{fewqfwedewqdewdqd0}(\asGc)^{*}:\Fun^{\cfil}(\AsGc,\bC)\stackrel{\simeq}{\to} \Fun^{h,\cfil} (G\nCalg_{\sepa},\bC)\ . \end{equation}\end{prop}\begin{proof}
The functor {$\asGc$} is homotopy invariant since the first functor $L_{h}$ in  \eqref{543t445ts} is so.
{In order to show that {$\asGc$} preserves countable filtered colimits we}  consider a  system $(A_n)_{n \in \nat}$ in $G\nCalg_{\sepa}$. We  {then observe}   that    the comparison map
	$$\hspace{-0.5cm}{\colim_{n\in \nat} \asGc(A_{n})\stackrel{def}{\simeq}}  \colim_{n \in \nat} Ry(L_{h}(A_n)) {\stackrel{!}{\to}} R(  \colim_{n \in \nat} yL_{h}(A_n)) \stackrel{!!}{\to} Ry(L_{h}(\colim_{n \in \nat} A_n)) { \stackrel{def}{\simeq}  \asGc(\colim_{n\in \nat} A_{n}))}   $$
	is an equivalence. {In fact, the map marked by $!$ is an}  equivalence  since $R$ commutes with countable filtered colimits by \cref{wreig90wregfwerf}.\ref{regkopwergerferfrewfw2}. {Furthermore, the   map  marked by $!!$}  is an equivalence since it is obtained by applying $R$ to a morphism of the form \eqref{wefwefwefqewf} from the set $W_{R}$.
	
	In order to verify the universal properties we analyse the corresponding universal properties of all three functors appearing in \eqref{543t445ts}. Since $R$ is a  {Dwyer}-Kan localization we have an equivalence
	$$R^{*}:\Fun(\AsGc,\bC)\stackrel{\simeq}{\to}\Fun^{W_{R}}(\Ind^{\aleph_{1}}(G\nCalg_{\sepa,h}),\bC)\ .$$ Since $R$ preserves countable filtered colimits this equivalence restricts to an equivalence 
\begin{equation}\label{fewqfwedewqdewdqd}R^{*}:\Fun^{\cfil}({\AsGc},\bC)\stackrel{\simeq}{\to}\Fun^{W_{R},\cfil}(\Ind^{\aleph_{1}}(G\nCalg_{\sepa,h}),\bC)\ .
\end{equation}
Thereby, in order to see essential surjectivity note that if $R^{*}F$ preserves countable filtered  {colimits}, then 
$  L^{*}R^{*}F\simeq F$ also preserves countable  filtered colimits.
	
	It is the  universal property of $y$  that 	
	$$y^{*}:\Fun^{\cfil}(\Ind^{\aleph_{1}}(G\nCalg_{\sepa,h}),\bC)\stackrel{\simeq}{\to} \Fun(G\nCalg_{\sepa,h},\bC)$$ 	
	is an equivalence. It restricts to an equivalence
	\begin{equation}\label{fewqfwedewqdewdqd1}y^{*}:\Fun^{W_{R},\cfil}(\Ind^{\aleph_{1}}(G\nCalg_{\sepa,h}),\bC)\stackrel{\simeq}{\to}  \Fun^{\fin}(G\nCalg_{\sepa,h},\bC)\ ,\end{equation} 	where the superscript {$\fin$} stands for finitary functors{, {by which we} mean functors $F$ such that $FL_h$ preserves countable filtered colimits.}  To see essential surjectivity we  consider a functor $F$  in $\Fun^{\cfil}(\Ind^{\aleph_{1}}(G\nCalg_{\sepa,h}),\bC)$  such that  $y^{*}F$ is finitary. Then $F$ inverts the morphisms \eqref{wefwefwefqewf} in $W_{R}$.
	Indeed, we have the equivalences 
	$$F(\colim_{n\in \nat} yL_{h}(Z_{n}))\simeq \colim_{n\in \nat} F(  yL_{h}(Z_{n}))\simeq
	 F(    yL_{h}(\colim_{n\in \nat}  Z_{n}))$$
	 since $F$ preserves countable filtered colimits and $y^{*}F$ is finitary.

	 Finally, the equivalence \eqref{fqwoeijoqwijdqewdqwed}   restricts to an equivalence 
	 \begin{equation}\label{fewqfwedewqdewdqd2}L_{h}^{*}:\Fun^{\fin}(G\nCalg_{\sepa,h},\bC)\stackrel{\simeq}{\to} \Fun^{h,\cfil}(\Fun(G\nCalg_{\sepa},\bC)\end{equation}
by definition of the property of being finitary. We obtain the equivalence \eqref{fewqfwedewqdewdqd0} by composing the equivalences \eqref{fewqfwedewqdewdqd}, \eqref{fewqfwedewqdewdqd1}, and \eqref{fewqfwedewqdewdqd2}.
	 \end{proof}

In the following   we state two more universal properties of the functor $\asGc$ involving  the symmetric monoidal structures and left-exactness. 
We first note that $\asGc$ is by \eqref{543t445ts} a composition of functors which have symmetric monoidal  refinements. In the following  the subscript $\otimes/\lax
$ refers to symmetric monoidal or lax symmetric monoidal functors.  
\begin{prop}\label{gkowpergreffwer9}
 The functor $\asGc$ has a symmetric monoidal refinement such that  for any symmetric monoidal $\infty$-category $\bC$ with countable filtered colimits we have equivalences
\begin{equation}\label{fewqfwedewqdewdqd0rr}(\asGc)^{*}:\Fun_{\otimes/\lax}^{\cfil}(\AsGc,\bC)\stackrel{\simeq}{\to} \Fun_{\otimes/\lax}^{h,\cfil} (G\nCalg_{\sepa},\bC)\ . \end{equation}\end{prop}\begin{proof}
 The proof  is  completely analogous to the one   of \cref{uqwighfuihvsdiohjfhjkbvdxc}. It uses that $L_{h}$ (see  \eqref{fqwoeijoqwijdqewdqwe332323d}) and $R$ (by \cref{wreig90wregfwerf}.\ref{regkopwergerferfrewfw}) are symmetric monoidal Dwyer-Kan localizations, and the universal property of {$y$}.
\end{proof}

In the following $\lex$ stands for left-exact functors and the superscript $\mathrm{Sch}$ means Schochet exactness as  explained in  \cref{okprthrtertegtrgrg}.
\begin{prop} The functor $\asGc$ is Schochet exact and for any  left-exact $\infty$-category $\bC$ with  countable filtered colimits which commute with finite limits we have  an equivalence 
\begin{equation}\label{fewqfwedewqdferferfwefefewdqd0}(\asGc)^{*}:\Fun^{\cfil,\lex}(\AsGc,\bC)\stackrel{\simeq}{\to} \Fun^{h,\cfil,\mathrm{Sch}} (G\nCalg_{\sepa},\bC)\ . \end{equation} \end{prop}\begin{proof}  The   equivalence \eqref{ghiwoergerfefrefw} 
restricts to an equivalence  \begin{equation}\label{fewqfwedewqdewdqdrfdferwfreffrrr}L_{h}^{*}:\Fun^{\fin,\lex}(G\nCalg_{\sepa,h},\bC)\stackrel{\simeq}{\to} \Fun^{h,\cfil,\mathrm{Sch}}(G\nCalg_{\sepa},\bC)\end{equation}
by the definition of the notion of being  finitary {(see the proof of \cref{uqwighfuihvsdiohjfhjkbvdxc} for this notion)}. 
 We now argue that the equivalence \eqref{fewqfwedewqdewdqd1} restricts to an equivalence
 \begin{equation}\label{frewfrwevfdvsvfv}y^{*}:\Fun^{W_{R},\cfil,\lex}(\Ind^{\aleph_{1}}(G\nCalg_{\sepa,h}),\bC)\to \Fun^{\fin,\lex}(G\nCalg_{\sepa,h},\bC)
\end{equation} 
   provided {that} $\bC$ {is} left exact and has countable filtered colimits which are exact. 
 First of all $y^{*}$ preserves left-exact functors since $y$ preserves finite limits.
 In order to see essential surjectivity assume  that $F$ is in $\Fun^{W_{R},\cfil}(\Ind^{\aleph_{1}}(G\nCalg_{\sepa,h}),\bC)$
  such that $y^{*}F$ is left-exact.  Then $F$ is the left Kan-extension of the left exact functor  $y^{*}F$.
  In view of  $F \simeq  {y_{!}y^{*}F\simeq}  \lvert - \rvert_{\bC} \circ {\Ind^{\aleph_{1}}( y^{*}F)}$ it is left-exact since
  $ \lvert - \rvert_{\bC}$ is left-exact by the assumption on $\bC$ and ${\Ind^{\aleph_{1}}( y^{*}F)}$ is  clearly left-exact. 
     
   Finally, the functor
  $R$, being a right Bousfield localization,  is a left-exact Dwyer-Kan localization so that the equivalence
  \eqref{fewqfwedewqdewdqd} restricts to an equivalence
     \begin{equation}\label{fewqfwedewqdewdqdrrrr}R^{*}:\Fun^{\cfil,\lex}(\AsGc,\bC)\stackrel{\simeq}{\to}\Fun^{W_{R},\cfil,\lex}(\Ind^{\aleph_{1}}(G\nCalg_{\sepa,h}),\bC)
\end{equation}
for every left-exact $\infty$-category $\bC$ with countable filtered colimits.
The equivalence \eqref{fewqfwedewqdferferfwefefewdqd0} is given by the 
 composition of the equivalences
\eqref{fewqfwedewqdewdqdrrrr},\eqref{frewfrwevfdvsvfv}, and \eqref{fewqfwedewqdewdqdrfdferwfreffrrr}.\end{proof}

We now recall the results of  \cite{Dadarlat_1994} and relate them with our theory. This reference deals with the case of the trivial group $G$. But using the equivariant generalization of Blackadar's theory described in \cref{erijgioepgwegerfwef}
and the equivariant theory of asymptotic morphisms developed in  \cite{Guentner_2000} the results of   \cite{Dadarlat_1994} have straightforward equivariant generalizations.
In particular, in order to  apply \cite{Dadarlat_1994} to the equivariant case we
just consider all $C^{*}$-algebras appearing in this paper with their natural $G$-action
and all homomorphisms are equivariant.

{Let} $\AsG$ {be the category} of separable $G$-$C^\ast$-algebras and homotopy classes of asymptotic morphisms  \cite{Guentner_2000}.  It comes with  a canonical functor
$$\asG:G\nCalg_{\sepa}\to \AsG\ .$$
The maximal tensor product on $G\nCalg$ descends to a symmetric monoidal structure on $\AsG$ such that $\asG$ has a canonical symmetric monoidal refinement. 
 The combination of  \cite[Thm 3.5]{Dadarlat_1994} and  \cite[Thm 3.7]{Dadarlat_1994} can be stated as follows:
\begin{theorem}\label{ergkopwregwerffwrf}
There exists a right Bousfield localization $$L_{0}:  \AsG    \leftrightarrows  \ho(\Ind^{\aleph_{1}}(G\nCalg_{\sepa,h}))  :R_{0}$$
such that the essential image of $L_{0}$ consists of objects of the form $\colim_{n\in \nat} yL_{h}(A_{n})$ for shape systems
$(A_{n})_{n\in \nat}$ in $G\nCalg_{\sepa}$.\footnote{The right-adjoint $R_{0}$ is denoted by $L$ in the reference \cite{Dadarlat_1994} since it is a version of a colimit functor.}  
\end{theorem}
 
The following proposition shows that $\AsGc$ provides  an $\infty$-categorical enhancement of $\AsG$.
 \begin{prop} \label{ewrogjpwergwerfwerfwrfw}
 There exists an essentially unique equivalence
 filling the triangle
 \begin{equation}\label{fweqedqwedqededq}  \xymatrix{&   \ho(\Ind^{\aleph_{1}}(G\nCalg_{\sepa,h}))\ar[dl]_{\ho(R)}\ar[dr]^{R_{0}}& \\ \ho(\AsGc)\ar@{-->}[rr]^{\simeq}&&\AsG}\ . \end{equation}  Moreover, $\ho(R)$ and $R_{0}$ are   symmetric monoidal  Bousfield localizations and  the  commutative  triangle  above refines to one   in symmetric monoidal categories.
 \end{prop}
\begin{proof}
Since $R$ is a right Bousfield localization, also $\ho(R)$ is a right Bousfield localization whose left adjoint we will denote by $\ho(L)$. The functor $R_{0}$ is a right Bousfield localization by \cref{ergkopwregwerffwrf}.
 In order to compare these right Bousfield localizations it suffices to compare their classes of local objects, i.e., the essential images of their left-adjoints. Comparing  
 \cref{wreig90wregfwerf}.\ref{gokpwergrwefwrfwref111} with \cref{ergkopwregwerffwrf} we see that
 $$\EssIm(\ho(L))\subseteq \EssIm(L_{0})$$
   since explicit shape systems are  in particular  shape systems. 
 We will show that $$   \EssIm(L_{0})\subseteq \EssIm(\ho(L))\ .$$
To this end let $(A_{n})_{n\in \nat}$ be any shape system for $A$. Then  we have an isomorphism  $\colim_{n\in \nat} yL_{h}(A_{n})\cong L_{0}(A)$ in $ \EssIm(L_{0})$. By \cref{ojkergperwerge} we can   choose an explicit shape system $(A_{n}')_{n\in \nat}$ for $A$.  
    Then we have an equivalence $  \colim_{n\in \nat} yL_{h}(A_{n})\simeq \colim_{n\in \nat} yL_{h}(A'_{n})$ in 
  $\ho(\Ind^{\aleph_{1}}(G\nCalg_{\sepa,h}))$ since both objects represent $L_{0}(A)$.
   We conclude that  the object  $ \colim_{n\in \nat} yL_{h}(A_{n})$ is   in $\EssIm(\ho(L))$.

   {
   Providing a symmetric monoidal refinement of the dashed functor is the same as specifying a symmetric monoidal refinement of $R_0$, because $\ho(R)$ is a symmetric monoidal Bousfield loclization. The tensor product in $ \ho(\Ind^{\aleph_{1}}(G\nCalg_{\sepa,h}))$ is given by 
   $$\colim_{n\in \nat}y(A_{n})\otimes \colim_{m\in \nat}y(B_{m}):= \colim_{n\in \nat} y(A_{n}\otimes B_{n})$$ for any pair of systems $(A_{n})_{n\in \nat}$, $(B_{n})_{n\in \nat}$ in $G\nCalg_{h}$. 
   Using the explicit formula for $R_{0}$ given in \cite[2.2]{Dadarlat_1994}
   one checks that $R_0$ has a canonical symmetric monoidal refinement.}
 \end{proof}

\begin{rem}
We  have a commutative square
\begin{equation}\label{} \xymatrix{G\nCalg_{\sepa}\ar[r]^-{\asGc}\ar[d]^{\asG}&\AsGc\ar[d]^{\ho}\\ \AsG     &\ho(\AsGc )\ar[l]_-{\simeq}}\ .\end{equation}
While the functor $\asGc$ has nice characterizations in terms  universal properties (see \cref{uqwighfuihvsdiohjfhjkbvdxc} and \cref{gkowpergreffwer9}) we can not characterize $\asG$ in a similar manner. The reason is that the universal properties for $\AsGc$ involve countable filtered colimits, but $\ho$ does not preserve such colimits. \hB\end{rem}
 
We now consider the second map in \eqref{wrgwgrefrefrewff} which we decompose as a composition
\begin{equation}\label{fqwfqwedwedqqwededq} \AsGc \xrightarrow{L_{ K_G}} L_{\hat K_G}\AsGc \xrightarrow{R_{K_G}}S_{K_{G}}\AsGc\xrightarrow{R_{S^{2}}} \EsGf\ . 
\end{equation}
Thereby $L_{\hat K_{G}}$ is a  left Bousfield localization and $R_{K_{G}}$ and $R_{S^{2}}$ are right Bousfield localizations.  We call $S_{K_{G}}\AsGc$ the $K_{G}$-stabilization of $\AsGc$. All localizations are in addition symmetric monoidal, and the  right-adjoints preserve countable filtered colimits.  In the following we describe the functors in detail.

The first two localizations are analogous to the ones discussed in  \Cref{tgokptegerferfwefwf}. Recall that we have a map
  $\epsilon:\C\to \hat K_{G}$ which presents $\hat K_{G}$ as an idempotent algebra in $G\nCalg_{\sepa,h}$.
 Since $ R y:G\nCalg_{\sepa,h}\to \AsGc$ is symmetric monoidal we get an idempotent algebra $Ry(\hat K_{G})$ in $\AsGc$. It induces the symmetric monoidal left Bousfield localization
 $$L_{\hat K_{G}}:=Ry(\hat \K_{G})\otimes -:\AsGc\leftrightarrows L_{\hat K_{G}}\AsGc:\incl\ .$$ 
 The following statements are general facts about left Bousfield localizations induced by an idempotent algebra
 in a situation where the tensor product is bi-left exact and preserves countable filtered colimits in each argument.
The right adjoint of the localization $L_{\hat K_{G}}$ preserves countable filtered colimits, and the tensor product in $L_{\hat K_{G}}\AsGc$  preserves countable filtered colimits in each argument.  Furthermore,  the functor $L_{ \hat K_{G}}$ is a left-exact Dwyer-Kan localization, $L_{ \hat K_{G}}\AsGc$ is pointed and left exact,
  the tensor product in $L_{ \hat K_{G}}\AsGc$ is also bi-left exact, and countable  filtered colimits in this category  commute with finite limits.

It follows from the universal property of the left adjoint $L_{\hat K_{G}}$ in \eqref{vwervwevdfvsfvwevfvs} that we have a symmetric monoidal factorization
$$\xymatrix{G\nCalg_{\sepa,h}\ar[rr]^{L_{\hat K_{G}}Ry}\ar[dr]_{L_{\hat K_{G}}}&&L_{\hat K_{G}}\AsGc\\&L_{\hat K_{G}}G\nCalg_{\sepa,h}\ar@{..>}[ur]_{\bar L_{\hat K_{G}}}&}$$
as indicated by the dotted arrow.
 The inclusion $\alpha:K_{G}\to \hat K_{G}$ presents $ L_{\hat K_{G}}(K_{G})$ as an idempotent coalgebra in 
$ L_{\hat K_{G}}G\nCalg_{\sepa,h}$.  Hence $\bar L_{\hat K_{G}}( L_{\hat K_{G}}(K_{G}))$ has the structure of an idempotent coalgebra in $L_{\hat K_{G}}\AsGc$.
We therefore get a symmetric monoidal right  Bousfield localization
$$ \incl: S_{K_{G}}\AsGc \leftrightarrows L_{\hat K_{G}}\AsGc:R_{K_{G}}:= \bar L_{\hat K_{G}}( L_{\hat K_{G}}(K_{G}))    \otimes -\ .$$
The following statements are again general facts about a right Bousfield localization induced by an idempotent coalgebra
in a situation where the tensor product is bi-left exact and preserves countable filtered colimits in each argument.
The functor  $R_{K_{G}}$ preserves countable  filtered colimits.  The tensor product    in $S_{K_{G}}\AsGc$ preserves countable filtered colimits in each argument. 
The localization $R_{ K_{G}}$ is a left-exact Dwyer Kan localization and 
$S_{K_{G}}\AsGc$ is left exact and pointed. In this category
 countable filtered colimits commute with finite limits and  the tensor product   is bi-left exact.
 
 We next state the universal property of the composition
 $$ \asGck:= R_{K_{G}}\circ L_{\hat K_{G}}\circ \asGc: G\nCalg_{\sepa}\to  \AsGc\to  S_{K_{G}}\AsGc\ .$$
  \begin{prop}\label{werkogpergerffwef}  
  The functor $ \asGck$ is homotopy invariant,  Schochet exact, preserves countable filtered colimits, and is $K_{G}$-stable. Moreover, pull-back along this functor induces for any $\infty$-category $\bC$ admitting  countable filtered colimits an equivalence
\begin{equation}\label{friqhiwe9fowedweqdwedq} (\asGck)^{*}:\Fun^{\cfil}( S_{K_{G}}\AsGc,\bC)\stackrel{\simeq}{\to} \Fun^{h,\cfil,K_{G}}(G\nCalg_{\sepa},\bC)\ .
\end{equation} 
 \end{prop}
 \begin{proof}
 It is clear by construction that $\asGck$ has the asserted properties. It remains to verify the equivalence 
 \eqref{friqhiwe9fowedweqdwedq}. A map $f$ in $\AsGc$ will be called $K_{G}$-equivalence if $\asGc(K_{G})\otimes f$ is an equivalence in $\AsGc$. Then 
 by construction we have an equivalence \begin{equation}\label{efwqffeqwfdafdsf}
  L_{\hat K_{G}}^{*}R^{*}_{K_{G}}:\Fun^{\cfil}( S_{K_{G}}\AsGc,\bC)\stackrel{\simeq}{\to} \Fun^{\cfil,K_{G}}(\AsGc,\bC)\ , \end{equation}
 where the superscript $K_{G}$ in the right-hand side   means functors which send $K_{G}$-equivalences in the sense above to equivalences. It remains to show that the equivalence \eqref{fewqfwedewqdewdqd0} restricts to an equivalence
 \begin{equation}\label{fewqfwedewqdewdefeffeqd0}(\asGc)^{*}:\Fun^{\cfil,K_{G}}(\AsGc,\bC)\stackrel{\simeq}{\to} \Fun^{h,\cfil,K_{G}} (G\nCalg_{\sepa},\bC)\ . \end{equation}
 The open part is essential surjectivity. Let $F$  in $\Fun^{\cfil}(\AsGc,\bC)$ be such that
 $(\asGc)^{*}F$ is $K_{G}$-stable   as defined  in \cref{fjioqwefdqwedewdqwed}.\ref{fiqewpoqdqded1}. By definition this means that $(\asGc)^{*}F$
 sends {all $K_{G}$-equivalences  in $G\nCalg_{\sepa}$  to 
  equivalences.} 
But since    the localizations $L_{\hat \K_{G}}$ and $R_{K_{G}}$ are generated by   idempotent algebras or coalgebras
 we must  only show that $F$ inverts the maps  $\hat A\to \asGc(\hat K_{G})\otimes \hat A$ and $\asGc(K_{G})\otimes \hat A\to  \asGc( \hat K_{G})\otimes \hat A$  for all objects $\hat A$ of $\AsGc$.  {Since
$\asGc$ is essentially surjective  we can choose an equivalence $\hat A \simeq \asGc (A)$ for some $A$ in $G\nCalg_{\sepa}$.}  Then these maps are images under $\asGc$ of   {the} $K_{G}$-equivalences
$A\to \hat K_{G}\otimes A$ and $K_{G}\otimes A\to \hat K_{G}\otimes A$
 in $G\nCalg_{\sepa}$ and therefore
inverted by $(\asGc)^{*}F$ by assumption.
  \end{proof}

We now state the universal properties of $\asGck$ involving symmetric monoidal structures and finite limits.
\begin{prop}\label{rijogwergergerwfw} The functor $\asGck$ is symmetric monoidal and for every symmetric monoidal $\infty$-category $\bC$ admitting countable filtered colimits  we have an equivalence
    \begin{equation}\label{friqhiwe9fowedweqdwerr4r4r44dq}(\asGck)^{*} :\Fun_{\otimes/\lax}^{\cfil}( S_{K_{G}}\AsGc,\bC)\stackrel{\simeq}{\to} \Fun_{\otimes/\lax}^{h,\cfil,K_{G}}(G\nCalg_{\sepa},\bC)\ .
\end{equation} 
\end{prop}
\begin{proof}
In order to see this we compose the restriction of  \eqref{fewqfwedewqdewdqd0rr} to $K_{G}$-stable functors with the symmetric monoidal version 
$$ L_{\hat K_{G}}^{*}R^{*}_{K_{G}}:\Fun_{\otimes/\lax}^{\cfil}( S_{K_{G}}\AsGc,\bC)\stackrel{\simeq}{\to}
\Fun_{\otimes/\lax }^{\cfil,K_{G}}(  \AsGc,\bC)$$
of \eqref{efwqffeqwfdafdsf} which ist justified by the fact that $R_{K_{G}}$ and $L_{ \hat K_{G}}$ are symmetric monoidal localizations. 
\end{proof}

\begin{prop} \label{ertwigogijregf8zt28359terwgwreg} The functor $\asGck$ is  Schochet exact and
 for every $\infty$-category $\bC$ with finite limits and exact  countable filtered colimits we have an equivalence \begin{equation}\label{efwqedqwdqwdqwed}(\asGck)^{*}:\Fun^{\cfil,\lex}( S_{K_{G}}\AsGc,\bC)\stackrel{\simeq}{\to}
\Fun^{h,\cfil,K_{G},\mathrm{Sch}}(  G\nCalg_{\sepa},\bC)\ .
\end{equation}\end{prop}\begin{proof} We compose the restriction 
of \eqref{fewqfwedewqdferferfwefefewdqd0} to $K_{G}$-stable functors 
with the functor
$$L_{\hat K_{G}}^{*}R^{*}_{K_{G}}:\Fun^{\cfil,\lex}( S_{K_{G}}\AsGc,\bC)\stackrel{\simeq}{\to}
\Fun^{\cfil,\lex,K_{G} }(  \AsGc,\bC)\ .$$ 
First of all $L_{\hat K_{G}}^{*},R^{*}_{K_{G}}$ preserves left-exact functors since $L_{\hat K_{G}}$ and $R_{K_{G}}$ preserve finite limits.
 For essential surjectivity we argue   as in the justification of the equivalence \eqref{fewqfwedewqdewdefeffeqd0} above. Alternatively one could use  \eqref{fewqfwedewqdewdefeffeqd0} and \eqref{fewqfwedewqdferferfwefefewdqd0}  and   that  $L_{\hat K_{G}}^{*},R^{*}_{K_{G}}$ are left exact localizations.
  \end{proof}

In the last step we construct the localization denoted by $R_{S^{2}}$ in \eqref{fqwfqwedwedqqwededq}.
To this end we import the Bott map  $$ \beta: \asG(S^2(\C) \otimes K_G) \to \asG(K_G)$$ in $\AsG$ from  \cite[Prop 6.16]{Guentner_2000}. Recall the definition of  $\Hom$-sets in the classical equivariant $E$-theory
\begin{equation}\label{gwerguweoifjorefwrfweff}\Hom_{\EsGn}(\esGn(A),\esGn(B)):=\Hom_{\AsG}(\asG(S^{2}(A)\otimes K_{G}),\asG(S^{2}(B)\otimes K_{G}))\ .
\end{equation} 
The statement  \cite[Prop 6.16]{Guentner_2000}  that $\asG(S^{2}(\C))\otimes \beta$ is an equivalence in $\EsGn$ boils down to 
  the statement   that $$\hspace{-0.5cm}\asG(S^{2}(\C)\otimes K_{G})\otimes \asG(S^{2}(\C)\otimes K_{G}) \xrightarrow{\id\otimes \beta} \asG(S^{2}(\C)\otimes K_{G})\otimes \asG(K_{G})\cong \asG(S^{2}(\C)\otimes K_{G})$$ 
is an isomorphism  in $\AsG$. We now consider the symmetric monoidal localization $\ho(S_{K_{G}}\AsGc)\simeq 
R_{0,K_{G}}L_{0,\hat K_{G}}\AsG$ of $\AsG$ at the $K_{G}$-equivalences 
 (the equivalence follows from   \cref{ewrogjpwergwerfwerfwrfw}).  We let $S^{2}$ in $\ho(S_{K_{G}}\AsGc)$ denote the object represented by  $\asG(S^{2}(\C)\otimes K_{G})$ under the equivalence above.
Since $K_{G}$ represents the tensor unit of $\ho(S_{K_{G}}\AsGc)$ we can consider the pair 
$(S^{2},\beta )$  
as an    idempotent coalgebra in $\ho(S_{K_{G}}\AsGc)$. Since the structure of 
an  idempotent coalgebra
can be lifted along the functor $\ho$, we get an idempotent coalgebra $(S^{2},\beta )$ in $S_{K_{G}}\AsGc$.  Then $R_{S^{2}}$ in  \eqref{fqwfqwedwedqqwededq}
is defined as the right Bousfield localization induced by this coalgebra. {In more detail,  $\EsGf$ can be identified with the essential image of tensoring with $S^2$ on $ S_{K_{G}}\AsGc $, and $R_{S^2}$ is the right Bousfield localization 
\[ 
	\EsGf  \rightleftarrows S_{K_{G}}\AsGc: R_{S^2}
\] 
defined by tensoring with $S^2$, and with left adjoint given by the  inclusion. {Here, the counit $\beta $ induces the counit of the adjunction.} }

\begin{prop}\label{gojwprgewrfwferfw} \mbox{}
	\begin{enumerate}
	\item\label{egjkwerrewf} The category $\EsGf$ is pointed, left-exact and admits countable filtered colimits which commute with finite limits.
	\item \label{egjkwerrewf1} $R_{S^{2}}$  is a left-exact Dwyer-Kan localization and preserves countable filtered colimits.
	\item \label{egjkwerrewf2} $R_{S^{2}}$ is a symmetric monoidal Dwyer-Kan localization and the  induced tensor product on $\EsGf$ is bi-left exact and preserves countable filtered colimits in each argument. 
		\item\label{egjkwerrewf3}  The category $\EsGf$ is stable.  	\item \label{egjkwerrewf5} The functor $\esGf:G\nCalg_{\sepa}\to \EsGf$ is homotopy invariant Schochet exact, $K_{G}$-stable and preserves countable filtered colimits.  
		\item \label{egjkwerrewf4} Pull-back along $\esGf$ induces  for every stable countably cocomplete $\infty$-category $\bC$ an equivalence 
		 \begin{equation}\label{rqwfwfw}   (\esGf)^\ast: \Fun^{\cfil, \lex}(\EsGf, \bC) \to \Fun^{h,K_{G},\cfil,\mathrm{Sch}}(G\nCalg_{\sepa, h}, \bC) \ .\end{equation} 
		 \item \label{egjkwerrewf45} Pull-back along the symmetric monoidal refinement of  $\esGf$ induces  for every   countably cocomplete symmetric monoidal stable  $\infty$-category $\bC$ an equivalence 
		 \begin{equation}\label{ewfasdfqweqfg}   (\esGf)^\ast: \Fun_{\otimes/\lax}^{\cfil, \lex}(\EsGf, \bC) \to \Fun_{\otimes/\lax}^{h,K_{G},\cfil,\mathrm{Sch}}(G\nCalg_{\sepa, h}, \bC) \ .\end{equation} 

	\end{enumerate}
\end{prop}
\begin{proof}
The Assertions \ref{egjkwerrewf}, \ref{egjkwerrewf1} and \ref{egjkwerrewf2} follow from the corresponding properties of $S_{K_{G}}\AsGc$ and its symmetric monoidal structure and the fact that $R_{S^{2}}$ is a right Bousfield localization induced by a idempotent coalgebra. The arguments are analogous to the case of $R_{K_{G}}$ above.

In order to see Assertion \ref{egjkwerrewf3} we note that 
 the functor $S^{2}\otimes -$ is equivalent to the loop endofunctor $\Omega^{2}$, see the proof of \cref{jwegoeferfrefwerfst} for a related argument. The map
$\beta:S^{2}\to \beins$ in  $ S_{K_{G}}\AsGc$ is inverted by $R_{S^{2}}$ and induces an equivalence $S^{2}\stackrel{\simeq}{\to} \id$  on $\EsGf$. We thus get an equivalence of functors   $\Omega^{2}\stackrel{\simeq}{\to} \id$ which  implies stability.

In order to show Assertion  \ref{egjkwerrewf5} we use $\esGf \simeq R_{S^{2}}\circ \asGck$, \cref{werkogpergerffwef} stating that $\asGck$ has these properties, and  Assertion \ref{egjkwerrewf1}  ensuring  that the stated  properties are preserved upon composition with $R_{S^{2}}$.

We  now show Assertion \ref{egjkwerrewf4}.
Let $W_{R_{2}} $ be the set of morphisms $\beta \otimes A: S^2 \otimes A \to A$  for all $A$ in $ S_{K_{G}}\AsGc$. Then by definition of $R_{S^{2}}$, for every $\infty$-category $\bC$ we have an equivalence 
$$R_{S^{2}}^{*}:\Fun(\EsGf, \bC)\stackrel{ \simeq}{\to}  \Fun^{W_{S^{2}}}(S_{K_{G}}\AsGc, \bC) \ .$$
We now assume that $\bC$ is stable and countably cocomplete. Then it 
 restricts to the first equivalence in the composition
\begin{align}\label{wefqwdqwedewdq} 
\Fun^{\cfil, \lex}(\EsGf, \bC)\stackrel{R_{S^{2}}^{*},\simeq}{\to} &\: \Fun^{W_{S^{2}},\cfil, \lex }(S_{K_{G}}\AsGc, \bC)\\ \stackrel{\simeq}{\to}&\:\:\Fun^{\cfil, \lex }(S_{K_{G}}\AsGc, \bC)\ .\nonumber
\end{align}
We must show that the second functor is an equivalence, too.
It is here we use that $\bC$ is stable. Let $F$ be in $ \Fun^{\cfil, \lex }(S_{K_{G}}\AsGc, \bC)$. We then must show that
$F$ inverts the morphisms in $W_{S^{2}}$. But this follows from
$$\Omega^2 F(\beta \otimes A) \simeq F(\Omega^2 \beta \otimes A) \simeq F(S^2 \otimes \beta \otimes A) $$
and the fact that $S^{2}\otimes \beta$ is an equivalence. Here we used 
 the  left-exactness of $F$ for the first equivalence and
the fact that $\asGck$ is homotopy invariant and reduced for the second (see the argument for Assertion   \ref{egjkwerrewf3}).
The composition of \eqref{wefqwdqwedewdq} and \eqref{efwqedqwdqwdqwed} yields the desired equivalence.

For Assertion \ref{egjkwerrewf45} we use that $R_{S^{2}}$ is a symmetric monoidal localization which implies the symmetric monoidal analogue 
 \begin{equation}\label{wefqwdqererwedewdq} 
R_{S^{2}}^{*}:\Fun_{\otimes/\lax}^{\cfil, \lex}(\EsGf, \bC)\stackrel{ \simeq}{\to}   \Fun_{\otimes/\lax}^{\cfil, \lex }(S_{K_{G}}\AsGc, \bC)
\end{equation} 
of \eqref{wefqwdqwedewdq}. 
We get \eqref{ewfasdfqweqfg} by composing \eqref{wefqwdqererwedewdq} with 
the restriction of \eqref{friqhiwe9fowedweqdwerr4r4r44dq} to left-exact, respectively Schochet exact functors (combine \cref{rijogwergergerwfw} and \cref{ertwigogijregf8zt28359terwgwreg}).
\end{proof}

 The following proposition shows that the homotopy category of $\EsGf$ is canonically equivalent to the classical equivariant $E$-theory category.
\begin{prop}\label{mnxcviohsdfioasdojkloeqf}
There exists a canonical commutative square
\begin{equation}\label{fwqwfewdewedwdwqdqwddd}\xymatrix{ G\nCalg_{\sepa}\ar[r]^{\esGn}\ar[d]^{\esGf} &\EsGn   \\ \EsGf\ar[r]^{\ho} &\ho(\EsGf)\ar@{-->}[u]_{\simeq}^{\phi} }
\end{equation} 
of symmetric monoidal functors   where $\phi$ is an equivalence
\end{prop}
\begin{proof}We expand the diagram as
 \begin{equation}\label{fqwfknflvdfer}
 \xymatrix{&&G\nCalg_{\sepa}\ar[drr]^{\asG}\ar[dll]_{\asGc} &&\\\AsGc\ar[rr]^-{\ho}\ar[d]^{R_{S^{2}}R_{K_{G}}L_{\hat K_{G}}}&&\ho(\AsGc) \ar[d]^{\ho(R_{S^{2}}R_{K_{G}}L_{\hat K_{G}}  )}\ar[rr]^-{\simeq}_-{\eqref{fweqedqwedqededq}} \ar[d]&&\AsG \ar[d]^{R_{0,S^{2}}R_{0,K_{G}}L_{0,\hat K_{G}} }  \\ \EsGf \ar[rr]^{\ho}&&\ho(\EsGf)\ar@{-->}[rr]^{\phi}_{\simeq}&&\EsGn}\ ,\end{equation}
where the right vertical arrow is the composition of the one-categorical left-and right Bousfield  localizations
$$R_{0,S^{2},K_{G}}:\AsG\xrightarrow{L_{0,\hat K_{G}}} L_{0,\hat K_{G}} \AsG\xrightarrow{R_{0,K_{G}}} S_{0,K_{G}}\AsG\xrightarrow{R_{0,S^{2}}} \EsGn$$  induced by  the idempotent algebra $\asG(\hat K_{G})$ in $\AsG$
and the coalgebras $L_{0,\hat K_{G}}(\asG(K_{G}))$ in $ L_{0,\hat K_{G}} \AsG$ and
$R_{K_{G}}L_{0,\hat K_{G}}(\asG(S^{2}((\C)\otimes K_{G})))$ in $S_{0,K_{G}}\AsG$.
In order to justify that this gives the classical $E$-theory we first observe that for $A,B$ in $G\nCalg_{\sepa}$ the mapping set in the  localization 
are given by
$$\Hom_{\AsG}(\asG(A\otimes S^{2}(\C)\otimes \hat K_{G}\otimes K_{G}),\asG(S\otimes S^{2}(\C)\otimes \hat K_{G}\otimes K_{G}))\ .$$
The map $\epsilon:\C\to \hat K_{G}$ from \eqref{sdfvewrvsfdvsfdvsfv} induces a homotopy equivalence $K_{G}\to \hat K_{G}\otimes K_{G}$.
Using this homotopy equivalence and the homotopy invariance of $\asG$ we can replace the tensor factor $\hat K_{G}\otimes K_{G}$ by $K_{G}$ and get the usual formula \eqref{gwerguweoifjorefwrfweff}. The functor
$\ho(R_{S^{2}} R_{K_{G}} L_{\hat K_{G}})$ simillarly decomposes as a sequence of localizations
$$   \ho(\AsGc)\xrightarrow{\ho(L_{\hat K_{G}})} \ho(L_{\hat K_{G}} \AsGc)\xrightarrow{\ho(R_{K_{G}})} \ho(S_{ K_{G}} \AsGc)\xrightarrow{\ho(R_{S^{2}})} \ho(\EsGf)\ .$$
This implies that the equivalence $\ho(\AsGc)\stackrel{\simeq}{\to} \AsG$ obtained in \cref{ewrogjpwergwerfwerfwrfw} descends to the desired equivalence $\phi$  in  \eqref{fqwfknflvdfer}.
\end{proof}

\subsection{Comparison of the models for $\mathrm{E}$-theory} \label{xvbnfiohregiohsdfpojsdf}

In the preceding three sections we have constructed three  functors
$$ \hspace{-0.5cm}\esGght:G\nCalg_{\sepa}\to  \EsGght\ , \quad  {\esGs}:G\nCalg_{\sepa}\to   {\EsGs}\ , \quad  \esGf:G\nCalg_{\sepa}\to  \EsGf $$ from separable $G$-$C^{*}$-algebras to
   stable $\infty$-categories which are characterized by slightly different 
 universal properties {\eqref{gqjiofwefqwefqwefqew}} for $\esGght$, 
\eqref{rferwiufherufihweriufhwreiufhi} for {$\esGs$}, and \eqref{rqwfwfw} for $\esGf$.
 In order to recall,  $\esGght$ is by {definition} the universal functor from $G\nCalg_{\sepa}$ to an $\infty$-category which inverts the 
classical $E$-theory equivalences. 
The functor {$\esGs$} is  by construction the universal homotopy invariant, $K_{G}$-stable, {exact  and countable sum preserving} functor to a {countably cocomplete} stable $\infty$-category. Finally, $\esGf$ is by construction the universal homotopy invariant, countable filtered colimit preserving, $K_{G}$-stable and Schochet exact functor
 to a countably cocomplete stable $\infty$-category.

By an inspection of the universal properties above we observe that $\esGf$ misses the  exactness property of  the other two functors. 
 {But   \cref{ifqerwfqewdfewdewdqedf} shows} that
$\esGf$ is automatically exact.   {This will be used in \cref{iouheqwrfiotuiojfsjkbna1} in order to verify that $\esGs$  is  equivalent to   $\esGf$. {In  \cref{kopghrthrhertegtrg} we have seen that  $\esGs$
 is also a Dwyer-Kan localization.} We use this information in \cref{iouheqwrfiotuiojfsjkbna2} in order to show that $\esGs$ and $\esGght$ are equivalent as well.}

 {\begin{rem}
 This argumentation is independent of the verification of the universal property of the classical equivariant $E$-theory functor $\esGn$ stated in  \cite{Guentner_2000}. If one is willing to use it one could immediately conclude that
 $\esGs$ and $\esGght$ are equivalent as symmetric monoidal functors in an essentially unique manner by comparing  
  \eqref{fasdfoiuqiofafdfffs1} and \eqref{fasdfoiuqiofafdfffs2}. See also \cref{jiogpgergrwgw9}. \hB
 \end{rem}}

 \begin{prop} \label{ifqerwfqewdfewdewdqedf}
	The functor $\esGf: G\nCalg_{\sepa} \to \EsGf$ is exact.
\end{prop}
\begin{proof} In view of the universal property of $L_{h}$ from \eqref{cdcadsnckladscasdcadc13} and the homotopy invariance of $\esGf$ and $\esGn$
the commutative square \eqref{fwqwfewdewedwdwqdqwddd} induces a commutative square
\begin{equation}\label{fwqwfewdewedwderererwqdqwddd}\xymatrix{ G\nCalg_{\sepa,h}\ar[r]^{\esGnh}\ar[d]^{\esGfh} &\EsGn   \\ \EsGf\ar[r]^{\ho} &\ho(\EsGf)\ar@{-->}[u]_{\simeq}^{\phi} } \ .
\end{equation} 
	We consider a short exact sequence $0 \to  A \to B \to C \to 0$ in $\nCalg_{\sepa}$ and let  $c: L_{h}(A) \to \Fib(L_{h}(B) \to L_{h}(C))$ be the induced comparison map in $G\nCalg_{\sepa,h}$.
	We must show that this map gets inverted by  the canonical functor $ \esGfh$.
The functor $ \esGnh$	inverts $c$ by \cite[Prop. 5.14]{Guentner_2000}. Since $\ho$ is conservative 
  this implies that also $\esGfh$ inverts $c$. 
\end{proof}

 \begin{prop}\label{iouheqwrfiotuiojfsjkbna1} There exists an essentially unique commutative triangle
$$ \xymatrix{& G\nCalg_{\sepa} \ar[ld]_{\esGs} \ar[rd]^{\esGf} \\ \EsGs \ar@{-->}[rr]^{\simeq} && \EsGf} $$
  of symmetric monoidal $\infty$-categories.\end{prop}\begin{proof}
  We use that a left exact and countable filtered colimit preserving functor between countably cocomplete stable $\infty$-categories is the same as a countable colimit preserving functor.
  Since  $\esGf$ is exact by \cref{ifqerwfqewdfewdewdqedf} we obtain the symmetric monoidal and countable colimit preserving dashed arrow from the universal property \eqref{fasdfoiuqiofafdfffs1}. Since an exact functor is Schochet exact,
  {$\EsGs$ admits countable colimits and}
   $\esGs$ preserves filtered colimits by {\cref{gjerogrgesrgseg}.\ref{kophertherth9}} we get a potential inverse symmetric monoidal and countable colimit preserving arrow from \eqref{ewfasdfqweqfg}. Using the universal properties of $\esGs$ and $\esGf$ again we see that  these arrows are indeed mutually inverse to each other. 
  \end{proof}

 {\begin{prop} \label{iouheqwrfiotuiojfsjkbna2}
	There exists an essentially unique commutative triangle 	$$ \xymatrix{& G\nCalg_{\sepa} \ar[ld]_{\esGght} \ar[rd]^{\esGs} \\ \EsGght \ar@{-->}[rr]^{\simeq} && \EsGs} $$
	 of symmetric monoidal $\infty$-categories.\end{prop}
\begin{proof} {By \cref{kopghrthrhertegtrg} we know that
 $\esGs$ is a  symmetric monoidal} Dwyer-Kan localization.  By  \cref{iouheqwrfiotuiojfsjkbna1} and  \cref{mnxcviohsdfioasdojkloeqf}     this functor inverts  precisely the classical $E$-theory equivalences. Since this  is the defining property of {{the symmetric monoidal} Dwyer-Kan localization} $\esGght$ 
  we get the assertion.
	 \end{proof}}

 {From now on we can safely omit the subscripts indicating the particular constructions of the equivariant $E$-theory functors and just use the notation
 $$\esG:G\nCalg_{\sepa}\to \EsG$$
 for any of them.}

  {Recall from \cite[Thm. 1.5]{KKG} that $$\kkGs:G\nCalg_{\sepa}\to \KKGs$$ is the universal homotopy invariant, $K_{G}$-stable and semi-exact functor to a stable $\infty$-category, where semi-exactness means that it sends exact sequences in $G\nCalg_{\sepa}$ with an equivariant  completely positive contractive split to a fibre sequence. 
 In addition we know that $\KKGs$ admits countable colimits, and that the functor
 $\kkGs$ preserves countable sums. Pull-back along $\kkGs$ therefore induces for every countably cocomplete stable $\infty$-category $\bC$ an equivalence
 $$(\kkGs)^{*}:\Fun^{\colim_{\omega}}(\KKGs,\bC)\stackrel{\simeq}{\to} \Fun^{h,K_{G},\mathrm{se},\oplus}(\nCalg_{{\sepa}},\bC)\ ,$$ where the superscript $\mathrm{se}$ indicates semi-exact functors.
  The only difference to $\esG$ is that the latter sends all exact sequences to fibre sequences. We therefore have a natural  commutative diagram: 
}

\begin{equation}\label{herjthiertogjertioger}\xymatrix{&G\nCalg_{\sepa}\ar[dl]_{\kkGs}\ar[dr]^{\esG}&\\\KKGs\ar@{..>}[rr]^{c}&&\EsG}
\end{equation}

\begin{prop} \label{gregijweoirjferffwrf}
 {The functor {$c$ presents} its targets as a Verdier quotient.}
\end{prop}
\begin{proof} 
	  We let $N$ {denote} the kernel of $c$. {Then the functor}  
	\begin{equation}\label{vklnmweaiofndfaodfj}
		G\nCalg_{{\sepa}} \xrightarrow{\kkGs} \KKGs \to \KKGs/N
	\end{equation}
	is exact   {since by the exactness of $\esG$} for every exact sequence $0 \to A \to B \to C \to 0$ in $G\nCalg_{{\sepa}}$ the object
	$$ \Fib(\kkGs(A) \to \Fib(\kkGs(B) \to \kkGs(C)))$$
	{belongs to}  $N$. {In addition, the functor in \eqref{vklnmweaiofndfaodfj}   is homotopy invariant, $K_{G}$-stable and countable sum-preserving since $\kkGs$ has these properties.} 
		 {By the universal property of $\esG$ and the properties of}  \eqref{vklnmweaiofndfaodfj} {discussed above} we obtain an inverse $\EsG \to \KKGs/N$ of the canonical {functor} $ \KKGs/N \to \EsG $ induced by $c$.  
   \end{proof}

We finally use \cref{gregijweoirjferffwrf} in order to descend some group-change adjunctions from equivariant  $KK$-theory to $E$-theory.
Let $H$ be a subgroup of $G$ and let $\Ind^{G}_{H}$ and $\Res^{G}_{H}$ {be} the functors induced by the induction and restriction functors on the level of separable $C^{*}$-algebras.
\begin{prop}\label{kogregwergwerg}
 We have an adjunction
\begin{equation}\label{qwefqwdqqeqed1123} \Ind_{H}^{G}:\EsH {\rightleftarrows} \EsG:\Res^{G}_{H}\end{equation} 
where $\Res^{G}_{H}$ preserves countable colimits.
\end{prop}
 \begin{proof}
 We start from the adjunction
\begin{equation}\label{qwefqwdqqeqed11231} \Ind_{H}^{G}:\KK_{\sepa}^{H}{\rightleftarrows} \KK^{G}_{\sepa}:\Res^{G}_{H} \end{equation}  obtained by restricting the adjunction in \cite[Thm. 1.23.1]{KKG} to the subcategories of compact objects.
 Since the induction and restriction functors on the level of separable $C^{*}$-algebras preserve exact sequences and countable sums 
 they descend through the Verdier quotient $c$ in \cref{gregijweoirjferffwrf}. Hence the adjunction \eqref{qwefqwdqqeqed11231} induces the adjunction \eqref{qwefqwdqqeqed1123} whose right-adjoint also preserves countable colimits.
 \end{proof}
 
 In the following $-\rtimes_{\max}G$ and $\Res_{G}$ denote the functors induced by the functors on separable $C^{*}$-algebras given by  the maximal crossed product and the functor equipping a $C^{*}$-algebra with the trivial $G$-action.
 \begin{prop}\label{kogregwergwerg1}\mbox{}
\begin{enumerate}\item  If $G$ is finite, then we have  an adjunction
\begin{equation}\label{qwefqwdqqeqrwerwerwred1123dddd}   \Res_{G} :\Es\leftrightarrows \EsG: -\rtimes_{\max}G\end{equation} whose right-adjoint preserves countable colimits.
\item We have an adjunction 
\begin{equation}\label{qwefqwdqqeqrwerwerwred1123dddd1}   - \rtimes_{\max}G:\EsG\leftrightarrows \Es:  \Res_{G}\end{equation} 
whose right-adjoint preserves countable colimits.
\end{enumerate}
\end{prop}
\begin{proof}
{We} argue similarly as in the proof of \cref{kogregwergwerg} in order to descend the corresponding adjunctions in $KK$-theory  \cite[Thm.1.23.3 \& 4]{KKG}  to $E$-theory. Here we use that $\Res_{G}$ and $ - \rtimes_{\max}G$  on the level of separable $C^{*}$-algebras preserve exact sequences and countable sums.
\end{proof}

\subsection{$\mathrm{E}$-theory is compactly assembled}\label{kogpwergefrefwrfwrf}

{Recall \cref{ijtowggwerwgerg9}  characterising  an   $\infty$-category $\bC$ admitting countable filtered colimits as  \pcas{}  
if its}  realization functor is a right-adjoint of an adjunction $$S:\bC\leftrightarrows \Ind^{\aleph_{1}}(\bC):|-|\ .$$
 {The left-adjoint $S$ is called the shape functor.  It is automatically fully faithful by \cref{jwoirthopgfrfwrgre}.}

The main result of this section is \cref{vsiowerfsdfuiohsdfoijasdf} stating that the countably cocomplete stable $\infty$-category $\EsG$ representing equivariant $E$-theory  introduced in \cref{wrejgioweferfw} is compactly assembled.
The idea is to use the construction of the version $\EsGf$ of this  category by a series of 
left- and right Bousfield localizations
\begin{equation}\label{fwefqwedqwedqwdqdqwd}\hspace{-0.5cm}\Ind^{\aleph_{1}}(G\nCalg_{\sepa,h})\xrightarrow{R}  \AsGc
 \xrightarrow{L_{\hat K_{G}}} L_{\hat K_{G}}\AsGc \xrightarrow{R_{K_{G}}}S_{K_{G}} \AsGc \xrightarrow{R_{S^{2}}} \EsGf
\end{equation} 
 which are described in detail in \cref{irthjzgjdiogjhjdicnmjksdisjf}. We start with the observation that an  {$\Ind^{\aleph_{1}}$}-completion is always \pcas{} and then  use  that the property of being  \pcas{}
  is inherited by  {appropriate Bousfield} localizations.

\begin{prop}\label{neivmsdhjghdikvhjf}
	For any $\infty$-category $\bC$ the category $\Ind^{\aleph_{1}}(\bC)$ is  \pcas{} and its shape functor is given by $\hat y:=\Ind^{\aleph_{1}}(y): \Ind^{\aleph_{1}}(\bC) \to \Ind^{\aleph_{1}}(\Ind^{\aleph_{1}}(\bC))$. 
\end{prop}
\begin{proof}(compare with \cite[Ex. 21.1.2.13]{sag})
We consider the functors $y:\bC\to \Ind^{\aleph_{1}}(\bC)$ and $y_{ \Ind^{\aleph_{1}}(\bC)}:\Ind^{\aleph_{1}}(\bC)\to \Ind^{\aleph_{1}}(\Ind^{\aleph_{1}}(\bC))$. We let $ |-|_{ \Ind^{\aleph_{1}}(\bC)}$ denote the left-adjoint of $y_{ \Ind^{\aleph_{1}}(\bC)}$. In order to show that $\hat y$  is a  further left adjoint of $ |-|_{ \Ind^{\aleph_{1}}(\bC)}$   we start with the bi-natural equivalence
	\begin{align*} \Map_{\Ind^{\aleph_{1}}(\Ind^{\aleph_{1}}(\bC))}(y_{ \Ind^{\aleph_{1}}(\bC)}\circ y(-),  y_{ \Ind^{\aleph_{1}}(\bC)}  (-))  &\simeq \Map_{\Ind^{\aleph_{1}}(\bC)}(y(-), -) \\\simeq &\:\: \Map_{\Ind^{\aleph_{1}}(\bC)}(y(-), |  y_{\Ind^{\aleph_{1}}}(-) |_{\Ind^{\aleph_{1}}(\bC)}) \end{align*}
	of functors $\bC^{\op} \times  {\Ind^{\aleph_1}}(\bC) \to \Spc$. Left Kan-extension along $\bC^{\op} \times  {\Ind^{\aleph_1}}(\bC) \to \bC^{\op} \times  {\Ind^{\aleph_1}}(\Ind^{\aleph_{1}}(\bC))$ induces the bi-natural equivalence
	$$ \Map_{\Ind^{\aleph_{1}}(\Ind^{\aleph_{1}}(\bC))}(y_{ \Ind^{\aleph_{1}}(\bC)}\circ y(-), -) \simeq \Map_{\Ind^{\aleph_{1}}(\bC)}(y(-), |-|_{ \Ind^{\aleph_{1}}(\bC)})\ . $$
	 Finally right Kan extension along $\bC^{\op} \times  {\Ind^{\aleph_1}}(\Ind^{\aleph_1} (\bC)) \to \Ind^{{\aleph_{1}}}(\bC)^{\op} \times  {\Ind^{\aleph_1}}(\Ind^{\aleph_1}(\bC)) $ induces the required bi-natural equivalence
	$$ \Map_{\Ind^{\aleph_{1}}(\Ind^{\aleph_{1}}(\bC))}(\hat y(-), -) \simeq \Map_{\Ind^{\aleph_{1}}(\bC)}(-,  |-|_{ \Ind^{\aleph_{1}}(\bC)})\ . $$
\end{proof}

Let $$L:\bD\leftrightarrows \bC:R$$  be a right Bousfield localization  {and recall that notation $\hat R$ from \eqref{wrefrefreferwffwrfw}.}
\begin{prop} \label{iopjsdfvniohwefriohjqw}
	If $R$   preserves countable filtered colimits and  $\bC$ is \pcas, then $\bD$ is \pcas{}  
	and its shape functor is given by $\hat R\circ S_{\bC}\circ L$.
	\end{prop}
\begin{proof}
The  equivalence   \eqref{fwewerfreferfwfref2}  can be rewritten as a  commutative square
$$ \xymatrix{\Ind^{\aleph_{1}}(\bC) \ar[rr]^{R \circ | - |_{\bC}} \ar[d]^{ \hat R} & &  \bD \ar@{=}[d]^{\id_{\bD}} \\ \Ind^{\aleph_{1}}(\bD) \ar[rr]^{|-|_{\bD}}  & & \bD  }  \ .$$
We furthermore have a square 
$$\xymatrix{\Ind^{\aleph_{1}}(\bC) \ar@{<-}[rr]^{S_{\bC} \circ L} \ar[d]^{\hat R} & &  \bD \ar@{=}[d]^{\id_{\bD}} \\ \Ind^{\aleph_{1}}(\bD) \ar@{<-}[rr]^-{ \hat R\circ S_{\bC} \circ L}  & & \bD }\ . $$
The vertical maps  in both squares are Dwyer-Kan localizations, and the two upper horizontal arrows are adjoint to each other. This implies by \cite[Prop. 7.1.14]{Cisinski:2017} that the lower to morphisms are also adjoint to each other.  So $\bD$ is \pcas{} and the shape functor of $\bD$  is given by the composition
$\hat R\circ S_{\bC}\circ L$. 
\end{proof}
Let
$$ L: \bC \rightleftarrows \bD: R $$
be a left Bousfield localization {and let $\hat L$ be as in \eqref{wrefrefreferwffwrfw}.}  
\begin{prop}\label{mxcvbnuiohwefioja}
	If $R$   preserves countable filtered colimits and  $\bC$ is \pcas, then $\bD$ is \pcas{}  
	and its shape functor is given by $\hat L\circ S_{\bC}\circ R$. \end{prop}
\begin{proof}
By \cref{kjogperwgefweff} the functor $L$ preserves shapes. 
Since $L$ is essentially surjective it is then  clear from \cref{tgijoworegvferwfrewfw} that $\bD$ is \pcas.
Furthermore, by \cref{kjogperwgefweff} we have $\hat L\circ S_{\bC}\simeq S_{\bD}\circ L$.
We precompose with $R$ and use $L\circ R\simeq \id_{\bD}$ on the right-hand side in order to conclude the formula for the shape functor of $\bD$.
\end{proof}
 
\begin{theorem}\label{vsiowerfsdfuiohsdfoijasdf}
	The category $\EsG$ is \pcas.
\end{theorem}
\begin{proof}
According to \cref{xvbnfiohregiohsdfpojsdf} we can use the construction of $\EsG$ as 
$\EsGf$. The latter is presented by   the sequence  \eqref{fwefqwedqwedqwdqdqwd} of   left- and right Bousfield localizations of $\Ind^{\aleph_1}(G\nCalg_{\sepa, h} ) $. This initial category is \pcas{} by \cref{neivmsdhjghdikvhjf}. We then use \cref{mxcvbnuiohwefioja} and \cref{iopjsdfvniohwefriohjqw} in order to see that being \pcas{} is inherited by the localizations. 
 To this end we use that the right-adjoints of all these Bousfield localizations preserve countable filtered colimits
as seen in \cref{irthjzgjdiogjhjdicnmjksdisjf}. 
\end{proof}

 The proof of \cref{vsiowerfsdfuiohsdfoijasdf} shows that also the intermediate categories in the sequence \eqref{fwefqwedqwedqwdqdqwd}  are compactly assembled. For later use we record:
\begin{kor}\label{vsiowerfsdfuiohsdfoijasdf1}
The $\infty$-category $\AsGc$ is \pcas.
\end{kor}

We say that $A$ in $G\nCalg_{\sepa}$ is twice suspended and $K_{G}$-stable if it isomorphic to $A\otimes S^{2}(\C)\otimes K_{G}$ for some $A'$ in $G\nCalg_{\sepa}$.
We let $S_{\AsGc}$ and $S_{\EsG}$ denote the shape functors for $\AsGc$ and $\EsG$.
{Recall from \cref{rwogpwegrewfrefrfwf} the notion of a shape system in $G\nCalg_{\sepa}$.}

\begin{prop}\label{jigowggergwe9}Assume that  $(A_{n})_{n\in \nat}$ is a shape system for $A$.
\begin{enumerate}
\item \label{wekrogpwergewrfrfw} We have an equivalence $S_{\AsGc}(A)\simeq \colim_{n\in \nat} y_{\AsGc}(\asGc(A_{n}))$.
\item \label{wekrogpwergewrfrfw1} If $A$ is twice syspended and $K_{G}$-stable, then $S_{\EsG}(A)\simeq \colim_{n\in \nat} y_{\EsG}(\esG(A_{n}))$.
\end{enumerate}
 \end{prop}
\begin{proof}
By \cref{iopjsdfvniohwefriohjqw} we have an equivalence 
$$S_{\AsGc}(A)\simeq  \hat R(S_{\Ind^{\aleph_{1}}(G\nCalg_{\sepa,h})}(L(A))) \ .$$
From the proof of \cref{wreig90wregfwerf} we know that  
$\colim_{n\in \nat} yL_{h}(A_{n})$ is a $W_{R}$-local replacement of $yL_{h}(A)$.
Therefore $L(A)\simeq \colim_{n\in \nat} yL_{h}(A_{n})$.  By \cref{neivmsdhjghdikvhjf}
we get
$$S_{\Ind^{\aleph_{1}}(G\nCalg_{\sepa,h})}(L(A))\simeq \colim_{n\in \nat} y_{\Ind^{\aleph_{1}}(G\nCalg_{\sepa,h})}(yL_{h}(A_{n}))\ .$$
We now apply $\hat R$ and use $y_{\AsGc} R\simeq \hat Ry_{\Ind^{\aleph_{1}}(G\nCalg_{\sepa,h})}$ in order to get 
$$S_{\AsGc}(A)\simeq  \colim_{n\in \nat}  y_{\AsGc}(R yL_{h}(A_{n}))$$
which in view of \eqref{543t445ts} is  the desired equivalence in  Assertion \ref{wekrogpwergewrfrfw}.

 In order to show Assertion \ref{wekrogpwergewrfrfw1}  we note that 
	  the adjoint of the composition of Bousfield localizations $R_{S^{2}} R_{K_{G}} L_{\hat K_{G}}:
	  \AsGc\to \EsG$ from \eqref{fqwfqwedwedqqwededq}  provides    a fully faithful inclusion $ I:\EsG \to \AsGc $ whose image consists of 
	  the objects of the form $\asGc(A)$ for twice suspended and $K_{G}$-stable algebras. Using
	   the explicit formulas for the shape functors in \Cref{iopjsdfvniohwefriohjqw} and \Cref{mxcvbnuiohwefioja} 
	   we conclude that
	 $$S_{\EsG}( \esG(A))\simeq   \hat R_{S^{2}}\hat R_{K_{G}} \hat L_{\hat K_{G}} S_{\AsGc}(I(A))\ .$$
The formula for $S_{\AsGc}$ from  Assertion \ref{wekrogpwergewrfrfw}     implies the desired formula for $S_{\EsG}$.
\end{proof}

The following proposition provides many examples of compact morphisms in $\EsG$.
\begin{prop}\label{gjoiowepgrfwefrefw}
 \mbox{}
Let $A\to B$ be a semi-projective morphism in $G\nCalg_{\sepa}$.
\begin{enumerate}
\item \label{wtggwregreewfre}
 The induced morphism   $\asGc(f): \asGc(A)\to \asGc(B)$ is strongly  compact in  $\AsGc$.   \item\label{wtggwregreewfre1} If $B$ is twice suspended and $K_{G}$-stable, then $\esG(f): \esG(A)\to \esG(B)$   is strongly compact in  $\EsG$.  \end{enumerate}
 \end{prop}
\begin{proof}
We choose a  shape system $(B_{n})_{n\in \nat} $ for $B $.
By semi-projectivity of $f$ get a factorization
$f:A\stackrel{f'}{\to} B_{n}\to B $ of $f$. This gives a map
$$g:y_{\AsGc}(\asGc(A))\stackrel{y_{\AsGc}(\asGc(f'))}{\to} y_{\AsGc}(\asGc(B_{n}))\to \colim_{n\in \nat} y_{\AsGc}(\asGc(B_{n}))\simeq S_{\AsGc}(B )$$ 
where we use \cref{jigowggergwe9}.\ref{wekrogpwergewrfrfw} for the last equivalence.
The morphism $\asGc(f)$  is equivalent to the composition   $$\asGc(A)\simeq |y_{\AsGc}(\asGc(A))|\stackrel{|g|}{\to}|S_{\AsGc}(\asGc(B))|\simeq \asGc(B )\ .$$
We now use {that} 
$|g|$ is strongly compact by \cref{wegjiopewgfsg}. This finishes the proof of Assertion \ref{wtggwregreewfre}.

The proof of Assertion \ref{wtggwregreewfre1} is completely analoguous using \cref{jigowggergwe9}.\ref{wekrogpwergewrfrfw1}.
\end{proof}

We finally  explain what we know   about examples of compact objects in $\AsGc$ and  $\EsG$.
  If $A$ is a semi-projective algebra in $G\nCalg_{\sepa}$ (i.e., $\id_{A}$ is semi-projective),
then $yL_{h}(A)$ is $W_{R}$-local and   compact in $\Ind^{\aleph_{1}}(G\nCalg_{h,\sepa})$.  It therefore remains compact in the localization $\AsGc$. We conclude:
\begin{kor}
For any semi-projective algebra $A$ in $G\nCalg_{\sepa}$ the object $\asGc(A)$ in $\AsGc$ is compact.
\end{kor}

For trivial $G$ known examples are $\C$, the one-fold suspension $S(\C)$ of $\C$, the Cuntz algebra $\cO_{\infty}$, the Toeplitz algebra, see  \cite{zbMATH03996430}. Furthermore, the class of semi-projective algebras is closed under forming finite free products.

The left Bousfield localization $L_{\hat K_{G}}:\AsGc\to L_{\hat K_{G}}\AsGc$ preserves compactness since its right-adjoint preserves
countable filtered colimits. In general, the right Bousfield localization  $R_{K_{G}}:
 L_{\hat K_{G}}\AsGc\to  S_{K_{G}} \AsGc $ in the composition \eqref{fwefqwedqwedqwdqdqwd} may destroy  compactness.  But if $G$ is finite, then  this localization is an equivalence and therefore preserves compactness.
 For the further discussion let us therefore assume that $G$ is finite. 
 If the resulting object $\asGck(A)$ in $S_{K_{G}}\asGc$ is local for the localization $R_{S^{2}}$, it remains compact after applying
$R_{S^{2}}$.

 {By the equivariant version of \cite[Thm. 4.3]{DaLo}  group objects in $ S_{K_{G}} \AsGc $ are local for $R_{S^{2}}$.} 
 Typical examples of group objects are objects in the image of the loop functor. As looping in $ S_{K_{G}}\AsGc  $ corresponds to suspending in $G\nCalg_{\sepa}$ we see that for finite groups $G$
a semi-projective algebra of the form $A\cong S(\C)\otimes A'$ (we say once suspended) induces a compact object $\esG(A) $ in $\EsG$. 
\begin{kor} \label{vhbnweouifhdvsiohadwfdja} If $G$ is finite and 
$A$ in $G\nCalg_{\sepa}$ is once suspended and semi-projective, then $\esG(A)$ is compact.
\end{kor}
The only example of a once suspended and semi-projective $G$-$C^{*}$-algebra we know is $S(\C)$.
\begin{kor}\label{gjiseogpgregesg}
If $G$ is finite, then $\esG(\C)$ is a compact object of $\EsG$.
\end{kor}
\begin{proof}
We have $\esG(\C)\simeq \Sigma \esG(S(\C))$ and $\Sigma$ preserves compact objects.
\end{proof}

 \subsection{Further properties of $\Es$ and $\KKs$}\label{gugeiwprgejrgpowjeipof}
 In this section we address some further structural questions about the categories $\KKs$ and $\Es$.
In particular we show that $\Es$ {is} strictly larger than  the bootstrap class (otherwise $\Es$ would be compactly assembled for a trivial reason), and that $\KKs$ is not rigid.   {In addition we} state interesting questions to which we {currently} do not have answers. {At the end of this section we use homological functors in order to provide  criteria for a map between separable $C^{*}$-algebras to induce an equivalence or a strong phantom {map} in $E$-theory.}

We start with some  preparations from $\infty$-category theory.
{We consider a} symmetric monoidal countably cocomplete stable {$\infty$-}category $ \bD $ with compact {tensor} unit {$\unit$ (see  \cref{ojkbpertherthregertg}) and with the property that}   the tensor product preserves countable colimits in each variable. {The following is an interpretation of the notion of rigid categories in the countably cocomplete setting, for rigid cateogries in the presentable setting see \cite{Efimov_2025}, \cite{Ramzi_2024_Rigid}}. 
\begin{ddd} \hfil
	\begin{enumerate}
		\item A map $f: D \to E$ in $\bD$ is trace-class if there exists an object $D^\prime$ in $\bD$, a map $ g: \unit \to D^\prime \otimes E $ and a map $h: D \otimes D^\prime \to \unit$ such that $f$ is equivalent to the composite
		\[ 
			D \xrightarrow{\id_D \otimes g} D \otimes D^\prime \otimes E \xrightarrow{h \otimes \id_E} E\ .
		\]
		\item {A trace{-}class exhaustion of $D$ in $\bD$ is a system 
		 $(D_{n})_{n\in \nat}$ with $D \simeq \colim_{n \in \nat} D_n$
		 and such that the structure maps $D_{n}\to D_{n+1}$ are trace-class for all $n$ in $\nat$.}
		 \item {$D$ in $\bD$ is called trace-class exhaustible if it  admits a trace-class exhaustion.}
		 
		\item The category $\bD$ is rigid if every object admits a trace-class exhaustion.
	\end{enumerate}
\end{ddd}

{
	\begin{rem}
		A category $\bD$ is rigid in the {sense above}   if and only if $\Ind_{\aleph_1}(\bD)$ is rigid as a presentable category (see \cite[Cor. 4.55]{Ramzi_2024_Rigid}). \hB
	\end{rem}
}

\begin{ex} \label{cijodjvnfue}
	{If} $R$ a commutative ring spectrum, {then} the category $\Mod(R)^{\aleph_1}$ is rigid. {{This} follows from the fact that} $\Mod(R)^{\aleph_1}$ is generated under countable filtered colimits by dualizable objects{, and since} every map between dualizable objects is trace-{class}.\hB
\end{ex}

{We recall the standing assumptions on $\bD $ from the beginning of \cref{gugeiwprgejrgpowjeipof}.}
\begin{prop} \label{knldsfe0wadsf} \hfil
	\begin{enumerate}
		\item \label{kophertgerge}Every trace-class map in $\bD$ is weakly compact.
		\item\label{hrtopgkpergortkgpeorggertg} Every rigid category is compactly assembled.
	\end{enumerate}
\end{prop}
\begin{proof}
	{For Assertion \ref{kophertgerge} we verify the condition described in \cref{jowtogpwgerferfwefewfref} {using the same notation}. To this end we consider a countable system $(Z_{i})_{i\in I}$   with $Z\simeq \colim_{i\in I} Z_{i}$ and let $k:E\to Z$ be some morphism.
	Using the compactness of $\unit$ and the fact that the functor  $D'\otimes -$ preserves filtered colimits we can  find a lift indicated by the dotted arrow in
	$$\xymatrix{\unit\ar@{..>}[r]\ar[d]^{g}&D'\otimes Z_{i}\ar[d]^{\id_{D'}\otimes \iota}\\ D'\otimes E\ar[r]^{\id_{D'}\otimes k}&D'\otimes Z}$$ for some $i$ in $I$ where $\iota:Z_{i}\to Z$ is the structure map.
	We then apply $D\otimes-$ and extend the diagram as follows:
	$$\xymatrix{ D\ar@{..>}[r]\ar[d]^{\id_{D}\otimes g}&D\otimes D'\otimes Z_{i}\ar[d]^{\id_{D\otimes D'}\otimes \iota}\ar[r]^-{h\otimes \id_{Z_{i}}}&Z_{i}\ar@/^0.5cm/[ddl]^{\iota}\\ D\otimes D'\otimes E\ar[r]^{\id_{D\otimes D'}\otimes k}\ar[d]^{h\otimes \id_{E}}&D\otimes D'\otimes Z\ar[d]^{h\otimes \id_{Z}}&\\ E\ar[r]^{k}&Z&}\ .$$}
	
	The second claim follows from the first and \cref{vnaefnauen}.
\end{proof}

We {consider the commutative} ring spectrum ${S} \coloneqq \map_{\bD}(\unit, \unit)$. The symmetric monoidal structure on $\bD$ induces a lax symmetric monoidal functor
 \begin{equation}\label{opkweorgkwperfo}R: \bD \to \Mod({S})\ ,
\end{equation}
 such that there exists a commutative triangle
\[ 
	\xymatrix{& & \Mod({S}) \ar[d] \\ \bD \ar[urr]^{R}  \ar[rr]_{\map_{\bD}(\unit, -)} & & \Sp}
\]
of lax symmetric monoidal functors {so that} the structure map ${S} \to R(\unit)$ is an equivalence. Since $1$ is compact in $\bD$ the fuctor $\map_{\bD}(\unit, -)$ and {therefore also the functor $R$ preserve all countable colimits.} {For the lax symmetric monoidal structure on  the functor  $\map_{\bD}(\unit,-)$ see   \cite[Cor. 6.8 (3)]{Nikolaus}. This functor therefore induces a functor $\Mod_{\bD}(\beins)\to \Mod(S)$.   The functor $R$ is then given by  composing  with the identification $\bD\simeq \Mod_{\bD}(\unit)$.}

\begin{ddd} \hfil
	\begin{enumerate}
		\item The bootstrap class $\cB$ in $\bD$ is the full subcategory of $\bD$ generated by the tensor unit under countable colimits.
		\item The Künneth class $\cKue$ in $\bD$ is the full subcategory spanned by the objects $D$ in $\bD$ such that for all $E$ in $\bD$ the structure map
		\[ 
			R(D) \otimes_{{S}} R(E) \to R(D \otimes E)
		\]
		is an equivalence.	
		\item {The class $\cT$ is the class of trace-class exhaustible objects.}
		

	\end{enumerate}
\end{ddd}
\begin{prop}\label{tokhperthrtgertg} \mbox{}
	\begin{enumerate}
		\item  \label{jzlpjrthrtzh} $\cK$ is a stable subcategory closed under countable colimits and   $\cB\subseteq \cK$.
		\item \label{jzlpjrthrtz2}   $\cT$ is a countably cocomplete stable subcategory.
		\item \label{jzlpjrthrtzh1} $\cT\cap \cK= \cB$.
	\end{enumerate}
\end{prop}

\begin{proof}
 {Since $R$   and $\otimes$ preserve countable colimits the Künneth class is closed under countable colimits.
Since the generator  $\unit$ of the bootstrap class is clearly in the Künneth class  
we can conclude 
Assertion \ref{jzlpjrthrtzh}.}

 {For  Assertion \ref{jzlpjrthrtz2}  we invoke}    \cite[Lemma 4.4.11]{som} {implying} that {pushouts of} trace-class exhaustible objects are again trace class exhaustible, that $\Omega$ preserves $\cT$ and finite coproducts of trace class exhaustible objects are clearly trace-class exhaustible.  It remains to show that $\cT$ is closed under countable {coproducts.} Let $(D_n)_{n \in \nat}$ be a family in $\cT$ and choose for every $n $ in $\nat$ a trace-class exhaustion $(D_{n,m})_{m \in \nat}$ of $D_n$. Then the system $(D_{0,n} \oplus \cdots \oplus D_{n,n})_{n \in \nat}$ with the structure maps
  \begin{equation}\label{eronncjkediedfg}
  	D_{0,n} \oplus \cdots \oplus D_{n,n} \xrightarrow{\alpha} D_{0,n} \oplus \cdots \oplus D_{n,n} \oplus D_{n+1,n} \xrightarrow{\beta} D_{0,n+1} \oplus \cdots \oplus D_{n+1,n+1} 
  \end{equation}
  is a trace-class exhaustion of $\bigoplus_{n \in \nat} D_n$. In this formula $\alpha$ is the inclusion of the first $n+1$-summands, and $\beta$ is the sum of the trace-class structure maps {$D_{i,n}\to D_{i,n+1}$}. Since trace-class maps are closed under finite sums and form an ideal the map in \eqref{eronncjkediedfg} is trace-class as required.

%
%
%
%

We now turn to Assertion \ref{jzlpjrthrtzh1}. {We first consider $D$ in $\cT \cap \cK$ satisfying additionally $R(D) \simeq 0$. Since $D$ is in $\cK$ this implies that {every trace-class map} 
 with codomain $D$ vanishes {since}  for all $C$ in $\bD$, we have
$$ 		\map_{\bD}(\unit, C \otimes D) \simeq \map_{\Mod(S)}(S, R(C \otimes D)) \overset{\text{Künneth}}{\simeq} \map_{\Mod(S)}(S, R(C) \otimes_S 0) \simeq 0\ , $$}{
where the first equivalence holds since $R$ is the restriction of the right-adjoint of the right-Bousfield localization $\hat R$ in the diagram \eqref{gregwegwgerf} below.}  
Since $D$ is trace-class exhaustible this implies $D \simeq 0$ {as follows:} We can choose a trace-class exhaustion $(D_{n})_{n\in \nat}$ of $D$. Then
every  structure map $D_n \to {D}$  is trace class and hence trivial. Since a trace-class map is weakly compact  by  \cref{knldsfe0wadsf}.\ref{hrtopgkpergortkgpeorggertg}, for every $n$ in $\nat$   there exists $m$ in $\nat$ large enough such that $D_n \to D_m$ is trivial.  This implies that $D\simeq \colim_{n\in \nat} D_{n}\simeq 0$ as desired.

In light of the above argument, given $D$ in $\cT \cap \cK$ not necessarily satisfying $R(D) \simeq 0$, it suffices to find an approximation $ C \to D$ with $C$ in the bootstrap {class such} that $R(C) \to R(D)$ is an equivalence. 
By  Lurie's version of the Schwede-Shipley theorem we get the right Bousfield localization   in the upper part of the following diagram
 \begin{equation}\label{gregwegwgerf}  \xymatrix{\Mod(S)   \ar@/^0.3cm/[rr]^{\hat F}& {\perp}& \Ind_{\aleph_{1}}(\bD) \ar@/^0.3cm/[ll]^{\hat R}\\  \Mod(S)^{\aleph_{1}} \ar@{^{(}->}[u]   \ar@/^0.2cm/[drr]^{F} \ar@{..>}[rr]^{}&&  \bD\ar@/_-0.3cm/[ull]_(0.3){R}\ar[u]^{y} \\ && \bD' \ar@{^{(}->}[u]\ar@/^0.3cm/[ull]_{\perp}^{R'}}\ .\end{equation}
 This crucially uses the assumption that $\unit$ is compact in $\bD$ which implies that  the tensor  unit $y(\unit)$  of $ \Ind_{\aleph_{1}}(\bD) $ is compact as well.
  The functor $\hat R$ is the canonical refinement of $$\map_{\Ind_{\aleph_{1}}(\bD)}(y(\unit),-):\Ind_{\aleph_{1}}(\bD)\to \Sp$$ to a lax symmetric monoidal  $\Mod(S)$-valued functor. As indicated the restriction of $\hat R$ along $y$ coincides with $R$.
The fully faithful  symmetric monoidal functor $\hat F$ preserves colimits and is uniquely determined by the fact that it sends $R$ to $y(\unit)$.   Its restriction  to $\Mod(S)^{\aleph_{1}}$ (indicated by the dotted arrow)  therefore takes values in the bootstrap class of $\bD$. Here we view $\bD$ as a full subcategory of $\Ind_{\aleph_{1}}(\bD)$ via $y$.

We let  $\bD':=\cK\cap \cT$  denote the
full subcategory of $\bD$ on objects in the Künneth class which  in addition admit a trace class exhaustion. 
A symmetric monoidal functor    preserves trace-class maps. Since $\hat F$ is symmetric monoidal and colimit preserving it preserves trace class exhaustions.
 Since $\Mod(S)^{\aleph_1}$ is rigid  by  \Cref{cijodjvnfue} we see that  the dotted arrow admits a further factorization over a symmetric monoidal functor $F$ as indicated.
 We now show that the restriction $R'$  of $R$ to $\bD'$ takes values in $\Mod(S)^{\aleph_{1}}$. {Let $E$ be in $\bD'$ and choose a trace-class exhaustion $(E_n)_{n \in \nat}$ of $E$. Since $E$ {belongs to the Künneth class}  the maps $R(E_n) \to R(E_{{n+1}})$ are still trace class and in particular compact{. Since} $\Mod(S)$ is compactly generated this implies that $R(E_n) \to R(E_m)$ factorizes over a compact object for $m$ large enough. Thus $R(E)$ lies in $\Mod(S)^{\aleph_1}$.}
 We conclude that the adjunction $\hat F \dashv \hat R$ restricts to an adjunction
 $F\dashv R'$. 
 
 {The object $C:=F  R^\prime(D)$ belongs to the bootstrap class.
 Since $F$ is fully faithful  the counit $C\simeq F R^\prime(D) \to D$ 
 induces an equivalence
 $R(C)   \stackrel{\simeq}{ \to} R(D)$ as desired.}
\end{proof}

\begin{rem}
The argument above implies  an equivalence  $$F:\Mod(S)^{\aleph_{1}}\stackrel{\simeq}{\to} \cT\cap \cK$$   with inverse   $R'$. \hB
\end{rem}


We have the countably cocomplete  stable $\infty$-categories $\KKs$ and $\Es$.
On $\KKs$ we consider the  two symmetric monoidal structures 
$\otimes_{\min}$ and $\otimes$ induced by the minimal and the maximal tensor product on $\nCalg_{\sepa}$, respectively,  while on
$\Es$ we only have the symmetric monoidal structure $\otimes$ induced by the maximal tensor product.
In any case the tensor product preserves countable colimits in each variable. 
The tensor units are given by $\kks(\C)$ and $\es(\C)$, respectivly.
 \begin{lem}\label{kopokherthertget}   The tensor units in $\KKs$ and $\Es$ are compact.
 \end{lem}
 \begin{proof}
 	{The case of $E$-theory is  \Cref{gjiseogpgregesg}.} 
 	
	{For the case of $KK$-theory
	we use that} $K$-theory functor $K:\nCalg_{\sepa}\to \Sp$ is corepresented in $\KKs$ by $\kks(\C)$:
 	$$K(-)\simeq \map_{\KKs}(\kks(\C),\kks(-)):\nCalg_{\sepa}\to \Sp\ .$$
 	Using  the description of   the $K$-theory groups of $C^{*}$-algebras in terms of homotopy classes of projections and unitaries  one can check  (the well-known fact) that $K$ preserves countable  sums. Since $\kks$ is surjective on objects and preserves countable sums  this implies the claim.  
 \end{proof}

We {can 
apply} the general definitions from above to $\KKs$ and $\Es$.  
\begin{ddd}
	\mbox{}
	\begin{enumerate}
		\item We let $\bGsk$ and $\bGs$ denote the bootstrap classes in $\KKs$ and $\Es$.
		\item We let  $\cTmakk$, $\trkma$ and $\cTmikk$, $\trkmi$ denote the {Künneth classes and the classes of trace-class exhaustible objects} 
	 in $\KKs$ for the tensor products $\otimes$ and $\otimes_{\min}$. 
		\item We let $\cTe$ and $\tre$ denote the   {Künneth class and the class of trace-class exhaustible objects}   in $\Es$.
	\end{enumerate}
\end{ddd}

\begin{rem} {Given a separable $C^{*}$-algebra $A$
	   it is an interesting question to  descide  when $ \es(A)$ belongs to $\bGs$. If one knows   that $\es(A)\in \tre$, then \Cref{tokhperthrtgertg} implies that $\es(A)\in \bGs$  if and only if $\es(A)\in  \cTe$.  The latter condition is equivalent to the condition  that the natural transformation
	$$ K(A) \otimes_{KU} K(-) \to K(A \otimes -) $$}of functors from $\nCalg$ to $\Mod(KU)$ is an equivalence{.  This  condition only involves} $K$-theory and not bivariant $E$-theory.  
	
	{It is therefore} an interesting question to determine {for which $C^\ast$-algebras $A$ the $E$-theory class $\es(A)$} admits trace-class a exhaustion. The same applies to $ \kks(A)$ in $\KKs$ with respect to $\otimes $ and $\otimes_{\min}$. \hB
\end{rem}

{A major source of the complexity in the context of $KK$-theory  is the  difference between the classes $ \cTmikk $ and $ \cTmakk $. This    is the content of the following theorem   due to Uuye \cite[Ex. 5.1, Ex. 5.3]{Uuye:2011aa} which  relies} heavily on deep results of Skandalis \cite{zbMATH04065714} and Ozawa {\cite[App. A]{Ozawa_2003}}.
 \begin{theorem}
	The classes $\cTmikk \backslash \cTmakk$ and $ \cTmakk \backslash \cTmikk $ are non-empty.
\end{theorem}

Interesting examples of $C^*$-algebras {whose $KK$-classes belong to} $\cTmikk$ are the reduced group $C^\ast$-algebras {$C^{*}_{r}(G)$}  for groups $G$ satsfying the Baum-Connes conjecture with coefficients. {Using} Skandalis \cite{zbMATH04065714} one can find a group $G$   {satisfing} the Baum-Connes conjecture with coefficients and such that $\kks(C^\ast_r(G))\not\in \cTmakk$. \par
A canonical source   { for objects in} $\cTmakk$ is the kernel of the comparison functor $ c: \KKs \to \Es $ {in \eqref{herjthiertogjertioger}}. An example of a $C^{*}$-algebra {whose $KK$-class is in  the kernel of $c$ but not in  not}   $\cTmikk$ {is given in}   Ozawa {\cite[App. A]{Ozawa_2003}}.

\begin{prop} \label{vbnuiojaweiofhjjdf} \mbox{}
	\begin{enumerate} 
		\item\label{oujiogewrgrewgwregr} The class $ \bGsk $ is strictly smaller than $\KKs$.
		\item\label{oujiogewrgrewgwregr1} The class $\bGs$ is strictly smaller than $\Es$. \label{uirjofjfgjfid}
		\item\label{oujiogewrgrewgwregr2} The class $\cTe$ is strictly smaller than $\Es$. 
		\item\label{oujiogewrgrewgwregr3} The category $\KKs$ is not rigid with respect to $\otimes$. \label{wiofnvgkeqka}
		\item\label{oujiogewrgrewgwregr4} The category $\KKs$ is not rigid with respect to $\otimes_{\min}$. \label{oiiaenwfgonaoej2qr}
	\end{enumerate}
\end{prop}
Let us note that the first three points of the above proposition are well-known.
\begin{proof}
	We first observe that {Assertion} \ref{oujiogewrgrewgwregr} follows from Assertion \ref{oujiogewrgrewgwregr1}, and {Assertion} \ref{oujiogewrgrewgwregr1} in turn follows from {Assertion} \ref{oujiogewrgrewgwregr2} because $\bGs$ is contained in $\cTe$.  {In order to prove {Assertion} \ref{oujiogewrgrewgwregr2}}  we choose, using \cref{vbnuiojaweiofhjjdf}, an object $A$ in $\nCalg_{\sepa}$ such that $\kks(A)\not\in \cTmakk$. Since the comparison functor $c:\KKs \to \Es$ induces an equivalence
	$$  \map_{\KKs}(\kks(\C),\kks(-) )\simeq \map_{\Es}(\es(\C),{\es(-)}) \simeq K(-){:\nCalg_{\sepa}\to \Mod(KU)}  $$
	 {of lax symmetric monoidal functors we conclude that $\es (A) \not\in \cTe$}. 
	  
	   {In order to prove Assertion} \ref{oujiogewrgrewgwregr3} {we choose $ A$ in $\cTmakk \backslash \cTmikk$ using \cref{vbnuiojaweiofhjjdf}} . If ${A}$ admits a trace-class exhaustion with respect to $\otimes$, being in $\cTmakk$ implies by \Cref{tokhperthrtgertg} that $ A $ {belongs to} the bootstrap class. This is a contradiction to the assumption that {$ A\notin \cTmikk$}. \\ \\
	The proof of  {Assertion} \ref{oujiogewrgrewgwregr4}, is, up to switching the role of max and min, the same as the one of  {Assertion} \ref{oujiogewrgrewgwregr3}.
\end{proof}
The above proposition shows that $\Es$ can not be compactly assembled for the trivial reason that the category is equivalent to the compactly generated category $\bGs $. It similarly implies that $\KKs$ can not be generated by dualizable objects, but it leaves the following questions open.
\begin{prob} \label{bijasfiondfmg} \hfil
	\begin{enumerate}
		\item Is $\Es$ compactly generated?
		\item \label{ioqerjtasdfj} Is $\Es$ rigd?
		\item Is $\KKs$ compactly assembled or even compactly generated?
	\end{enumerate}
\end{prob}

An important property for $C^*$-algebras is nuclearity:
\begin{ddd}\mbox{}\begin{enumerate}
		\item    {$A$ in $\nCalg_{\sepa}$ is} nuclear  if  
		the natural transformation
		\begin{equation}\label{gwerferfrefw}(A\otimes -)\to (A\otimes_{\min}-)
		\end{equation}
		of endofunctors of $\nCalg_{\sepa}$  is an isomorphism. 
		\item The $E$-theoretic nuclear class $\nGs$  is  the smallest countably cocomplete stable  subcategory
		of $\Es$ containing $\es(A)$ for all nuclear  $A$ in $\nCalg_{\sepa}$.
		\item The $KK$-theoretic nuclear class $\nGsk$  is  the smallest countably cocomplete stable  subcategory
		of $\Es$ containing $\kks(A)$ for all nuclear  $A$ in $\nCalg_{\sepa}$.
	\end{enumerate}
\end{ddd}


 Note that $\C$ is nuclear.
 We get inclusions 
 $$\bGs\subseteq \nGs\ , \qquad \bGsk\subseteq \nGsk \ .$$
The following is a famous classical question:
\begin{prob} \label{koperthgerth9} Do we have equalities
$$\bGsk= \nGsk , \qquad 
 \bGs= \nGs\ ?$$
\end{prob} 

\begin{prop}
	\mbox{}
	\begin{enumerate}
		\item \label{sdvuiohnoiasjdab} For all $A$ in $\nGsk$ the comparison map $c: \KKs \to \Es$ induces an equivalence
		\begin{equation} \label{wduioickop9jkah}
			\map_{\KKs}(A, -) \to \map_{\Es}(c(A), c(-)) \ .
		\end{equation}
		\item \label{ijkbnadsciouwiosjfa} The assertions $\bGsk= \nGsk$ and  $\bGs= \nGs$ are equivalent.
	\end{enumerate}
\end{prop}
\begin{proof}
	By stability it suffices to show that the map in \eqref{wduioickop9jkah} induces an isomorphism on $\pi_0$. If $A $ is the $KK$-class of a separable nuclear $C^*$-algebra this is a classical result due to Skandalis \cite{zbMATH04065714}. Since the class of $KK$-classes $A$ for which the map in \eqref{wduioickop9jkah} is an equivalence is closed under countable colimits and finite limits the Assertion \ref{sdvuiohnoiasjdab} follows.
	
	We now show Assertion \ref{ijkbnadsciouwiosjfa}. Since the comparison map $c$ preserves countable colimits the $KK$-theoretic assertion implies the $ E$-theoretic one. For the converse we consider the following square
	\begin{equation} \label{asxcjkybniosndfa}		\xymatrix{\bGsk \ar[rr]_{\simeq}^{c_{|\bGsk}} \ar@{^{(}->}[d] & & \bGs \ar@{^{(}->}[d]^{\simeq}  \\  \nGsk \ar[rr]_{\subseteq}^{c_{|\nGsk}}& & \nGs } \ .
	\end{equation}
	The upper functor is an equivalence since both bootstrap classes are generated by the compact objects $\kks(\C)$ or $\es(\C)$, respectively, and the canonical map between their endomorphism algebras is an equivalence. The lower functor is fully faithful by Assertion \ref{sdvuiohnoiasjdab}. In view of the diagram \eqref{asxcjkybniosndfa} the left vertical map is an equivalence.
\end{proof}


The \cref{koperthgerth9} is at least compatible with \Cref{vbnuiojaweiofhjjdf}.\ref{uirjofjfgjfid}.
\begin{prop}
	We have
	\begin{equation}
		\nGsk \not= \KKs \qquad \mbox{and} \qquad \nGs \not= \Es \ .
	\end{equation}
\end{prop} 

These assertions are well-known{, but}  we include an argument for the sake of completeness.

\begin{proof}
	{As the $KK$-theoretic claim is implied by the   $E$-theoretic one we only consider the latter.}     {We pick $A$ in $\nCalg_{\sepa}$ such that $\kks(A)\in \cTmikk \backslash \cTmakk$. Then there exists $B$ in $\nCalg_{\sepa}$ 
 which is not contained in the kernel of  the functor	\begin{equation} \label{voajknioeriojfji}
		F:=\Fib \left( K(A \otimes -) \to K(A \otimes_{\min} -)\right):\nCalg_{\sepa} \to \Mod(KU)\ .
	\end{equation}} Since $A$ is in $\cTmikk$ {the functor $$K(A \otimes_{\min} -):\nCalg_{\sepa}\to \Mod(KU)$$ is exact in addition to} being    homotopy-invariant, $\K$-stable and sum-preserving. The same applies to $K(A\otimes_{\max}-)$.  {The functor $F$ in
  \eqref{voajknioeriojfji} therefore has a factorization
  $$\xymatrix{\nCalg_{\sepa}\ar[rr]^{F}\ar[dr]_{\es}&&\Mod(KU)\\&\Es\ar@{..>}[ur]_{\bar F}&}\ ,$$
  where $\bar F$ is exact and countable colimit preserving.} Since {by definition   every nuclear $C^\ast$-algebra lies in the kernel of  $F$ we can conclude 
  that $\nGs\subseteq \ker(\bar F)$.  As $\es(B)\not\in \ker(\bar F) $   we conclude that  $\es(B)\not\in \nGs$.}
\end{proof}

{In view of}  \Cref{koperthgerth9} also the following more modest version of \Cref{bijasfiondfmg}.\ref{ioqerjtasdfj} is very interesting.
\begin{prob} \label{vuioaherjfoasof}
{Do we have $\nGs\subseteq \tre$.}
\end{prob}
Let us focus on the $E$-theoretic part of \Cref{koperthgerth9}{. If}  the answer turns out to be yes, then the answer to \Cref{vuioaherjfoasof} also has to be yes, simply because every algebra in the bootstrap class is also trace-class exhaustible (see \Cref{tokhperthrtgertg}.\ref{jzlpjrthrtzh1}). On the other hand, if one can {verify} \Cref{vuioaherjfoasof}, then one reduces \Cref{koperthgerth9} to showing that every nuclear $C^\ast$-algebra satisfies {the Künneth formula}. \par

As a final topic in this section, we consider the $\PrLst$-{dual} of $E$ given by
$$ \Fun^{L}(\mathrm{E}, \Sp) \simeq \Fun^{\colim_{\omega}}(\Es,\Sp) \simeq \Fun^{h,K_G,\exa, \cfil}(\nCalg_{{\sepa}}, \Sp), $$
{where} we have combined $\mathrm{E} \simeq \Ind_{\aleph_1}(\Es)$ and the universal property of $\Es$ \Cref{wrejgioweferfw} with {\cref{gjerogrgesrgseg}.\ref{kophertherth9}}. The functors on the right hand side have precisely the abstract properties of $K$-theory {which motivates the following definition.
\begin{ddd} \label{ijvwiopjkdfehj9}
The objects of  $\Fun^{h,K_G,\exa, \cfil}(\nCalg_{{\sepa}}, \Sp)$ are called  generalized $K$-theory functors.
\end{ddd}}

{
The typical example of a generalized $K$-theory functor is $K(- \otimes_{\max} A)$ for a $C^*$-algebra $A$. In fact, we are not aware of a generalized $K$-theory functor that is provably not of this form. This  leads to the following question.
\begin{prob}
	For every generalized $K$-theory functor $L$, does there exist a $C^*$-algebra $A$ such that $K(- \otimes_{\max} A) \simeq L$?
\end{prob}
\begin{rem}
	Another motivation for investigating the above problem is that in the context of localizing motives, it is indeed the case that every finitary and bifiber sequence preserving functor $\stCatp \to \Sp$ is equivalent to $K(- \otimes \bA)$ for some category $ \bA $ in $\stCatp$.  This follows by combining \cite[Theorem 1.5]{RSW} and \cite[Theorem 3.1]{Efimov_Rigid}. \hB
\end{rem}
}
One can can ask whether $\es(A)$ is determined by {the values}  $L(A)$ for  {all 
 generalized $K$-theory functors $L$. Let us emphasize that this question is non-trivial  since    a generalized $K$-theory functor must preserve filtered colimits
 and has codomain $\Sp$.  This excludes the functor $\es$ which otherwise would lead to a tautology.}
\begin{prop}\label{oenvgpwffoq} Let $f: A \to B$ be a map in $\nCalg_{{\sepa}}$.
	\begin{enumerate}
		\item \label{refrefwf} The map $\es(f)$ is an equivalence if and only if  {$L(f)$ is an equivalence for all generalized $K$-theory functors $L  $.}
		\item \label{refrefwf1} The map $\es(f)$ is a strong phantom if and only if    $\pi_{{0}}L(f)$ is {the zero map for all generalized $K$-theory functors $L$}. 
	\end{enumerate}	
\end{prop}
\begin{proof}
	We start by reducing {Assertion \ref{refrefwf} to Assertion \ref{refrefwf1}.} {In order to} prove that $\es(f)$ is an equivalence it suffices to check that its fibre $F$ is $0$. Since $\Es$ is compactly assembled, $F$ is $0$  if and only if it is a strong phantom object. Finally, we note that Assertion \ref{refrefwf1} implies that $F$ is a {strong} phantom object. \par
	 {It   remains} to check Assertion \ref{refrefwf1}.  {But the latter  is a specialization} of \Cref{vbeirjgowjgj}.
\end{proof}

\begin{rem}
	We are not aware of a proof of {\cref{oenvgpwffoq}.\ref{refrefwf}}   that does not rely on invoking that $\Es$ is compactly assembled. {Accordin{g}ly  we do not} know whether the collection of countable colimit preserving functors $\KKs \to \Sp$ jointly detects {equivalences.}
\end{rem}

{As a corollary of dualizability of $\mathrm{E}$  one can even recover  the latter from} the $\infty$-category of generalized $K$-theory functors. {To this end we let}
$$ \ev_{(-)}: \mathrm{E} \to \Fun^L(\Fun^L(\mathrm{E}, \Sp){,\Sp}), \quad \ev_A(L) {:=} L(A) $$
be the canonical {functor}.
\begin{kor}
	The canonical functor
	$$\mathrm{E} \xrightarrow{\ev_{(-)}} \Fun^{L}(\Fun^{L}(\mathrm{E},\Sp){,\Sp}) \simeq \Fun^{L}(\Fun^{h,K_G,\exa, \cfil}(\nCalg_{{\sepa}}, \Sp), \Sp) $$
	is an equivalence.
\end{kor}

\section{Analysis on homotopy categories}\label{tgkowpegfwerwggerfwf}

\subsection{The  light condensed topology}\label{erjogpwergwerfrwefwrefrwefwer}

In this section we introduce the light condensed topology on the $\Hom$-sets of an   $\infty$-category which admits finite products and countable filtered colimits.  For any object $C$ in $\bD$ the functor $\Fin^{\op}\to \bD$ sending a finite  set $X$  to $C^{X}:=\prod_{X}C$
left Kan-extends to light profinite {sets}, i.e., profinite  topological spaces which are  $\nat^{\op}$-indexed limits of finite sets.   One could try to consider  the power $C^{X}$ as the object of continuous functions from  $X$ to $C$. The light condensed topology is a way to make  this idea precise.

In the following sets will be considered as discrete topological spaces.
{The following notion was recently proposed  by Clausen-Scholze as a basis for the development of condensed mathematics.}
\begin{ddd} A light profinite topological space is a topological space which can be presented as  an $\nat^{\op}$-indexed limit of finite sets. \end{ddd}

\begin{rem}If $(X_{i})_{i\in I}$ is a system of finite sets indexed by a cofiltered countable category  $I$, then
$\lim_{i\in I}X_{i}$ is a light profinite space since we can choose a final functor $\nat^{\op}\to I$
and restrict the limit to $\nat^{\op}$ without changeing the resulting space. Furthermore, for $C$ in $\bD$ we have an equivalence $$C^{X}\simeq \lim_{i\in I} C^{X_{i}}\ .$$
\hB
 \end{rem}
 
We let $\Profinl$  denote the full subcategory of $\Top$ of light profinite spaces.

\begin{ex}\label{koergperfrefqfrfq}
The one-point compactification $\nat_{+}$ of the natural numbers is an object of $\Profinl$.
It can be presented as a limit of the system of finite sets
\begin{equation}\label{eoirthjerpthetrge}\dots\to \{0,1,2,\infty\}\to \{0,1,\infty\}\to \{0,\infty\}\to \{\infty\}\ ,
\end{equation} 
where the map $\{0,\dots,n,\infty \}\to \{0,\dots,n-1,\infty\}
$ sends $i$ to $i$ for $i\le n-1$, and $n$ and $\infty$ to $\infty$. 

The light profinite space $\nat_{+}$ corepresents converging sequences. In other words, 
for any topological space $X$, a converging sequence $(x_{n})_{n\in \nat}$ with limit $x$  in $X$ is the same as a continuous map
$\nat_{+}\to X$ sending $n$ to $x_{n}$ and $\infty$ to $x$.
\hB\end{ex}

\begin{constr}{\em  Let $\Fin$ denote the category of finite sets.
For any $\infty$-category  $\bD$ with finite products  we  have a power functor
$$\Fin^{\op}\times \bD\to \bD\ , \quad (X,C) \mapsto C^{X}$$
which is characterised by a natural equivalence
$$  \map_{\bD}(B,C^{X})\simeq \prod_{X} \map_{\bD}(B,C) $$
for any  $B$ in $\bD$.
If $\bD$  admits countable colimits, then we can 
left Kan-extend the power functor  in the first variable:
  \begin{equation}\label{qewfojpqdwqwededqwdqweded}
 \xymatrix{ \Fin^{\op}\times \bD \ar[rr]^{(X,C)\mapsto C^{X}}\ar[dr] & & \bD \\ &  (\Profinl)^{\op}\times  \bD \ar@{..>}[ur]_{(X,C)\mapsto C^{X}} &}\end{equation} 
  \hB
}
\end{constr}

Let $\bD$ be an $\infty$-category with finite products and countable filtered colimits. For   objects $B,C$ of $\bD$ we 
 write $$[B,C]:=\pi_{0}\Map_{\bD}(B,C)$$ for the set of equivalence classes of maps from $B$ to $C$   in the homotopy category of $\bD$. If $\bD$ is pointed or stable, then $[B,C]$ is naturally a pointed set or  abelian group, respectively.
Using \eqref{qewfojpqdwqwededqwdqweded}  we   get a functor   $$ (\Profinl)^{\op}\ni X\mapsto [B,C^{X}]\in \Set\ .$$
The  collection of inclusions $(i_{x}:*\to X)_{x\in X}$ of the points of $X$   induces
 the evaluation     \begin{equation}\label{wergwergerferfwefwefwf}\ev:[B,C^{X}]\to \Hom_{\Set}(X,[B,C])\ , \quad \tilde f\mapsto (X\ni x\mapsto i_{x}^{*}\circ \tilde f\in [B,C])\ .
\end{equation}    
 \begin{ddd}
 The image of the map $\ev:[B,C^{X}]\to \Hom_{\Set}(X,[B,C])$ is called the set of   representable maps.
 \end{ddd}

  \begin{ddd}
 The light condensed topology on   $[B,C]$  is defined as the maximal topology such that all representable maps $X\to [B,C]$  from light profinite spaces $X$  are continuous.
  \end{ddd}
We will write $[B,C]^{\cd}$ for the set $[B,C]$ equipped with the light condensed  topology.

\begin{rem}\label{oijrfopqrfewdqwfqfef}
The following is just a reformulation of the definition:
 For any topological space $Y$ a map $[B,C]^{\cd}\to Y$  is continuous if and only if  the composition $X\to [B,C]^{\cd}\to Y$ is continuous for  all light profinite spaces  $X$ and  representable maps $X\to [B,C]$.
\hB
\end{rem}

\begin{rem}
 In general it is not clear that every continuous map $X\to [B,C]^{\cd}$ from a  light profinite space $X$
is representable. But if $\bD$ is stable, and  if $B$ admits a weakly compact approximation (see \cref{jigowergwerfewrfwerf}),  then by \cref{jewrogfwerfrfwf} this is true up to a map with values in strong phantoms.
Indeed, by  \cref{woekrpgerfwrefrewf} the map $X\to  [B,C]^{\sh}$ is also continuous, and  \cref{jewrogfwerfrfwf} provides
the corresponding preimage under the evaluation  \eqref{wergwergerferfwefwefwf} in 
 $[B,C^{X}]$.
 \hB
\end{rem}

 \begin{rem}\label{oijrfopqrfewdqwfqfefseq}
 In order to see that a sequence $(g_{n})_{n\in \nat}$  in $[B,C]^{\cd}$ converges to $g$ in $[B,C]^{\cd}$  
 it suffices to provide  a map $\tilde g$ in $[B,C^{\nat_{+}}]$ with evaluations $g_{n}$ at the points $n$ and
 $g$ at $\infty$. We shall see  in \cref{qreokgprefqewfqewfqd} that if $B$ admits a weakly compact approximation, then the converse is also true, i.,e., if the sequence $(g_{n})_{n\in \nat}$ converges to $g$, then there exists such a map $\tilde g$.
\hB
\end{rem}

    \begin{lem}\label{wetkopgwegrefrewferf} For all $A,B,C$ in $\bD$
  the composition $$[{B},{C}]^{\cd}\times [{A},{B}]^{\cd}\to [A,C]^{\cd}$$ is separately continuous.
  \end{lem}\begin{proof}
We show continuity in the {left} argument.  
  We fix an element $f$ in $[A,B]$ and must show  that $-\circ f:[B,C]^{\cd}\to [A,C]^{\cd}$ is continuous. 
  To this end we must show that the map
  $X\xrightarrow{\ev(\tilde g)} [B,C]^{\cd}\xrightarrow{-\circ f}[A,C]^{\cd}$ is continuous for every  light profinite spaces  $X$  and $\tilde g$ in  $[B,C^{X}]$.  But this composition is  equal to $\ev(\tilde g\circ f)$ and therefore continuous by definition.

  We now show   continuity in the {right} argument.  
  We fix an element $g$ in $[B,C]$ and must show  that $g\circ -:[A,B]^{\cd}\to [A,C]^{\cd}$ is continuous. 
  To this end we must show that the map
  ${X}\xrightarrow{{\ev(\tilde f)}} [A,B]^{\cd}\xrightarrow{g\circ -}[A,C]^{\cd}$ is continuous for every light profinite spaces $X$  and
  $\tilde f $ in $[A,B^{X}]$. But this composition is equal to $\ev( g^{X}\circ \tilde f)$  and therefore continuous by definition.
     \end{proof}

Let $\phi:C\to D$ be a morphism in $\bD$ and $A,B$ objects.
\begin{lem}\label{wetijgowegegwerf}
The map $-\times \phi:[A,B]^{\cd}\to [A\times C,B\times D]^{\cd}$ is continuous.
\end{lem}
\begin{proof}
Let $X$ be a light profinite space and $f:X\to [A,B]$ be a continuous map represented by 
$\tilde f$ in $[A,B^{X}]$. Then
the image of $\tilde f\times \const(\phi)$  in
$[A\times C,B^{X}\times D^{X}]\cong [A\times C,(B \times D)^{X}]$
represents the continuous map $f\times \phi$.
 \end{proof}

  Let $f:A\to B$ be a morphism in $\bD$, $C$ be an object of $\bD$, and let $X$ be a light profinite space.\begin{lem}\label{woigpregferferwfref}
 {We assume that $f$ is weakly compact}. If  $X\to [B,C]^{\cd}$ is representable,  then the composition $X\to [B,C]^{\cd}\to [A,C]^{\cd}$
 factorizes over a finite quotient of $X$.   \end{lem}
\begin{proof}
Let $X\cong \lim_{i\in \nat^{\op}} X_{i}$ be a presentation of $X$ as a limit  of  finite sets with surjective structure maps.
Then $C^{X}\simeq \colim_{i\in \nat} C^{X_{i}}$. We consider the square 
$$\xymatrix{A\ar@{..>}[r]\ar[d]^{f} & C^{X_{i}}\ar[d] \\ B\ar[r] &C^{X} }\ , $$
where $i$ in $\nat$ and the dotted completion exists by the weak compactness of $f$.
This means that $ X\to [A,C]$ factorizes over the finite quotient  $X\to X_{i}$. 
 \end{proof}

 For the following we assume that $\bD$ is pointed and recall \cref{ijwqfofjewfqedfq9} of a strong phantom map.

 \begin{lem}\label{roeigwprefrewfrewfwrfv}
 The  
  kernel of  $\ev:[B,C^{X}]\to  \Hom_{\Set}(X,[B,C])$   consists of strong phantom maps. 
  
 \end{lem}
\begin{proof}
 Let $g$ be in $[B,C^{X}]$ and assume that $\ev(g)=0$. Then 
$i^{*}_{x}(g)=0$ for all $x$ in $X$. 
Let $f:A\to B$ be a weakly compact morphism.
We must show that $A\xrightarrow{f} B\xrightarrow{g}C$ vanishes.

Let $X\cong \lim_{i\in \nat^{\op}} X_{i}$ be a presentation of $X$
 as a countable limit of finite sets.  We can assume without loss of generality that the structure maps $X \to X_i$ are surjective for all $i$ in $\nat$. 
  We   have  an equivalence $C^{X}\simeq \colim_{i\in \nat} C^{X_{i}}$. Since $f$ is weakly compact 
we have a factorization
$$\xymatrix{A\ar@{..>}[r]^{h}\ar[d]^{f} &C^{X_{i}} \ar[d] \\B \ar[r]^{g} & C^{X}}$$  for some $i$ in $I$.  Let $ i_{x_{i}}:*\to X_{i}$ be the inclusion of some point. 
{By assumption there} exists a preimage $x$ of $x_{i}$ under $X\to X_{i}$. We then have 
  $0=i_{x}^{*}(g\circ f)=i_{x_{i}}^{*}h$.  Since $ ( i_{x_{i}}^{*})_{x_{i}\in X_{i}}:[A,C^{X_{i}}]\stackrel{\cong}{\to} \prod_{x_{i}\in X_{i}} [A,C]$ is an isomorphism
 we can thus conclude that $h =0$.
 This implies $g\circ f= 0$. 
  \end{proof}

 \begin{rem}
 One could ask under which conditions the map $[B,C^{X}]\to \Hom_{\Set}(X,[B,C])$ in \eqref{wergwergerferfwefwefwf}  is injective.
 This is clearly the case when $X$ is finite. But already in the case of  $X=\nat_{+}$ we get the map $[B,  C^{\nat_{+}}]\to \prod_{ x\in \nat_{+}}[B,C]$.  For giving a counter example  we assume that $\bD$ is stable and admits countable products. 
 Then  $\bigoplus_{x\in X} {C}$ is a summand in $C^{\nat_{+}}$.
 If ${B}$ is not a compact object in $\bD$, then we can not exclude the possibility of having a non-zero map ${B}\to 
 \bigoplus_{n\in \nat} C$ whose components are all zero. Consider e.g. the boundary map
 $\Omega (\prod_{n\in \nat}C/\bigoplus_{n\in \nat}C)\to \bigoplus_{n\in \nat}C$ of the fibre sequence
 extending the canonical map $ \bigoplus_{n\in \nat}C\to \prod_{n\in \nat}C$.   \hB
 
 \end{rem}

 For the rest of this section we  assume that $\bD$ is stable. Then  we have abelian groups $[A,B]$ and \cref{wetkopgwegrefrewferf} has the following consequence.

  \begin{kor}\label{wetkopgwegrefrewferf1} 
 The addition $[A,B]^{\cd}\times [A,B]^{\cd}\to [A,B]^{\cd}$ is separately continuous.
 \end{kor}
 \begin{proof}
 The sum is given by the composition
 \begin{eqnarray}\label{korpwegwregfefwef}
[A,B]^{\cd}\times [A,B]^{\cd}&\stackrel{(-\oplus \id_{B})\times \id_{[A,B]^{\cd}}}{\to}&[A\oplus  B,B\oplus  B]^{\cd}\times [A,B]^{\cd}\\&\stackrel{\id_{[A\oplus  B,B\oplus  B]^{\cd}}\times (\id_{A}\oplus - )}{\to}&
[A\oplus  B,B\oplus  B]^{\cd}\times [A\oplus  A ,A\oplus  B]^{\cd}\nonumber\\
&\stackrel{\circ}{\to}& [A\oplus  A,B\oplus  B]^{\cd}\nonumber\\&\stackrel{\mathrm{fold}\circ - \circ  \diag}{\to}&[A,B]^{\cd}\nonumber\ .
\end{eqnarray} {We conclude}
by a multiple application of \cref{wetijgowegegwerf} and \cref{wetkopgwegrefrewferf}.
 \end{proof}

\begin{rem}In view of 
  \cref{wetkopgwegrefrewferf} and \cref{wetkopgwegrefrewferf1} 
 one could ask whether the composition or addition is continuous.
The difficulty consists of 
 understanding continuous maps out of the product. One can check that the composition
$X\to [B,C]^{\cd}\times [A,B]^{\cd}\xrightarrow{\circ} [A,C]^{\cd}$ is continuous if the first map has representable
components. But this does not {suffice} to conclude continuity of $\circ$.
As a consequence we do not know whether $[B,C]^{\cd}$ is a topological group in general. 
 \hB
\end{rem}

Recall the notion of a summable sequence from \cref{erijogoewrpgfrefw1}.\ref{erijogoewrpgfrefw}.
We assume that $\bD$ is stable.
 \begin{lem}\label{jirofwfwrefwrfwrefw}
 If $(g_{n})_{n\in \nat}$ is a summable sequence in $[B,C]$, then $\lim_{n\to \infty} g_{n}=0$ in $[B,C]^{\cd}$.
 \end{lem}
\begin{proof}
We have a fibre sequence 
\begin{equation}\label{advsiojioqefasdvadsvs}\bigoplus_{\nat}C\xrightarrow{i} C^{\nat_{+}}\xrightarrow{i_{\infty}^{*}} C\ .
\end{equation}
Let  $\hat g:B\to \bigoplus_{\nat}C$ witness the summability of the sequence.
 We then get a map $\tilde g:=i\circ \hat g:B\to C^{\nat_{+}}$ which vanishes at $\infty$. 
 \end{proof}

 \begin{rem} We will see in \cref{rjiegowfdwewdqf} that the converse  is also true provided $B$ admits a weakly compact approximation.  \hB \end{rem}

\subsection{The shape topology}\label{werjigowegerfwerfwrefw}

Let $\bD$ an $\infty$-category with countable filtered colimits. 
 In this section we introduce
 the shape topology on the  $\Hom$-sets $[B,C]$ of $\bD$ as the minimal topology such that
 the maps $[B,C]\to [A,C]^{\disc}$ become continuous for all weakly compact maps $A\to B$.
  This topology is very well-behaved and coincides with the light condensed topology provided $\bD$ has finite products and   $B$ has a weakly compact approximation, see \cref{jigowergwerfewrfwerf}.
Similarly, many further assertions shown in this section depend on the assumption that certain objects admit weakly compact approximations.     But 
 note that if $\bD$ is compactly assembled, then every object of $\bD$ has a weakly compact approximation.

For objects $A,C$ of $\bD$ we write $[A,C]^{\disc}$ in order to stress that we consider this set with the discrete topology. 

Let $B,C$ be two objects of $\bD$.
 \begin{ddd}
 The  shape topology on the set $[B,C]$ is the minimal topology  such that
 the map $[B,C]^{\sh}\to [A,C]^{\disc}$ is continuous  for every weakly compact map $A\to B$.
 \end{ddd}
\begin{rem}\label{okqrepgqdedqwe}
The following is just a reformulation of the definition:
 For any topological space $X$ a map $X\to [C,B]^{\sh} $  is continuous if the composition $X\to [B,C]^{\sh}\to [A,C]^{\disc}$ is    continuous for every weakly compact map $A\to B$.
\hB \end{rem}
\begin{lem} \label{rqjifoergstgsgfgfg}
Assume that $(B_{n})_{n\in \nat}$ is a weakly compact approximation of $B$. Then the shape topology
on $[B,C]$ is the minimal topology such that the maps $[B,C]^{\sh}\to [B_{n},C]^{\disc}$ are continuous.
\end{lem}
\begin{proof}
Let $[B,C]^{\mathrm{ap}}$ denote the set $[B,C]$ equipped with the minimal topology such that the maps $[B,C]^{\sh}\to [B_{n},C]^{\disc}$ are continuous for all $n$ in $\nat$.
It is clear that the identity  map $[B,C]^{\sh}\to [B,C]^{\mathrm{ap}}$ is continuous.

Assume now that $f:A\to B$ is a weakly compact map. Then there exists a factorization
$$\xymatrix{A\ar[rd]^{f}\ar@{..>}[rr]^{\tilde f}&& B_{n}\ar[dl]\\&B&}$$
for some $n$ in $\nat$. We get a factorization 
$$[B,C]^{\mathrm{ap}}\to [B_{n},C]^{\disc}  \xrightarrow{(\tilde f)^\ast}  [A,C]^{\disc}$$
 of $f^{*}$ as a composition of two continuous maps.
This shows that the identity  map $[B,C]^{\mathrm{ap}}\to [B,C]^{\sh}$ is also continuous.
\end{proof}

\begin{kor} \label{wqefu0dewdewdqwdqewd}
If $B$    admits a weakly compact approximation, then  for every $C$ in $\bD$  {the} topological space $[B,C]^{\sh}$ has countable neighbourhood bases.
\end{kor}

\begin{prop}\label{worpjegperfrfwefewrf} If  $B$  admits a weakly compact approximation,
then the composition $$[A,B]^{\sh}\times [B,C]^{\sh}\to [A,C]^{\sh}$$ is continuous.
\end{prop}
\begin{proof}
 In view of \cref{okqrepgqdedqwe}, for every weakly compact map
  $A'\to A$ we must show that the restriction
$$[A,B]^{\sh}\times [B,C]^{\sh}\to [A,C]^{\sh}\to [A',C]$$ is locally constant.
Let $(f,g)$ be a point in  $[A,B]^{\sh}\times [B,C]^{\sh}$. Then there exists a neighbourhood $U_{f}$ of $f$
on which $[A,B]^{\sh}\to [A',B]$ is constant with value $f_{|A'}$.  Since $B$ admits a weakly compact approximation
we can choose a factorization
 $$ \xymatrix{A'\ar[r]^{f'}\ar[d] & B'\ar[d] \\A \ar[r]^{f} & B} $$
 for some weakly compact map $B'\to B$. There exists a neighbourhood $V_{g}$ of $g$
  on which $[B,C]\to [B',C]$ is constant with value $g_{|B'}$.
  It follows that the composition map  is constant on the   neighbourhood   $U_{f}\times V_{g}$ of the point $(f,g)$.
\end{proof}

Let $\phi:C\to D$ be a morphism in $\bD$ and $A,B$ objects.
\begin{lem}\label{wetijgowegegwerf1}
The map $-\times \phi:[A,B]^{\sh}\to [A\times C,B\times D]^{\sh}$ is continuous.
\end{lem}
\begin{proof}
Let $(s ,t):E\to A\times C$ be a weakly compact morphism. Then $s$ is weakly compact and we have a commutative square
$$\xymatrix{[A,B]^{\sh}\ar[r]^-{-\times \phi}\ar[d]^{s^{*}}&[A\times C,B\times D]^{\sh}\ar[d]^{(s,t)^{*}}\\ [E,B]^{\disc}\ar[r]^-{(-, t^{*}\phi)}&[E,B\times D]^{\disc}}\ .$$
We must show that the right-down composition is continuous, but this clear since the down-right composition is continuous by definition of the shape topology on $[A,B]^{\sh}$.
 \end{proof}

In order to consider  in addition the  light condensed topology on the $\Hom$-sets of $\bD$ we now assume that $\bD$ has finite products in addition to countable filtered colimits.
\begin{prop}\label{rejiogerfwerfwrefrwef}
If $B$ has a weakly compact approximation, then for every $C$ in $\bD$  the light condensed and shape topologies on the mapping groups $[B,C]$ coincide.
\end{prop}
\begin{proof}
The following lemma shows that the condensed topology is always finer than the shape topology.
\begin{lem}\label{woekrpgerfwrefrewf} For all $B,C $ in $\bD$ such that $B$ has a weakly compact approximation
the identity map of underlying sets
is a continuous map 
$$[B,C]^{\cd}\to [B,C]^{\sh}$$
\end{lem}
\begin{proof}
We use \cref{okqrepgqdedqwe}. We must show that for every light profinite space $X$ and continuous map 
$X\to [B,C]^{\cd}$  which is represented by an element of $[B,C^{X}]$, and  any weakly compact map $A\to B$, the composition $$X\to [B,C]^{\cd}\to [B,C]^{\sh}\to [A,C]$$
is locally constant.  But this follows from \cref{woigpregferferwfref}.\end{proof}

In order to show that $[B,C]^{\sh}\to [B,C]^{\cd}$ is also continuous we will show that
converging sequences in $[B,C]^{\sh}$ also converge in $[B,C]^{\cd}$. Since by \cref{wqefu0dewdewdqwdqewd}  the shape topology has countable neighbourhood bases  this implies continuity. 

   Let $(f_{i})_{i\in \nat}$ be a sequence which converges to $f$  in $[B,C]^{\sh}$.

 \begin{lem}\label{qreokgprefqewfqewfqd}
If $B$ has a weakly compact approximation, then
there exists $\tilde f$ in $  [B,C^{\nat_{+}}]$ with  values $f_{i}$ at all integers $i$  and the evaluation
$f$ at $\infty$.
\end{lem}
\begin{proof} We use the presentation of $\nat_{+} $   
	explained in \cref{koergperfrefqfrfq}.  We  choose a  weakly compact approximation $(B_{n})_{n\in \nat}$  of $B$.
 	We can then choose a strictly increasing sequence $(i(n))_{n\in \nat}$ of integers such that
	  $$f_{i(n)|B_{n}}\simeq f_{i(n)+1|B_{n}}\simeq  f_{i(n)+2|B_{n}}\simeq \dots\simeq f_{|B_{n}}$$
	  for every $n$ in $\nat$.  For every $n$ in $\nat$	we  define the morpism
	$$\hat f_{n}:=((f_{0})_{|B_{n}},\dots, (f_{i(n)})_{|B_{n}},(f_{|B_{n}})):B_{n}\to C^{\{0,\dots,i(n),\infty\}}$$
	Next we provide commutative squares
	$$\xymatrix{B_{n}\ar[r]^-{\hat f_{n}}\ar[d]&\ar[d]^{s^\ast}C^{\{0,\dots,i(n),\infty\}}\\B_{n+1}\ar[r]^-{\hat f_{n+1}}&C^{\{0,\dots,i(n+1),\infty\}},}$$ 
	where $s: \{0,\dots,i(n),\infty\} \to \{0,\dots,i(n+1),\infty\}$ is the structure map. This is equivalent to providing for each $k$ in $\{0,\dots,i(n+1),\infty\}$ a commutative {diagram}
	$$\xymatrix{ B_{n} \ar[rd]^{f_{s(k)|B_{n}}} \ar[r]^-{\hat f_{n}}\ar[d]&\ar[d]^{\pr_{s(k)}} C^{\{0,\dots,i(n),\infty\}}\\ B_{n+1}\ar[r]^-{f_{k |B_{n+1}}} & C}\ .$$ 
	For the right upper triangle, we choose the canonical filler. For the left lower triangle, we can also choose the canonical filler if $s(k) = k$ (i.e., if $k$ is in $\{0,\dots,i(n)-1,\infty\}$). If $k$ is in $ \{i(n), \dots,i(n+1)\}, $ then $s(k)= \infty$ and we choose an arbitrary filler which exists by the stabilization assumption. We thus get a map of systems 
	$\hat f:(B_{n})_{n\in \nat}\to (C^{\{0,\dots,i(n),\infty\}})_{n\in \nat}$.
	
	We claim that the map 
	$$\tilde f: B\simeq \colim_{n\in \nat} B_{n}\xrightarrow{\colim_{\nat}\hat f} \colim_{n\in \nat} C^{\{0,\dots,i(n),\infty\}}\simeq C^{\nat_{+}} $$   has the desired properties.
	For $k$ in $\nat$ the composition $i_{k}^{*}\circ \tilde f:B\to C^{\nat_{+}}\to C$
	is given by the colimit  of {a} system  {$(f_{k|B_{n}})_{\{n\in \nat \mid k \leq i(n)\}}$} which coincides with the  {constant} system 
	$(f_{k|B_{n}})_{n\in \nat}$  on the {partially ordered subset}  {$ \{n\in \nat \mid k \leq i(n-1)\}$}. We conclude that $i_{k}^{*}\circ \tilde f\simeq f_{k}$ for all $k$ in $\nat$.
	Furthermore, $i_{\infty}^{*}\circ \tilde f$ is the colimit of the {constant} system $(f_{|B_{n}})_{n\in \nat}$ and therefore 
	$i_{\infty}^{*}\circ \tilde f\simeq f$.
\end{proof}

Using \cref{oijrfopqrfewdqwfqfefseq} and \cref{qreokgprefqewfqewfqd} we see that if a sequence in $[B,C]^{\sh}$ converges, then it converges in $[B,C]^{\cd}$. This finishes the proof of \cref{rejiogerfwerfwrefrwef}.  \end{proof}

We consider $B,C$ in $\bD$ and 
let $(B_{n})_{n\in \nat}$ be a weakly  compact approximation of $B$. 
\begin{lem}\label{thkjptherhrtgergeg}
 The canonical map $p:  [B,C]^{\sh}\to \lim_{n\in \nat}[B_{n},C]^{\disc}$ presents  the target as the Hausdorff quotient of the domain.
\end{lem}
\begin{proof} 
The map $p $ is surjective and open by definition of the shape topology.
Let $p_{n}:[B,C]^{\sh}\to [B_{n},C]^{\disc}$ denote the canonical map.
Let $f$ be a point in $[B,C]^{\sh}$.  Then by \cref{rqjifoergstgsgfgfg} we have an equality ${p^{-1}(p(f))=}\bigcap_{n\in \nat} p^{-1}_{n}(\{p_{n}(f)\})=\overline{\{f\}}$.  Any continuous map $[B,C]^{\sh}\to X$ to a Hausdorff space 
therefore uniquely factorizes set-theoretically  through the quotient  $p$ and it is straightforward to see that this factorization is again continuous since $p$ is open.
 \end{proof}

Form now on we assume that $\bD$ is stable and countably cocomplete.

\begin{kor}\label{rijgoowegwefewrfwrf}
 $[B,C]^{\sh}$ is a topological abelian group for all $B,C$ in $\bD$. 
\end{kor}
\begin{proof}
Let $A\to B$ be a weakly compact morphism.
Then we have a commutative square
$$\xymatrix{[B,C]^{\sh}\times [B,C]^{\sh}\ar[r]^-{+}\ar[d]& [B,C]^{\sh}\ar[d]\\[A,C]^{\disc}\times [A,C]^{\disc}\ar[r]^-{+}&[A,C]^{\disc}}\ .$$
We must show that the right-down composition is continuous, but this is clear since 
the down-right composition is continuous by definition of the shape topology on $ [B,C]$.
\end{proof}

We consider $B,C$ in $\bD$ and 
let $(B_{n})_{n\in \nat}$ be a weakly  compact approximation of $B$. Then we
 have a Milnor exact sequence
\begin{equation}\label{vfqefcfdvsfdvre}0\to {\lim}^{1}_{n\in \nat} [B_{n},\Omega C]\to  [B,C]\to \lim_{n\in \nat}[B_{n},C]\to 0\ .
\end{equation} 
\begin{prop}\label{jierogwergwrefrewfwrf}
The closure of $0$ in $[B,C]^{\sh}$ is the subgroup of strong phantom maps and coincides with the image of the first map in the Milnor sequence.  
\end{prop}
\begin{proof}
It is clear from \cref{wekgowpefefwrf} that the kernel of the second map of the Milnor sequence consists precisely of the strong  phantoms. In view of the   definition of the shape topology and  \cref{rqjifoergstgsgfgfg} this kernel is the intersection of all closed neighbourhoods of $0$ 
  and thus coincides {with the} closure of $0$.
 \end{proof}

\begin{kor}\label{rjiegowfdwewdqf}  If $B$ has a weakly compact approximation, then a   sequence in $[B,C] $ is summable if and only if it converges to $0$ in  $[B,C]^{\sh}$.
  \end{kor}

\begin{proof}
By \cref{jierogwergwrefrewfwrf} almost all members of a zero sequence in $[B,C]^{\sh}$ 
are strong phantoms. By \cref{iujgofffsdafdfaf} such a sequence is summable.
Conversely, by \cref{jirofwfwrefwrfwrefw}  a summable sequence  in $[B,C]$ converges to zero 
in $[B,C]^{\cd}$, and therefore also in $[B,C]^{\sh}$.
\end{proof}

 The following is a generalization of \cref{qreokgprefqewfqewfqd} from $\nat_{+}$ to general   light profinite spaces $X$.
 But note that here we assume that $\bD$ is stable and the conclusion is weaker.
 \begin{prop} \label{jewrogfwerfrfwf}
  If $B$ has a weakly compact approximation, then for
any continuous map $f:X\to [B,C]^{\sh}$  there exists a map $\phi:X\to [B,C]^{\sh}$ with values in strong phantom maps
such that $f-\phi$ is representable.   
  \end{prop}
\begin{proof}  We choose a weakly compact approximation $(B_{n})_{n\in \nat}$ of $B$.
 Let $f:X\to [B,C]^{\sh}$ be a continuous map.  
By \cref{okqrepgqdedqwe} the restrictions $f_{|B_{n}}:X\to [B_{n},B]$ are locally constant.
Since $X$ is a light profinite   space every locally constant map out of $X$ factorizes over a finite  quotient. Let $(X\to X_{i})_{i\in I^{\op}}$ be the cofiltered system of finite quotients of $X$.
Then $C^{X}\simeq \colim_{i\in I }C^{X_{i}}$.
The map $f$ gives thus rise to an element in $(f_{n})_{n\in \nat^{\op}}$ in $ \lim_{n\in \nat^{\op}}  \colim_{i\in I}[ B_{n},C^{X_{i}}]$, where $f_{n}$ is represented by the map $B_{n}\to C^{X_{i(n)}}$ representing $f_{|B_{n}}$,
and where $X\to X_{i(n)}$ is precisely the quotient over which the latter factorizes.
By the surjectivity of the second map in the  Milnor sequence \eqref{vfqefcfdvsfdvre}   the image of $(f_{n})_{n\in \nat^{\op}}$ under
the canonical map 
$$  \lim_{n\in \nat^{\op}}  \colim_{i\in I}[ B_{n},C^{X_{i}}]\to  \lim_{n\in \nat^{\op}} [ B_{n},C^{X}]$$        lifts to an element  $\hat f$ in $[B,C^{X}]$. We let $f':=\ev(\hat f) :X\to [A,B]^{cd}$. Then by  \cref{jierogwergwrefrewfwrf} the difference $\phi:=f-f'$ takes values in strong phantom maps.
 \end{proof}

\subsection{Hausdorff quotients}\label{erjogiopwregwrefwref}
 
 Let $\bD$ be  an $\infty$-category  with countable filtered colimits.
 In this section we analyse the Hausdorff quotients of the spaces $[B,C]^{\sh}$.
 
 Let $B,C$ be objects of $\bD$.
 
 \begin{ddd}
 We let $\overline{[B,C]}$ denote the Hausdorff quotient of $[B,C]^{\sh}$.
 \end{ddd}
 
 \begin{lem}\label{wjeirgowergregwf9} If $A,B$  admit weakly compact approximations, then
 the  composition map has  a continuous factorization  $$\xymatrix{[B,C]^{\sh}\times [A,B]^{\sh}\ar[r]\ar[d] &[A,C]^{\sh} \ar[d] \\ \overline{[B,C]}\times \overline{[A,B]}  \ar[r] &\overline{[A,C]} } $$
 over the Hausdorff quotients.
  \end{lem}
  \begin{proof}
  The upper horizontal map is  continuous by \cref{worpjegperfrfwefewrf}.
This implies that the composition map in the lower sequence is continuous, too, since the maps to the Hausdorff quotients are open as seen in the proof of  \cref{thkjptherhrtgergeg}.
  \end{proof}

    We consider a system $(B_{n,m})_{(n,m)\in \nat\times \nat}$. We assume that $(B_{n,m})_{m\in \nat}$
    is a weakly compact approximation of $B_{n}:=\colim_{m\in \nat} B_{n,m}$.
   Finally we set $B:=\colim_{n\in \nat} B_{n}$ and 
   let $C$ be in $\bD$.
    \begin{lem}\label{erokjgperwgrewgwregw9}
    We have an isomorphism of topological spaces
    $$\overline{[B,C]}\stackrel{\cong}{\to} \lim_{n\in \nat}\overline{[B_{n},C]}\ .$$ 
    \end{lem}
    \begin{proof} 
    The diagonal $\nat\to \nat\times \nat$ is cofinal and
    $(B_{n,n})_{n\in \nat}$ is a weakly compact approximation of $B$.
  We get isomorphisms
    \begin{align*}
\overline{[B,C]}& \cong \lim_{n\in \nat^{\op}} [B_{n,n},C]^{\disc} \cong
 \lim_{(n,m)\in \nat^{\op}\times \nat^{\op} } [B_{n,m},C]^{\disc}\\ & \cong
 \lim_{n\in \nat^{\op}} \lim_{m\in \nat^{\op}} [B_{n,m},C]^{\disc}  \cong
  \lim_{n\in \nat^{\op}}\overline{[B_{n},C]}
\end{align*}
   using \cref{thkjptherhrtgergeg} in the first and last step.\end{proof}
    
    The main point which makes \cref{erokjgperwgrewgwregw9} work is that the system $(B_{n})_{n\in \nat}$
    has a functorial choice of systems of compact approximations. It is not clear that
    such a choice exists in general. But if the objects $B_{n}$
  admit shapes for all $n$ in $\nat$ (see \cref{ekrogpegegferfwf}), then we can rigidify the choice of approximations by assuming that they represent the shapes.

    \begin{prop}\label{wokpteghwergwerfwerg9}
For any system $(B_{n})_{n\in \nat}$  in $\bD$ with $\colim_{n\in \nat} B_{n}\simeq B$    such that $B_{n} $ belongs to $\bD^{\shp}$ for all $n$ in $\nat$  we have an equivalence
    $$\overline{[B,C]}\stackrel{\cong}{\to} \lim_{n\in \nat^{\op}}\overline{[B_{n},C]}\ .$$ 
\end{prop}
\begin{proof}
We get a system $(S(B_{n}))_{n\in \nat}$ in $\Ind^{\aleph_{1}}(\bD)$. For every $n$ we choose a system $(B_{n,m}')_{m\in \nat}$ in $\bD$ such that $S(B_{n})\simeq \colim_{m\in \nat} y(B_{n,m}')$.
  We now define a system  $(B_{n,m})_{(n,m)\in \nat\times \nat}$
    inductively by $n$ such that   $\colim_{m\in \nat}y(B_{n,m})\simeq S(B_{n})$.  Then by \cref{wrijotgwegrerfwerf9} the system $(B_{n,m})_{m\in \nat}$ is a weakly compact approximation of $B_{n}$ for every $n$ in $\nat$ and  we can deduce the assertion from \cref{erokjgperwgrewgwregw9}.

  Assume we have constructed the system $(B_{n,m})_{\{0,\dots,k-1\}\times \nat}$.
  Then we consider the map $S(B_{k-1})\to S(B_{k})$ in $\Ind^{\aleph_{1}}(\bD)$, using the usual $\lim-\colim$-formula for the mapping {anima} in the $\Ind$-category, as a point in the {anima}
  $$\lim_{m\in \nat^{\op}}\colim_{m'\in \nat} \map_{\bD}(B_{k-1,m},B'_{k,m'})\ .$$
   This point can be represented by a diagram 
   $$\xymatrix{\dots\ar[r]&B_{k-1,m-1}\ar[d]\ar[r]&B_{k-1,m}\ar[d]\ar[r]&B_{k-1,m+1}\ar[d]\ar[r]&\dots\\\dots\ar[r]&B_{k,m'(m-1)}'\ar[r]&B_{k,m'(m)}'\ar[r]&B_{k,m'(m+1)}'\ar[r]&\dots}$$
   for some increasing function $m':\nat\to \nat$.   We then set $B_{k,m}:=B_{k,m'(m)}'$.
This finishes the induction step.
\end{proof}

From now one we assume that $\bD$ is stable.
Recall the \cref{jigowergwerfewrfwerf} of    a weakly compact exhaustion. Let $X$
 be a   light profinite space.
\begin{prop}\label{ewijrgowefgerferfw} We assume that   $B$   admits a weakly compact approximation.   \begin{enumerate}
\item  We have a canonical  surjective map of sets
$$\overline{[B,C^{X}]} \to  \Hom_{\Top}(X,\overline{[B,C]})\ .$$ \item If $B$ admits a weakly compact exhaustion, then this map is a bijection. 
\end{enumerate}
\end{prop}
\begin{proof}
The  map is given by the  lower horizontal map in
$$\xymatrix{[B,C^{X}]\ar[r]^-{\ev}\ar[d] & \Hom_{\Top}(X,  [B,C] ) \ar[d] \\ \overline{[B,C^{X}]}\ar[r] & \Hom_{\Top}(X,\overline{[B,C]}) }$$
which exists since if $f$ in $ [B,C^{X}]$ is {strong} phantom, then $i_{x}^{*}f$ in $[B,C]$ is {strong} phantom for every $x$ in $X$. Here we use that by \cref{jierogwergwrefrewfwrf} the kernels of the left- and right vertical maps are given by
phantom maps or maps which have values in phantom maps, respectively.
 
 We now  show that this map is  surjective. 
 We fix a weakly compact approximation   $(B_{n})_{n\in \nat}$ for $B$.   We consider $f$ in $ \Hom_{\Top}(X,\overline{[B,C]})$.
 Let $\hat f:X\to   [B,C]$ be any set-theoretic lift. Then $X\xrightarrow{\hat f} [B,C]\to [B_{n},C]$
 coincides with $X\xrightarrow{f}  \overline{[B,C]}\to [B_{n},C]$ which is continuous. 
 Since $n$ is arbitrary we conclude by \cref{rqjifoergstgsgfgfg} that $\hat f$ is continuous.
 By \cref{jewrogfwerfrfwf} there exists $\tilde f$ in $[B,C^{X}]$  
 which represents $\hat f$ up to a map with values in strong phantom maps.
 This implies that the image of $\hat f$ in $ \overline{[B,C]}$ is the desired preimage of $f$.
 
 We now assume that   $(B_{n})_{n\in \nat}$ is a  weakly compact exhaustion for $B$ and show injectivity.
 Let $\bar f$ be  in $\overline{[B,C^{X}]}$ and assume that it goes to zero in $\Hom_{\Top}(X,\overline{[B,C]})$.
 Let $\tilde f$ in $ [B,C^{X}]$ be a preimage of $\bar f$.
 In order to check that $\bar f=0$ it suffices to check that the image of $\tilde f$ under $[B,C^{X}]\to [B_{n},C^{X}]$ vanishes for all $n$ in $\nat$.
 
 Let us fix $n$ in  $\nat$. Then   the image of $\tilde f$ in $[B_{n+1},C^{X}]$  is the image of a map $\tilde f'$ in $[B_{{n+1}},C^{Y}]$ under the canonical map
 $[B_{{n+1}},C^{Y}]\to  [B_{{n+1}},C^{X}]$, where $X\to Y$ is a quotient to a suitable finite set $Y$.
 
 Since $i_{x}^{*}\tilde f$ is a strong phantom map for every $x$ in $X$ we know that
 $i_{x}^{*}\tilde f=i_{y}^{*}\tilde f'$ is a strong phantom map, where $y$ is the image of $x$ in $Y$. 
 Since the structure map $B_{n}\to B_{n+1}$ is weakly compact
 we can now conclude that
 the image of $\tilde f'$ under $[B_{n+1},C^{Y}]\to [B_{n},C^{Y}]$ vanishes. 
 This implies that  the image of $\tilde f$ under $[B,C^{X}]\to [B_{n},C^{X}]$ vanishes.
    \end{proof}

Assume that all objects of $\bD$ have weakly compact approximations.
\begin{ddd} 
We can form the additive topologically enriched category $\bar \bD$ with the same objects as $\bD$, and with the topological morphism groups $$\Hom_{\bar \bD}(B,C):=\overline{[B,C]}$$ between objects $B,C$. 
\end{ddd}

The composition in $\bD$ is well-defined and continuous 
by \cref{wjeirgowergregwf9}. 
We have a functor
$$\bar \ho: \bD\to \ho\bD \to \bar \bD\ .$$

Recall from \cref{ergkoperwgwerfwefref} that $ \bD^{\shp}$ is the full subcategory   of objects  $C$ admitting shapes. By \cref{werojgpwertgwerrfwerfw}   it  is stable and closed under countable  colimits.
 Then \cref{wokpteghwergwerfwerg9} says:
\begin{kor} \label{kopgpertherhetrhgeget}
The restriction $\bar \ho_{|\bD^{\shp}}:\bD^{\shp}\to \bar \bD$  preserves countable  filtered colimits.
\end{kor}
 By \cref{ewijrgowefgerferfw} 
 we have:
 \begin{kor} For any two objects $B,C$ of $\bD$
 we have a canonical surjective map $$\Hom_{\bar \bD}(B,C^{X})\to \Hom_{\Top}(X,\Hom_{\bar \bD}(B,C))$$
 which is bijective if $B$  has a weakly compact exhaustion.
 \end{kor}

\begin{prop}\label{wekotgpwgrefwf} If $\bD$ is stable and compactly assembled,  
then the functor $\bar \ho:\bD\to \bar \bD$ is conservative. 
\end{prop}
\begin{proof}
Let $f:B\to C$ be a morphism in $\bD$ such that its image 
 $\bar f$ in $\Hom_{\bar \bD}(B,C)$ is invertible.
   Let $\bar g$   in $\Hom_{\bar \bD}(C,B)$ be an inverse of $\bar f$ represented by a morphism $g:C\to B$. Then $\phi:= \id_{B}- g\circ f $ is a phantom.
We can conclude by  \cref{weijogperfwrefregdghfg}
  that  $g\circ f=\id_{B} -\phi$ is an equivalence.
Similarly we show that $f\circ g$ is an equivalence.
This implies that $f$ is an equivalence.

\end{proof}

 \subsection{Analysis in $\AsGc$ and $  \EsG$}\label{wgjkopwetgerwfrefwrfwref}
 
 We have two interesting examples of  \pcas{} $\infty$-categories 
 $\AsGc$  (see \cref{vsiowerfsdfuiohsdfoijasdf1}) and $  \EsG$ (see \cref{vsiowerfsdfuiohsdfoijasdf}) to which we can specialize the 
 general results about the topological $\Hom$-spaces and groups from the preceding sections. We note that both  {categories} are pointed, left-exact and  admit countable filtered colimits which commute with  finite limits.
   Furthermore, $ \EsG$ is stable.
By \cref{ewrogjpwergwerfwerfwrfw} the homotopy category of $\AsGc$  can be identified with the homotopy category of asymptotic morphisms  $\AsG$ from \cite{Guentner_2000}, and (see e.g. \cref{jwegoeferfrefwerf}) the homotopy category of $\EsG$ is the classical equivariant $E$-theory category $\EsGn$. 

In the non-equivariant case  the $\Hom$-sets of  $\AsG$ have been equipped with topologies in  \cite{Carrion:2023aa}. In this section  we show that these topologies coincide with the ones induced by
the  \pcas{}  $\infty$-category $\AsGc$, and that many of the results from 
\cite{Carrion:2023aa} are special cases of the statements shown in the preceding sections. 
Similarly, in the case of $E$-theory we will show that the topology on the $\Hom$-groups of $\EsGn$ constructed in the non-equivariant case in \cite{Carrion:2023aa} coincides with the one 
induced by the  \pcas{}  $\infty$-category $ \EsG$, and that the statements about this topology 
from \cite{Carrion:2023aa} are {again} special cases of the general facts shown in  in the preceding sections.

   A new aspect   is that we can compare the analysis in $G\nCalg_{{\sepa}}$ with the analysis in {$\EsG$ or $\AsGc$} 
   via the functors $\esG:G\nCalg_{\sepa}\to \EsG$ (see \cref{wrejgioweferfw}) or  $\asGc :G\nCalg_{\sepa}\to \AsGc$ (see \eqref{543t445ts}).
   
 For $B$ in $G\nCalg_{\sepa}$ we  have two functors 
$(\Profinl)^{\op}\to  \EsG$
$$X\mapsto \esG (C(X,B))\ , \quad  X\mapsto \esG(B)^{X}\ ,$$
where the second uses the power functor   \eqref{qewfojpqdwqwededqwdqweded}. 
Similarly we have functors $(\Profinl)^{\op}\to  \AsGc$
$$X\mapsto  \asGc (C(X,B))\ , \quad  X\mapsto \asGc (B)^{X}\ .$$

 \begin{prop}\label{ojiviopwgfrwerferfwf}
We have  {natural equivalences} of functors \begin{eqnarray}\label{fkowerpkfperferfwre}\esG (B)^{-}&\stackrel{\simeq}{\to} & \esG (C(-,B)) :(\Profinl)^{\op} \to  \EsG \\
\asGc  (B)^{-}&\stackrel{\simeq}{\to}&  \asGc  (C(-,B)) :(\Profinl)^{\op} \to  \AsGc \nonumber  \ . \end{eqnarray} 

\end{prop}
\begin{proof}
We consider the case of $\esG$. The case of $\asGc $ is analoguous.
We first  construct the  natural transformation in \eqref{fkowerpkfperferfwre}.
 Since the domain 
  is defined by a left Kan extension 
along $\Fin^{\op}\to (\Profinl)^{\op}$, it suffices to construct a natural transformation of the restriction
\begin{equation}\label{fweq9if09i09qwedewdqwedwdqwed} \esG (B)^{-}\to \esG (C(-,B)) :\Fin^{\op}\to\EsG \ .
\end{equation} 
Since both functors in \eqref{fweq9if09i09qwedewdqwedwdqwed} send finite disjoint unions to products  (for the right-hand sides this follows from the exactness of $\esG$, and in the case of $\asGc $ we employ Schochet exactness)
they are both left  Kan-extensions of their restriction to $ *$.  
The canonical equivalence 
$$ \esG (C(*,B))\simeq  \esG (B)^{*}$$
thus fixes an equivalence \eqref{fweq9if09i09qwedewdqwedwdqwed}.
If $X\cong \lim_{i\in I}X_{i}$ is a presentation of $X$ as a countable cofiltered limit of finite sets, then
$C(X,B)\cong \colim_{i\in I^{\op} } C(X_{i},B)$. Since
$\esG $ preserves countable filtered colimits by \cref{cgrgregeffweeeee} (or \cref{uqwighfuihvsdiohjfhjkbvdxc} for $\asGc $),
we see that \eqref{fkowerpkfperferfwre} is an equivalence, too.
 \end{proof}

 \begin{prop}\label{wkjoergopwerfwerfvfsvfdvsdvsdfv} For every separable  compact Hausdorff space $X$ and $A,B$ in $G\nCalg_{\sepa}$ we
have  {canonical maps} \begin{eqnarray*}[ \esG (A), \esG (C(X,B))] &\to& \Hom_{\Top}(X,[ \esG (A), \esG (B)]^{\sh}) \\{}[ \asGc  (A), \asGc (C(X,B))] &\to& \Hom_{\Top}(X,[ \asGc  (A), \asGc  (B)]^{\sh})\ .\end{eqnarray*}
If $X$ is a light profinite  space, then {their   images  consist} precisely of the
representable maps.
\end{prop}
\begin{proof} We again consider the case of $\esG$. The case of $\asGc $ is analoguous.
For the first assertion we must show that the map of  sets 
$$[ \esG (A), \esG (C(X,B))] \to \Hom_{\Set}(X,[ \esG (A), \esG (B)])$$
sending $f$ in $[ \esG (A), \esG (C(X,B))]$ to the function
$x\mapsto i_{x}^{*}f$ takes values in continuous maps.  Continuity of a map out of a separable  compact  Hausdorff space
can be tested with sequences. 

Let $f$ be an element in $[ \esG (A), \esG (C(X,B))]$.
We consider a continuous map
  $t:\nat_{+}\to X$.
Then we have a square  in $\ho \EsG $.
$$\xymatrix{  [\esG (A) ,\esG (C(X,B))] \ar[r]\ar[d] &\Hom_{\Set}(X,[ \esG (A), \esG ]) \ar[d] \\ [ \esG (A), \esG (C(\nat_{+},B))] \ar[r] &\Hom_{\Set}(\nat_{+},[ \esG (A), \esG ])  }  \ . $$
It suffices to show that the right-down composition sends $f$ to a continuous map. But this is the case in view of \cref{ojiviopwgfrwerferfwf} since this map is representable by $t^{*}f$ in $  [ \esG (A), \esG (C(\nat_{+},B))] $.

The second assertion is a consequence of \cref{ojiviopwgfrwerferfwf}.
\end{proof}
Recall that  from \cref{okprthrtertegtrgrg} that $G\nCalg_{\sepa}$ is topologically enriched. 
The following consequence of \cref{wkjoergopwerfwerfvfsvfdvsdvsdfv} is a version of \cite[Prop. 2.4]{Carrion:2023aa}.
\begin{kor}\label{gjweriofwefgdbgb}
For all $A,B$ in $G\nCalg_{\sepa}$ we have   canonical continuous 
 maps \begin{eqnarray*}
\underline{\Hom}_{G\nCalg_{\sepa}}(A,B)&\to& [\esG(A) ,\esG (B)]^{\sh} \ .\\
\underline{\Hom}_{G\nCalg_{\sepa}}(A,B)&\to& [\asGc (A) ,\asGc  (B)]^{\sh} 
\end{eqnarray*}
 \end{kor}
By this corollary the functors $\EsGn$ and $\asG$ refine to functors between topologically enriched categories.

The following result is only for $E$-theory since we employ \cref{ewijrgowefgerferfw} in the proof which has only been shown  under a stability assumption.
\begin{kor}
For every light profinite space $X$ we have a surjective map
$$\overline{[ \esG (A), \esG (C(X,B))]}  \to  \Hom_{\Top}(X,\overline{[ \esG (A), \esG (B)]})\ .$$
\end{kor}
\begin{proof}
This follows immediately from \cref{ojiviopwgfrwerferfwf} and \cref{ewijrgowefgerferfw}.
\end{proof}

\begin{ex}\label{jireogwegrewgw9} {In this example we explain an application of the theory developed above to approximate unitary equivalence. Recall, that}
if $u$ is  in $U(M(B))^{G}$ ($u$ is a $G$-invariant unitary multiplier on $B$) and $f:A\to B$ is a homomorphism, then
$\esG (f)\simeq  \esG (u^{*}f u)$. 

Assume now  that $(u_{n})_{n\in \nat}$ is a sequence in $U(M(B))^{G}$ such that
$\lim_ {n\in \nat} u_{n}^{*}f u_{n}=g$ in $\underline{\Hom}_{G\nCalg_{\sepa}}(A,B)$. Then $f$ and $g$ are called approximately
unitary equivalent. 
 
\begin{kor}\label{wtokpghgwrerefsdfg}
If $f$ and $g$ are approximately unitary equivalent, then
$\esG (f)-\esG (g)$ is a strong phantom morphism.
\end{kor}
\begin{proof}
We note that $ (\esG (u_{n}fu_{n}^{*}))_{n\in \nat}$ is the constant sequence and that by \cref{gjweriofwefgdbgb} we have an equality
$\lim_{ n\in \nat} \esG (u_{n}fu_{n}^{*})=\esG (g)$.
This implies the result in view of \cref{jierogwergwrefrewfwrf}.
\end{proof}
The following is a consequence of \cref{wekotgpwgrefwf} and \cref{wtokpghgwrerefsdfg}.
\begin{kor}
If $f:A\to B$ {and}  $g:B\to A$  {in $G\nCalg_{\sepa}$ are} such that
$f\circ g$ and $ g\circ f$ are approximately unitary equivalent to isomorphisms, then
$\esG (f)$ is an equivalence.
\end{kor}
{
As a consequence of \cref{weokrgpwregrefrfw9} we get:
\begin{kor}
If $f:A\to B$ and $g:B\to C$ are homomorphisms  $G\nCalg_{\sepa}$ which are approximately unitarily equivalent to zero, then $\esG(g\circ f)=0$.
\end{kor}
It is clear that $g\circ f$ is again approximately unitarily equivalent to zero so that $\esG(g\circ f)$ is a phantom map by \cref{wtokpghgwrerefsdfg}. But the fact that this class vanishes  is less trivial.
}
\hB
\end{ex}

\begin{ex} \label{ogjwerferwferwfwe}{In this example, we apply the general theory
from above to quasi-diagonal extensions.  
An} exact sequence $$0\to A\to B\to C\to 0$$
in $G\nCalg_{\sepa}$ will be called weakly quasi-diagonal \cite{MR1784676} if there exists a family of invariant projections $(p_{n})_{n\in \nat}$
in $M(B)$ satisfying:
 \begin{enumerate}  \item For all $n$ in $\nat$ we have $p_{n}B\subseteq A$.
 \item For every $b$ in $B$ we have $\lim_{n\to \infty} (p_{n}b-bp_{n}) =0$.
  \item For every $a$ in $A$ we have  $\lim_{n\to \infty}  p_{n}a=a$.
\end{enumerate}
In  \cite{MR1784676} it was in addition required that the family $(p_{n})_{n\in \nat}$ is increasing.
If $p_{n}$ is in $A$ for all $n$ in $\nat$, then  the exact sequence is called quasi-diagonal.
\begin{prop}\label{werrwer34rwffefwefwerf}
The boundary map $\delta:\esG(C)\to \Sigma \esG(A)$ associated to a weakly quasi-diagonal extension is a strong phantom map
\end{prop}
\begin{proof}
We consider $p:=(p_{n})_{n\in \nat}$ as a multiplier of $\prod_{n\in \nat} B$.
We let $\tilde B$ be the subalgebra of $\prod_{n\in \nat} B$ generated by
$\bigoplus_{n\in \nat} B$, $\diag(B)$ and the elements $p\diag(b)$ for all $b$ in $B$.
The projection $q: \prod_{n\in \nat} B\to \prod_{n\in \nat} C$ sends
the generators of $\tilde B$ to elements of the subalgebra $\bigoplus_{n\in \nat} C+\diag(C)$.
We therefore get a map of exact sequences $$\xymatrix{0\ar[r]&A\ar[r]\ar[d]&B\ar[r]\ar[d]&C\ar@{=}[d]\ar[r]&0\\0\ar[r]&\tilde A/ \bigoplus_{n\in \nat}A\ar[r]
 &\tilde B/\bigoplus_{n\in \nat}B\ar[r]&C\ar[r]&0}\ ,$$
where $\tilde A:=\ker(q)\cap \tilde B$ and the two left vertical maps are induced by the diagonal.
The multiplier $p$ induces an invariant central  projection   in $M(\tilde B/\bigoplus_{n\in \nat}B)$
which splits the lower sequence. Consequently, its boundary map $\bar \delta:\esG(C)\to \Sigma \esG(\tilde A/ \bigoplus_{n\in \nat}A)$ vanishes. Now $\bar \delta$ is equivalent to the composition
$$ \esG(C) \xrightarrow{\delta}  \Sigma \esG(A)\xrightarrow{\Sigma \esG(\diag)}\Sigma \esG(\tilde A)\to  
\colim_{n\in \nat}\Sigma \esG(\tilde A/ \bigoplus_{k=1}^{n}A)$$
  Let $g:\esG(D)\to \esG(C)$ be a compact morphism in $\EsG$. 
Since
$$ \esG(D)\xrightarrow{g}  \esG(C)\xrightarrow{\delta}  \Sigma \esG(A)\xrightarrow{\Sigma\esG(\diag)}\Sigma \esG(\tilde A)\to  
\colim_{n\in \nat} \Sigma\esG(\tilde A/ \bigoplus_{k=1}^{n}A)$$
vanishes, by compactness of $g$ and since $\esG$ is exact and preserves sums, 
 there exists $n$ in $\nat$ such that
$$ \esG(D)\to \esG(C)\xrightarrow{\delta}  \Sigma \esG(A)\xrightarrow{\Sigma\esG(\diag)} \Sigma\esG(\tilde A)\to  
 \Sigma \esG(\tilde A/ \bigoplus_{k=1}^{n}A)$$
vanishes. Recall that $\tilde A$ is a subalgebra of $\prod_{n\in \nat} A\cong \ker(q)$ so that we can consider the projection
$\pr_{n+1}:\tilde A\to A$ onto the $(n+1)$th component. It induces a projection $\bar \pr_{n+1}:  \tilde A/ \bigoplus_{k=1}^{n}A\to A$.
The composition $$A\xrightarrow{\diag}\tilde A\to \tilde A/ \bigoplus_{k=1}^{n}A\xrightarrow{\bar \pr_{n+1}}A$$
is the identity of $A$. This implies that  
$ \esG(D)\xrightarrow{g}  \esG(C)\xrightarrow{\delta} \Sigma \esG(A)$ vanishes. 
Since $g$ was an arbitrary compact morphism to $\esG(C)$ we conclude that $\delta$ is a phantom map.
\end{proof}

In the case of a trivial group $G$,
 in  \cite{Salinas_1974}, \cite{zbMATH02232820},  \cite{MR1784676} the set of  classes in 
 $$\KK_{1}(A,B):=[ \kks(C),\Sigma \kks(A)]^{\KKs}$$ 
 classifying quasi-diagonal extensions has been identified with the   closure of zero in a natural topology on the $\KK$-groups.
 If $C$ is in the bootstrap class, then the comparison map $c$ from \eqref{herjthiertogjertioger} induces an isomorphism $$[ \kks(C),\Sigma \kks(A)]^{\KKs}\to [\es(C),\Sigma \es(A)]^{\Es} $$
 identifying these closures  and \cref{werrwer34rwffefwefwerf} recovers these classical results in view of 
 \cref{jierogwergwrefrewfwrf}. 
 Note that the papers listed above also study the converse question: Does the fact that  the extension class is a strong phantom imply that the extension is quasi-diagonal?
 \hB
\end{ex}

\begin{rem}\label{wtkoprhterherhgertg}
In this remark we explain how parts of the results of 
 \cite{Carrion:2023aa} are direct consequences of the fact that $\AsG$ or $\EsGn$ are the homotopy categories of the \pcas{}   $\infty$-categories 
$\AsGc$ or  $\EsG $ and the general statements from the preceding sections.
Everywhere we use that these categories are known to be  \pcas{} by \cref{jifofqweewf9} which implies the existence of compact approximations.

    \setlist[enumerate]{itemsep=0cm}  \begin{enumerate} 
\item The shape topology on   $ [ \esG (B), \esG (C)]^{\sh}$ or  $[\asGc (A),\asGc (B)]^{\sh}$ coincides with the topology on the $E$-theory groups or $\Hom$-sets in $\AsG$
introduced   
   \cite{Carrion:2023aa}. This follows from the characterization of the latter topology
   in  \cite[Prop. 2.6]{Carrion:2023aa} together with   \cite[Prop. 2.8]{Carrion:2023aa} which is the same as the characterization of the shape topology
   in \cref{rqjifoergstgsgfgfg}. Here we  use the fact that  a shape system for $B$ in $G\nCalg_{\sepa}$ 
   provides  a  weakly compact approximation 
    of $\asGc (B)$ in $\AsGc$,   or $ \esG (B)$ in $\EsG $ (assuming that $B$ is twice suspended and $K_{G}$-stable), see  \cref{jigowggergwe9}.
   \item  {The topologies on    $ [ \esG (B), \esG (C)]$ or  $[\asGc (A),\asGc (B)]$  introduced in  \cite{Carrion:2023aa} a priori depend on the choice of the separable $G$-$C^{*}$-algebras $A$ and $B$. 
  It only follows a posteriori  from the continuity of the composition  
    \cite[Thm. 2.15]{Carrion:2023aa}  that they  only depend on the classes of $A$ and $B$ in $\AsG$ or $\EsGn$.
In contrast, the  shape topology   is by definition an  invariant of the classes of $A$ and $B$ in $\AsGc$ or $\EsG$.}
      \item  The topological spaces $ [ \esG (B), \esG (C)]^{\sh}$ or  $[\asGc (A),\asGc (B)]^{\sh}$ are   first countable by \cref{wqefu0dewdewdqwdqewd}. This corresponds to     \cite[Thm. 2.10]{Carrion:2023aa}.   
  \item A converging sequence in $ [ \esG (B), \esG (C)]^{\sh}$ or $ [ \asGc  (B), \asGc  (C)]^{\sh}$ is 
by definition   a continuous map  $\nat_{+}\to [ \esG (B), \esG(C)) ]^{\sh}$ or $\nat_{+}\to   [ \asGc  (B), \asGc  (C)]^{\sh}$, respectively.  By \cref{qreokgprefqewfqewfqd}  and \cref{ojiviopwgfrwerferfwf} such a sequence can be represented {by}    elements of $[ \esG (B), \esG (C(\nat_{+},C))] $ or  $[ \asGc  (B),\asGc   (C(\nat_{+},C))]$, respectively.  The converse is clear: any element in $[ \esG (B), \esG (C(\nat_{+},C))] $ or $[ \asGc  (B),\asGc   (C(\nat_{+},C))]$ gives rise to a converging sequence by \cref{ojiviopwgfrwerferfwf}, \cref{oijrfopqrfewdqwfqfefseq} and the equality of the shape and light condensed topologies \cref{rejiogerfwerfwrefrwef}.
 There facts are   \cite[Thm. A]{Carrion:2023aa} and \cite[Thm 2.12]{Carrion:2023aa}, where they are 
 called Pimsner's Condition.
  \item By \cref{worpjegperfrfwefewrf} the composition in $\EsGn$ or $\AsG$  
  is jointly continuous for the shape topology.  This corresponds to  \cite[Thm 4.5]{Carrion:2023aa} and \cite[Thm 2.15]{Carrion:2023aa} .    \item    \cref{jierogwergwrefrewfwrf}  and \cref{wokpteghwergwerfwerg9}  describes the Hausdorffifications of the groups  $[ \esG (B), \esG (C)]^{\sh}$ and the spaces $[ \asGc  (B),\asGc   ( C)]^{\sh}$.   
     In this way we reproduce  \cite[Thm. 3.10]{Carrion:2023aa} and the corresponding  parts of   \cite[Thm. 4.6]{Carrion:2023aa}.
  \item By \cref{okqrepgqdedqwe} for any  compact map $\asGc (A) \to \asGc (B)$ the composition
$$\underline{\Hom}_{G\nCalg}(B,C)\to [\asGc(B) ,\asGc (C)]^{\sh}\to  [\asGc (A) ,\asGc(C)]^{\sh}$$
factorizes over a discrete space. By \cref{gjoiowepgrfwefrefw} this is in particular the  case   if the map  $\asGc (A) \to \asGc (B)$ comes from a   semi-projective map $A\to B$. This corresponds to  \cite[Prop. 2.8]{Carrion:2023aa}.
\item\label{rogoiwepgreferwfre}  By \cref{wjeirgowergregwf9}
  the composition maps
$$\overline{[\asGc  (B) ,\asGc  (C) ]}\times \overline{[\asGc  (A) ,\asGc  (B) ]}\to \overline{[\asGc  (A) ,\asGc (C) ]}$$ and
$$\overline{[\esG (B) ,\esG (C) ]}\times \overline{[\esG (A) ,\esG (B) ]}\to \overline{[\esG (A) ,\esG (C) ]}$$
are continuous. This corresponds to 
 \cite[Prop. 3.2]{Carrion:2023aa}  and the part of \cite[Thm. 4.6]{Carrion:2023aa}.  
 \item  By \cref{thkjptherhrtgergeg}  combined with \cref{gjoiowepgrfwefrefw} we have isomorphisms of topological groups or spaces
 \begin{eqnarray*}  \overline{[\esG (B),\asGc  (C)]}&\cong& \lim_{n\in \nat} [\esG (
A_{n}),\esG(C)]^{\disc}\\   \overline{[\asGc  (B),\asGc  (C)]}&\cong& \lim_{n\in \nat} [\asGc  (
A_{n}),\asGc  (C)]^{\disc}\ , \end{eqnarray*}
where $(A_{n})_{n\in \nat}$ is a  shape system for $B $ and we assume that $B$ 
is twice suspended and $K_{G}$-stable in the case of $\esG$.
This corresponds to 
  \cite[Thm. 3.7]{Carrion:2023aa} except for the countability assertions
  which is specific to the $C^{*}$-algebra situation and not a consequence of the general  $\infty$-categorical considerations. 
 
  \item We can form a topologically enriched additive category $\overline{ \EsG}$ (denoted by $EL$ in  \cite{Carrion:2023aa}) and the topologically enriched category $\overline{\AsG}$ with  the same objects as $G\nCalg_{\sepa}$
  and with topological  morphism groups 
  $\overline{[\esG (B) ,\esG (C) ]}$ or spaces  $\overline{[\asGc  (B) ,\asGc   (C) ]}$.
  The composition is well-defined by \cref{wtkoprhterherhgertg}.\ref{rogoiwepgreferwfre}.
     By  \cref{kopgpertherhetrhgeget}  the functors $$\overline{\esG}:G\nCalg_{\sepa}\to\overline{\EsG}\ , \qquad  
 \overline{\asGc }:G\nCalg_{\sepa}\to\overline{\AsG}$$
  preserves filtered colimits.
     These are   \cite[Thm. B]{Carrion:2023aa}  and  \cite[Thm. 3.10]{Carrion:2023aa}.
  \item By  \cref{wekotgpwgrefwf} the functor  $\EE^{G} \to\overline{ \EE^{G}}$  is conservative.
This 
   recovers  \cite[Prop. 4.7]{Carrion:2023aa}.
  Again there is no analogue for $\AsG\to \overline{\AsG}$ since \cref{wekotgpwgrefwf} requires stability.
  \end{enumerate}
  \end{rem}

 \bibliographystyle{alpha}
\bibliography{forschung2021}

\end{document}